\documentclass[12pt]{amsart}

\usepackage[utf8]{inputenc}
\usepackage[english]{babel}
\usepackage[a4paper,twoside,top=1.2in, bottom=1.0in, left=0.8in, right=0.8in]{geometry}

\usepackage{graphicx}
\usepackage{caption} 

\usepackage{amsmath}
\usepackage{amsthm}
\usepackage{amssymb}
\usepackage{esint} 
\usepackage{stmaryrd} 
\usepackage{enumitem} 
\usepackage{mathrsfs} 
\usepackage{faktor} 
\usepackage{dsfont} 
\usepackage{xcolor}
\usepackage{array}
\usepackage{hhline}
\usepackage{comment, verbatim} 
\usepackage{imakeidx} 
\usepackage{xparse} 
\usepackage{mathtools}
\usepackage{fancyhdr} 
\usepackage{ifthen} 
\usepackage{forloop} 
\usepackage{xstring}
\usepackage{emptypage} 
\usepackage{minibox}
\usepackage{galois}
\usepackage[all]{xy}
\usepackage{color}

\definecolor{dkgreen}{rgb}{0,0.6,0}
\definecolor{gray}{rgb}{0.5,0.5,0.5}
\definecolor{mauve}{rgb}{0.58,0,0.82}
\usepackage{listings}
\usepackage[initials]{amsrefs} 

\usepackage{tikz}
\usetikzlibrary{arrows,shapes,patterns,calc,fadings,decorations.pathreplacing,cd}

\usepackage{hyperref} 
\hypersetup{
	colorlinks = true,
	linkcolor = {blue},
	urlcolor = {red},
	citecolor = {blue},
}

\newcounter{results}[section] 

\newcounter{steps}[section] 

\theoremstyle{plain}
\newtheorem{theorem}[results]{Theorem}

\newtheorem{lemma}[results]{Lemma}
\newtheorem{proposition}[results]{Proposition}
\newtheorem{corollary}[results]{Corollary}

\newtheorem*{theorem*}{Theorem}
\newtheorem*{lemma*}{Lemma}
\newtheorem*{proposition*}{Proposition}
\newtheorem*{corollary*}{Corollary}
\newtheorem*{exercise*}{Exercise}
\newtheorem*{fact*}{Fact}

\theoremstyle{remark}
\newtheorem{remark}[results]{Remark}

\theoremstyle{definition}
\newtheorem{definition}[results]{Definition}
\newtheorem{example}[results]{Example}

\newtheorem*{definition*}{Definition}
\newtheorem*{example*}{Example}

\theoremstyle{plain}
\newtheorem{maintheorem}{Theorem}

\newtheorem{maincorollary}[maintheorem]{Corollary}

\usepackage{accents}
\newlength{\dhatheight}

\newcommand{\Z}{\ensuremath{\mathbb Z}}
\newcommand{\R}{\ensuremath{\mathbb R}}
\newcommand{\C}{\ensuremath{\mathbb C}}

\newcommand{\st}{\ensuremath{\ :\ }} 
\newcommand{\id}{\ensuremath{\mathds{1}}}

\DeclareMathOperator{\tr}{tr}

\DeclareMathOperator{\sym}{Sym}

\DeclareMathOperator{\II}{I\!I} 
\DeclareMathOperator{\dist}{dist} 
\DeclareMathOperator{\inj}{inj} 

\DeclareMathOperator{\ind}{index} 
\DeclareMathOperator{\nul}{nullity} 
\DeclareMathOperator{\Ogroup}{O}
\DeclareMathOperator{\SOgroup}{SO}
\DeclareMathOperator{\Stab}{Stab} 
\DeclareMathOperator{\Iso}{Isom} 

\newcommand{\Diff}{\mathrm{Diff}}
\newcommand{\met}{\mathrm{Met}} 

\newcommand{\be}{\begin{equation}}
\newcommand{\ee}{\end{equation}}

\newcommand{\al}{\alpha}

\newcommand{\nocontentsline}[3]{}
\newcommand\stoptoc{%
   \let\origcontentsline\addcontentsline
   \let\addcontentsline\nocontentsline
}
\newcommand\resumetoc{%
   \let\addcontentsline\origcontentsline
}

\newcommand{\BB}{\mathbb{B}}

\newcommand{\RR}{\mathbb{R}}
\newcommand{\SSp}{\mathbb{S}}

\newcommand{\cH}{\mathcal{H}}

\newcommand{\cJ}{\mathcal{J}}

\newcommand{\cM}{\mathcal{M}}

\newcommand{\cT}{\mathcal{T}}

\newcommand{\cX}{\mathcal{X}}

\newcommand{\mbfA}{\mathbf{A}}
\newcommand{\mbfB}{\mathbf{B}}
\newcommand{\mbfC}{\mathbf{C}}

\newcommand{\mbfP}{\mathbf{P}}

\newcommand{\mbfV}{\mathbf{V}}

\newcommand{\mbfx}{\mathbf{x}}

\newcommand{\mbfy}{\mathbf{y}}

\newcommand{\rmI}{\mathrm{I}}
\newcommand{\rmJ}{\mathrm{J}}

\newcommand{\mfr}{\mathfrak{r}}

\newcommand{\scC}{\mathscr{C}}

\newcommand{\scH}{\mathscr{H}}

\newcommand{\scR}{\mathscr{R}}

\newcommand{\scT}{\mathscr{T}}

\newcommand{\orig}{\mathbf{0}}

\newcommand{\Hom}{\operatorname{Hom}}

\newcommand{\spt}{\operatorname{spt}}

\newcommand{\injrad}{\operatorname{injrad}}
\newcommand{\Reg}{\operatorname{Reg}}
\newcommand{\Sing}{\operatorname{Sing}}

\newcommand{\Exp}{\operatorname{Exp}}

\newcommand{\MSI}{\textbf{MSI}}
\newcommand{\conetimes}{\times\!\!\!\!\!\times}
\newcommand{\eucl}{\mathrm{euc}}

\newcommand{\Tr}{\operatorname{tr}}

     \title[Observing MSI on open sets of metrics] 
     {Observing minimal submanifolds with isolated singularities on open sets of metrics}
     \author{Alessandro Carlotto, Yangyang Li, Zhihan Wang}
     \address{
     \newline \indent Alessandro Carlotto: 
     \newline Universit\`a di Trento, Dipartimento di Matematica,
via Sommarive 14, 38123 Trento, Italy
\newline\textit{E-mail address: alessandro.carlotto@unitn.it}
 \newline \newline \indent Yangyang Li: 
\newline Department of Mathematics, University of Notre Dame, 218 Hayes-Healy Center, Notre Dame, IN, 46556, United States of America
       \newline \textit{E-mail address: 
       yangyangli@nd.edu} 
     	   	 \newline \newline \indent Zhihan Wang: 
\newline Department of Mathematics, Cornell University, 310 Malott Hall,
Ithaca, NY, 14853, United States of America
       \newline \textit{E-mail address: zw782@cornell.edu} 
       }

\begin{document}

\begin{abstract} We develop a perturbative theory for minimal submanifolds with only (strongly) isolated singularities, with respect to varying background Riemannian metrics. 
As an application, we obtain a general ``sporadic persistence" criterion, thus naturally complementing the non-persistence theorem by the authors. In particular, such a criterion is satisfied by the spherical suspension in $\mathbb{S}^{3m+2}$ of the Veronese embedding of $\mathbb{K}\mathbb{P}^2$ for $\mathbb{K}=\R,\C,\mathbb{H},\mathbb{O}$ (the normed division algebras) in the round spheres of dimension $3m+1=4,7,13,25$ respectively, thereby providing families of continuously varying $(2m+1)$-dimensional stationary varifolds with isolated singularities that are observable on an open set of metrics. By our previous results, these phenomena cannot possibly occur in the codimension one scenario.
\end{abstract}

\maketitle
    \tableofcontents

\section{Introduction}\label{sec:Intro}

The purpose of the present manuscript is to contribute to the study of minimal submanifolds with isolated singularities (that is, in the language of \cite{CLW24}, of $\MSI$), obtaining some results that complement (essentially: in the \emph{opposite direction}) those contained in our previous work on the subject. More specifically, motivated by a question posed by Neves in relation to Hsiang's solution of the spherical Bernstein conjecture \cite{Hsi83}, we proved in \cite{CLW24} that non-smooth $\MSI$ (albeit easily constructible in space forms, e.\,g. by spherical suspensions of closed, minimal submanifolds) are \emph{generically not observable} (namely: they do not exist for a Baire-generic choice of the background metric) 
except for some \emph{potentially persistent} ones that are singled out by the property of saturating certain inequalities pertaining to the geometry of the links at their singularities. Our scope here is to provide a general, verifiable criterion for actual persistence of $\MSI$ on open sets of metrics (see Theorem \ref{thm:main} for a precise statement) and then to present instances where such a criterion is actually satisfied (see Corollary \ref{cor:main}).  We recall that the aforementioned inequalities in fact \emph{never} hold in the codimension one case, which rules out any sort of ``cheap" examples, and that they rely on the actual computation of the Morse index of (smooth, closed) minimal submanifolds in round spheres, that is a notoriously challenging task in its own right.  

The previous existence results are obtained by developing a perturbation theory for $\MSI$. Since -- among other reasons -- most examples of $\MSI$ are in space forms, it is immediately clear that one is naturally confronted with the symmetries of the ambient manifold, which results in obstructions to a direct application of the implicit function theorem in suitable (albeit a priori not obviously determined) functional spaces. What we do is rather, in general terms, a Lyapunov-Schmidt reduction ``in the large" of the problem with respect to nonlinear constraints. As a byproduct, we also obtain multiplicity results of Morse-theoretic character, which we work out explicitly in the special case of the (codimension two) spherical suspension of the Veronese embedding of the real projective plane $\R\mathbb{P}^2$ in the round four-dimensional sphere.
That example, which is perhaps the simplest possible instance of application of our theory, already displays a number of interesting features. 

We wish to stress that, differently e.\,g. from the perspective considered by Smale in \cite{Sma93Crelle}, here we do \emph{not} fix the singular structure (namely: the singular points and the corresponding tangent cones) but rather avoid any sort of a priori geometric restriction, so to make the obstruction space as small as possible and, in certain cases, to get rid of obstructions at all. In fact, when working with varying Riemannian metrics (as opposed to varying boundaries in Euclidean space) the idea of fixing the singular structure is doomed to fail as far as one wishes to obtain perturbative existence results, for reasons we are about to explain. In general terms, the approach we present is inspired (at least for the ``local" aspects of the theory, see Section \ref{sec:AnMeanCurv}) by White's influential work \cite{Whi91} for smooth, closed minimal submanifolds. However, although his own functional-analytic setup for the problem is very general and widely applicable, the singular case  poses a number of substantial and, in some cases, foundational challenges, which we will briefly describe in the sequel of the present introduction (see Subsection \ref{subs:Technical challenges}), after presenting our setup and the main theorems we can obtain.  
\subsection{Statements of the main results.}\label{subs:Statements}
Let $M$ denote a compact, boundaryless differentiable manifold (i.\,e. a manifold of class $C^\infty$) of dimension $N\geq 3$.
Given a smooth Riemannian metric $\bar g$ on $M$ and a minimal submanifold with only strongly isolated singularities in $(M,\bar g)$, of dimension $2\leq n<N$  -- see Definition \ref{def:MSI} of $\MSI$ and Section \ref{sec:Prelim} for context -- we introduce the following notation. We let:
\begin{itemize}
\item
$G:=\Iso(M,\bar g)$ be the group of Riemannian
isometries of the ambient manifold. By the
Myers--Steenrod theorem and compactness of $M$,
this is a finite-dimensional compact Lie group;

\item
$H:=\Stab_G(\Sigma)$ be the subgroup of $G$
consisting of ambient isometries that preserve
$\Sigma$ as a set. This is a closed subgroup
of $G$, hence an embedded Lie subgroup by
Cartan's closed subgroup theorem;

\item
$G/H$ be the set of left cosets, endowed with
its canonical quotient smooth structure.
It is a compact, boundaryless smooth manifold
and a homogeneous space under the natural
left action of $G$%
\footnote{The subgroup $H$ need not be normal
in $G$, so $G/H$ need not be a quotient Lie
group; an equatorial hypersphere in a round
sphere already provides an example.
Nevertheless, since $G$ is compact, $G/H$
always admits a $G$-invariant Riemannian
metric, obtained by descending a bi-invariant
Riemannian metric on $G$.}.
\end{itemize}

We may think of $G/H$ as the 
space of congruent copies of $\Sigma$ in $(M,\bar g)$; note that $G$ may (and, by virtue of a classical result of Ebin, generically will) consist of just one element, or be a finite set. 
That being said, we let:
\vspace{-1mm}
\begin{multline*}
LSN(G/H):= \text{Lusternik-Schnirelmann number of} \ G/H \\
(\text{least number of critical points of a smooth function} \ f:G/H\to \R )
\end{multline*}
\vspace{-1mm}
and
\vspace{-1mm}
\begin{multline*}
MN(G/H):= \text{Morse number of} \ G/H \\
(\text{least number of critical points of a smooth Morse function} \ f:G/H\to \R).
\end{multline*}
We recall the following well-known facts in differential topology (for which we refer the reader to the standard monograph \cite{CLOT03}): if $f$ is a Morse function on  $G/H$, then the number of critical points of
index $i$ is greater than or equal to the rank of the $i$-th cohomology space \emph{(over any coefficient field)}, and so the Morse number of  $G/H$ (and of any compact manifold alike) is bounded below 
by the sum of the ranks of such cohomology spaces;
on the other hand, the Lusternik-Schnirelmann number of $G/H$ \emph{minus one} is bounded below
by the Lusternik-Schnirelmann category of $G/H$ (and, in turn, the latter is bounded below by
the cup-length of $G/H$, which is the maximum number of elements
in the cohomology ring, all having positive degree, such that their cup product is nonzero). These topological invariants will be studied and computed in Appendix \ref{sec:topology} for the special case treated in Corollary \ref{cor:main}, for $\mathbb{K}=\R$. 

\

With this terminology and notation in place, our main result (in a slightly informal fashion) reads as follows:

\begin{maintheorem}\label{thm:main}
Let $(M,\bar g)$ be a compact, boundaryless Riemannian manifold, and let $\Sigma$ be an $\MSI$ therein satisfying the following two assumptions:
\begin{enumerate}[label=(\roman*)]
\item\label{MainThm:hp1}{the effective index, and the effective nullity at the links of all singular points both vanish;}
\item\label{MainThm:hp2}{all Jacobi fields are integrable, in the sense
that each is the normal component along $\Sigma$
of a Killing vector field on $(M,\bar g)$.}
\end{enumerate}
Then: for any metric $g$ that is sufficiently close to $\bar g$ (in the smooth topology) the manifold $(M,g)$ contains at least $LSN(G/H)$ minimal submanifolds with only strongly isolated singularities ($\MSI$) modeled on $\Sigma$, and generically at least $MN(G/H)$ such submanifolds.
\end{maintheorem}

We refer the reader to Theorem \ref{thm:mainPrecise} for a fully accurate version of this statement, relying on the setup we develop in the core of the paper.

\begin{remark}\label{rem:ClarifyMainTheorems} Some further clarifications are appropriate:
\begin{itemize}
 \item in the context described above, we will say that $\Sigma'$ is modeled on $\Sigma$ if -- loosely speaking -- $\Sigma'$ is the graph of a superlinearly decaying normal section over a copy of $\Sigma$ obtained by, possibly, slightly moving its singular points and acting on its tangent cones by rotations (the precise definitions are given in Section \ref{sec:Diffeo});
 \item differently from the case of smooth, closed minimal submanifolds it is important here to specify the functional space where Jacobi fields are considered (e.\,g. to properly interpret hypothesis (\ref{MainThm:hp2}) above). Basically, for fixed $\Sigma$ (and corresponding normal bundle $\mbfV$ etc\ldots) we consider a doubly augmented functional space $\hat W^{k,2}_{\tau}(\Sigma;\mbfV)$ as presented in Definition \ref{def:BiAugmentedSpaces}. The definition, albeit a bit technical at first, turns out to be geometrically natural (as it encodes the ``degrees of freedom" described in the previous item);
\item the notions of \emph{effective index} and \emph{effective nullity} of a minimal submanifold of a round sphere (and, by abuse of language, of the cone over it)  are presented in Section \ref{sec:Refine} (in fact the former is already present in our earlier work, see \cite[Section 2]{CLW24}); more specifically, we refer the reader to Definition \ref{def:EffMorsIndexNull};
 
\item the theorem holds true equally well for $C^\infty$ metrics (which is tacitly understood)
or for metrics of finite $C^q$ regularity with
$q\geq q_0(k,n)$, and -- in both cases -- the word \emph{generic} is understood in the sense of Baire; it is only for the ``improved counting claim" in the generic case that the $C^\infty$ version is derived from the $C^q$ version;
\item if $\Sigma$ has no nontrivial Jacobi fields,
Corollary \ref{cor:ExistenceUnobstructedCase}
gives existence and uniqueness of an $\MSI$, that is a
generalized graph over the fixed representative
$\Sigma$ for every sufficiently close metric.
Under the hypotheses of the theorem, $G/H$
is then a finite set, possibly containing
more than one element.
\end{itemize}
\end{remark}

There is a special instance of applicability of Theorem \ref{thm:main} that warrants special attention for -- among other things -- both assumptions above can be checked at once.

\begin{maincorollary}\label{cor:main} Let $\Sigma_0$ be a closed, non totally geodesic, minimal submanifold in the round sphere $\SSp^{N-1}$, and let $\Sigma$ denote the spherical suspension of $\Sigma_0$ in $\SSp^{N}$. Assume that the effective index, and the effective nullity of $\Sigma_0$ both vanish.
Then: for any metric $g$ that is sufficiently close to the standard round metric the manifold $(S^N,g)$ contains at least $LSN(G/H)$ minimal submanifolds with only strongly isolated singularities ($\MSI$) modeled on $\Sigma$, and generically at least $MN(G/H)$ such submanifolds. 

The preceding hypotheses are indeed satisfied, and thus the corresponding conclusions hold true, in the case when $\Sigma_0$ is the image of the Veronese embedding of $\mathbb{KP}^2$ into the round sphere $\SSp^{N-1}$ (as defined in Appendix \ref{sec:topology}, see Remark \ref{rem:VeroneseOtherFields}), where $N = 3\dim_{\mathbb{R}} \mathbb{K} + 2$ and $\mathbb{K}$ is one of the normed division algebras $\mathbb{R}$, $\mathbb{C}$, $\mathbb{H}$, $\mathbb{O}$.
In particular, for the Veronese embedding of $\mathbb{RP}^2$, we have $LSN(G/H)=4$ and $MN(G/H)=8$.
\end{maincorollary}

\begin{remark}\label{rema:mainVeronese}
In the instance of $\mathbb{RP}^2$ embedded in $\SSp^{4}$ described in the statement, the space $G/H$ is diffeomorphic to a $B^7$-bundle over $\SSp^5$ where $B^7$ stands for the (7-dimensional) Berger space (see Appendix \ref{subs:Veronese}). 
In fact, using adequate tools in Algebraic Topology one can show that the fibration in question is non-trivial (in particular: the 12-dimensional manifold in question is not a product $\SSp^5\times B^7$).
This study and the proper justification of the previous multiplicity bounds are the object of Appendix \ref{sec:topology}.
\end{remark}

\begin{remark}\label{rem:OtherFields}
The second part of the previous corollary clearly builds upon the computation of the index and nullity of the Veronese $\mathbb{KP}^2$; in the case $\mathbb{K}=\mathbb{R}$ this follows from work by Karpukhin \cite{Kar21}, while for $\mathbb{K}=\mathbb{C}, \mathbb{H}$, or $\mathbb{O}$ we rely on the very recent result by Rauchenberger and Semmelmann \cite{RauSem26}, who proved (rephrased in the terminology of the present paper) that the effective index and nullity of the Veronese embedding of $\mathbb{K}\mathbb{P}^2$, for $\mathbb{K}$ one of the three fields above, both vanish. The reader is referred to Lemma \ref{lem:RauSem}. Multiplicity bounds in these cases could be
obtained by an approach similar to that of
Appendix \ref{sec:topology}, though computing
the cohomology rings of the corresponding
critical manifolds would involve more elaborate
Serre spectral sequence calculations, possibly
with computer assistance.
\end{remark}

This persistence result, showing that certain $\MSI$ occur on open sets of ambient metrics, demonstrates
the sharpness of the non-persistence theorem in \cite{CLW24}.
Furthermore, it follows from a recent, related result by the authors (see \cite{CLW26}, to be compared with the improved index estimate by Bernstein and Ketover \cite{BK26}, and with the earlier ones proven in \cite{KusWan24} for the special case of tori) that for a Baire-generic smooth Riemannian metric $g$ on a five-manifold $M$, the link of the tangent cone at every singular point of a three-dimensional $\MSI$ is congruent to the Veronese projective plane.
 Consequently, if $\Sigma_0\subset\SSp^4$ is a closed,
embedded minimal surface that is neither totally
geodesic nor congruent to the Veronese projective plane,
its spherical suspension cannot persist in the sense
of Theorem \ref{thm:main}.
Finally, a two-dimensional $\MSI$ with connected
singular links is necessarily smooth, since every
connected closed minimal curve in a round sphere
is a great circle. Together with the generic
non-persistence result in codimension one from
\cite{CLW24}, this shows that the Veronese example
in Corollary \ref{cor:main}, with $n=3$ and $N-n=2$,
realizes the smallest possible dimension and
codimension for persistence of non-smooth $\MSI$
with connected singular links.

It is to be mentioned that partly similar results have been obtained for other, significant geometric problems. Focusing on those that are closest to the present work, we shall cite here the study of special Lagrangian conifolds by Joyce  (see in particular \cite[Section 5]{Joy03}, that is based upon \cite{Joy04}) and, much more recently, the results by Gutwein and Langlais on deformation of harmonic maps with isolated singularities, see \cite{GutLan26}.

\subsection{Technical challenges and structure of the paper}\label{subs:Technical challenges} Let us now briefly illustrate the technical challenges that naturally come up when aiming for a perturbative theory for $\MSI$ and, in parallel, describe the organization of the present manuscript. To begin with, let us ignore -- just for a second -- the anticipated necessity to move singularities and ``rotate" the corresponding limit cones (we will come to that point later) and just focus on developing a perturbation theory for normal graphs over a given $\MSI$. In order to ensure embeddedness of the submanifolds one would produce (thus: to ensure one would indeed obtain $\MSI$) it is necessary to consider superlinearly decaying sections of the normal bundle (namely: sections $u\in W^{k,2}_{\tau}(\Sigma;\mbfV)$ for $\tau>1$). In turn, this immediately poses two problems. On the one hand, we need to prove a version of the ``counting formula" for the Fredholm index of the Jacobi (stability) operator of an $\MSI$ for such a range of weights (which is different than \cite[Section 3]{CLW24}, where -- aiming for a strongest possible general \emph{non-persistence} result -- we had, rather, to allow for milder decay on approach to the singularities). This task is accomplished in Section \ref{sec:Refine}, and what we get (see Proposition \ref{thm:CountRefined}) is a formula that brings both the Morse index \emph{and the nullity} of the links at the singularities into play. In particular, in order for the Jacobi operator of the $\MSI$ in question to have Fredholm index zero it is necessary that the \emph{effective index and nullity vanish at all singularities}. Roughly speaking, this means that for the link of each singular cone (that is a minimal submanifold of a round sphere $\SSp^{N-1}$) all contributions to the nullity must be imputable to ambient isometries (namely to a subgroup of $\Ogroup(N)$) and all contributions to the index must be imputable to ambient translations $\R^{N}$ (as first noted by Simons in \cite{Sim68}). At that stage, it becomes quite clear that, in order to make the obstruction space as small as possible (and, when feasible, to reduce to an unobstructed problem) one needs to allow for deformations of $\MSI$ that are more general than graphical, and rather encompass (local) translations of the singularities as well as (local) rotations of the corresponding cones. But, as even basic examples show (think of a polygon in the plane) we cannot possibly (locally) parametrize the space of all such deformations of a given $\MSI$ by means of (suitably regular\ldots)  sections of the normal bundle thereof. The way-around we designed is, essentially, to decouple the problem: we handle the local rigid motions near singularities by means of suitably designed diffeomorphisms, and the graphical deformations by sections of the normal bundle. 

Here, a second challenge naturally comes up:  Euclidean (regular) minimal cones do not maintain their minimality as we consider a different positive-definite quadratic form. In fact, we witness linear blow-up of the mean curvature for arbitrarily small perturbations of the ambient metric.
Hence, if we again begin by focusing solely on graphical deformations, the natural ``mean curvature map" (see below for due clarifications) is by no means well-defined as a map between the natural functional spaces $W^{k,2}_{\tau}(\Sigma;\mbfV)\to W^{k-2,2}_{\tau-2}(\Sigma;\mbfV)$. Thus, we also need to introduce another class of diffeomorphisms that take care of the necessary ``local adjustments" (roughly speaking: anisotropic dilations) which have to be done: this is the object of Section \ref{sec:Diffeo}. The aforementioned ``decoupling" is reflected in the way we set up (e.\,g. in terms of choice of functional spaces) the deformation problem locally, around a given $\MSI$, which we do in Section \ref{sec:AnMeanCurv}. In particular, in analogy with the approach in \cite{Whi91}, the need of working with section of a fixed bundle (that is, the normal bundle of the $\MSI$ in question at the reference metric) forces us to introduce a map (which we named \emph{analytic mean curvature}) that has the desired mapping properties and serves the basic scope of detecting minimality like the standard mean curvature would do. That being said, it is quite a technical challenge to actually justify the expected regularity of such a map. We present a rather direct approach through (carefully designed) local trivializations of the normal bundle of an $\MSI$ and a fairly general study of abstract composition maps as bundle maps. This is done in Appendix \ref{app:C^q}, that plays a technically important role in the economy of the present paper.

All in all, the machinery developed in Section \ref{sec:AnMeanCurv} is sufficient to get a perturbative theory in the unobstructed case (as spelled out in Corollary \ref{cor:ExistenceUnobstructedCase}), but is only a piece of the puzzle when working with $\MSI$ having non-trivial Jacobi fields. What we accomplish in Section \ref{sec:LSreduction} (see, in particular, Proposition  \ref{Prop:Main_LS_Red}) is then to construct a suitable functional $\mathfrak{a}:\Gamma(q)\times (G/H)\to\R$ providing a ``Lyapunov-Schmidt reduction" of the problem (i.\,e. a functional whose critical points determine $\MSI$, and in fact all and only appearing in a perturbative regime).
This is highly non-trivial for the following reason. Fix a reference $\MSI$, say $\Sigma$, in $(M,\bar g)$:
each member $\Sigma'=\gamma(\Sigma)$ of its
isometry orbit corresponds to the coset
$\gamma\cdot H\in G/H$, whereas on the other hand defining a ``generalized normal graph" over $\Sigma'$ will crucially depend on the choice of the representative for the coset in question. In other words, the (inherently topological) impossibility of selecting global ``slice maps" for the problem at hand obstructs the natural construction one would like to go for. Note that in other geometric circumstances this obstacle is overcome due to some special topological features of the critical manifold (e.\,g. it is associated to a parallelizable bundle, or similar \ldots), which however cannot be hoped for in general. 

With such tools at disposal, in Section \ref{sec:MainProofs} we put everything together to provide the proofs of the main results. Note that the multiplicity results in Corollary \ref{cor:main} in the case $\mathbb{K}=\R$ require, on the one hand, the careful study of the space of Veronese footballs in $\SSp^5$ (see Appendix \ref{sec:topology}), and, on the other hand (for the corresponding improved bound in the Baire-generic case), a suitable Sard-type Lemma for $\mathfrak{a}$ as one varies the metric $g$ in an open neighborhood of $\bar g$, that is item \label{Item_LS_a-generic-nondegenerate} of Proposition \ref{Prop:Main_LS_Red}.

\section{General setup and notation}\label{sec:Prelim}

\subsection{Key definitions and preliminaries}\label{subs:DefPrelim}Throughout this article, $M$ is a smooth, connected, compact
boundaryless manifold of dimension $N\geq3$.
Regularity of maps and tensors is understood
with respect to smooth local coordinates.
For $q\in\mathbb N\cup\{\infty\}$, $q\geq2$,
let $\Gamma(q)$ denote the space of $C^q$
Riemannian metrics on $M$.
For $g\in\Gamma(q)$, we denote by $\dist_g$
the induced distance, by $\scH^n(g,\cdot)$
the corresponding $n$-dimensional Hausdorff
measure, and by $\inj^g_p(M)$ the injectivity
radius at $p\in M$; $\exp^g_x(\cdot): T_x M\to M$ stands for the exponential map, in metric $g$, with center $x\in M$.

\begin{definition}\label{def:RegSet}
Let $(M,g)$ be a Riemannian manifold of dimension $N\geq 3$; let $B^{g}_r(x)$ denote the metric ball of center $x$ and radius $r$, with respect to the metric $g$. Given an integral varifold $V$, of dimension $2\leq n<N $, we say that $p\in\spt(\|V\|)\subset M$ belongs to its regular set if there exists $r>0$ such that $\spt(\|V\|)\cap \overline{B^g_r(p)}$ is a smooth, embedded, compact, connected $n$-dimensional submanifold of class $C^2$ with boundary contained in $\partial B^g_r(p)$. The set of regular points of $V$ shall be denoted by $\Reg(V)$; its complement in $\spt(\|V\|)$ will be referred to as singular set of $V$, i.\,e. we let $\Sing{V}=\spt(\|V\|)\setminus \Reg(V)$.
\end{definition}

Note that, when we say that $\Sigma$ is an embedded $C^{\ell\geq 2}$ submanifold of dimension $2\leq n<N$ we mean that for every $p\in \Sigma$, there exist $C^\infty$ coordinate charts $M\supset U_p\to \RR^N$ on a neighborhood of $p$, mapping $\Sigma$ to a graph of some $\RR^{N-n}$-valued function, of class $C^{\ell}$.

\begin{remark}\label{rem:OptRegMinSubm}
 $\Reg(V)$ (respectively: $\Sing{V}$) is open (respectively: closed) in $\spt(\|V\|)$. Furthermore, if $V$ is stationary, then by standard elliptic regularity (Schauder theory applied to the minimal surface system) if -- in the setting of the preceding definition -- $p\in\Reg(V)$ and the ambient metric $g$ is $C^{q}$ then, in fact, $\spt(\|V\|)$ is an  $n$-dimensional submanifold of class $C^{q,\alpha}$ in an open neighborhood of the point in question for any $\alpha\in (0,1)$; in particular if the ambient metric is smooth (by which we shall mean $C^{\infty}$) then so will be $\spt(\|V\|)$, at least away from the singular set of $V$.
\end{remark}

Let us proceed and recall from \cite[Section 2]{CLW24} the definition of the objects that we study throughout this article:

\begin{definition}\label{def:MSI}
Let $(M,g)$ be a Riemannian manifold of dimension $N\geq 3$. For every stationary integral varifold $V$, of dimension $2\leq n<N $, we call a point $p\in \Sing(V)$ \textbf{strongly isolated}, if \emph{some} tangent cone of $V$ at $p$ is \textbf{regular}, i.\,e. it is of multiplicity one and has smooth link (equivalently: if the singular set of such a cone coincides with the origin). We will then say that a stationary integral varifold $V$ \textbf{has only strongly isolated singularities} (and refer to it as $\MSI$) if  every $p\in \Sing(V)$ is strongly isolated (vacuously true when $\Sing(V) = \emptyset$).
\end{definition}

\begin{remark}\label{rem:MSI}
   We remark that, in the setting of the preceding definition, at any singular point $p\in\Sing(V)$ there holds uniqueness of the tangent cone in question, thanks to a deep result of Simon \cite{Sim83} (see also \cite{Sim85}). Furthermore, the corollary stated after Theorem 5
therein implies that every singular point is
isolated. Thus $\Sing(V)$ is discrete and, being closed in the compact manifold $M$, finite.
 \end{remark}

 \begin{remark}\label{rem:MSInotation} Throughout this article, we restrict attention
to multiplicity-one varifolds.
   Based on the previous remark, we will switch to a somewhat more  convenient notation: given one such varifold, we will rather employ the letter $\Sigma$ to denote it, meaning that $V=|\Sigma|$ (that is to say: $V$ is the \emph{multiplicity 1} integral varifold associated to the submanifold $\Sigma$)
   and we will practically identify $\Sigma$ with the regular part of $\overline{\Sigma}$. Thus $\Sigma$ is treated as a regular (cf. Remark \ref{rem:OptRegMinSubm}) boundaryless but not necessarily closed submanifold. 
   The (unique) tangent cone of $\Sigma$ at a singular point $p$ is denoted by $\mbfC_p(\Sigma)$, or simply $\mbfC_p$ if there is no risk of confusion. 
\end{remark}

For $\Sigma$ a fixed $\MSI$ in a Riemannian
manifold $(M,g)$ with $g$ smooth, let
$\injrad(g,\Sigma)$ be defined as in \cite{CLW24}.
We say that
$
\rho=\rho_{\Sigma,g}
\in C^0(M;[0,1])
\cap C^\infty(M\setminus\Sing(\Sigma);\R)
$
is a radius function for $\Sigma$ with respect
to $g$ if it is positive on
$M\setminus\Sing(\Sigma)$ and there exists
$0<r_0\leq\injrad(g,\Sigma)/10$ such that
$
\rho_{\Sigma,g}
=\dist_g(\Sing(\Sigma),\cdot)$
 on
$\bigcup_{p\in\Sing(\Sigma)}B^g_{9r_0}(p).$
We call the largest such $r_0$ the scale of
$\rho$. 
When $\Sing(\Sigma)=\emptyset$, we take
$\rho_{\Sigma,g}\equiv1$ and
$r_0=\injrad(g,\Sigma)/10$.

\subsection{Distinguished metric and related notation.}

Throughout the article, we let $\bar g$ denote a distinguished $C^\infty$ Riemannian metric on $M$ (in various applications, like e.\,g. Corollary \ref{cor:main} $\bar g$ will be the standard, unit round metric on the $N$-dimensional sphere), and the letter $g$ will typically denote a nearby metric, that is suitably close to $\bar g$ in $\Gamma(q)$ (as specified at due course). We let $\Diff(M)$ denote the space of diffeomorphisms on $M$, endowed with the smooth topology determined by the background metric $\bar g$; similarly all norms on $\Gamma(q)$ are understood measured with respect to $\bar g$ in the standard fashion. We shall denote by $\Gamma_{\delta}(q)$ the open subset of $\Gamma(q)$ consisting of metrics distant less than $\delta$ from $\bar g$.
Since we are concerned here with existence and multiplicity results of perturbative character, we will adopt privileged notation for $\bar g$:  the $\bar \cdot$ symbol shall typically identify differential operators in metric $\bar g$ such as e.\,g. for $\bar\nabla$ that is the Riemannian gradient of a regular function. 

\subsection{Conventions.}
We shall collect here some more notation and conventions that are assumed throughout the paper.

\vspace{2mm}

\textbf{Cutoff functions and their scalings.} First of all, let $\zeta\in C^\infty(\RR; [0,1])$ be a monotonically non-decreasing cutoff function (fixed once and for all) such that $\zeta \equiv 0$ on $(-\infty, 1]$ and $\zeta \equiv 1$ on $[2, +\infty)$. Given $\delta>0$ we shall then set $\zeta^{(\delta)}(s)=\zeta(s/\delta)$.
    Secondly, let $\xi\in C^\infty(\RR; [0,1])$ be given by $\xi:=1-\zeta$ and, in analogy with the above, we declare  $\xi^{(\delta)}(s)=\xi(s/\delta)$.

            \vspace{2mm}

\textbf{Euclidean models.} In $\R^N$ we shall denote by $g_{\eucl}$ the standard Euclidean metric and we let $\BB_r(x)$ denote the open Euclidean ball of center $x\in \R^N$ and radius $r>0$; in the special case of balls centered at the origin we shall simply write $\BB_r$. The ambient dimension will be explicitly indicated whenever this is appropriate to avoid confusion, as for $\BB^N_r$ vs. $\BB^n_r$.

            \vspace{2mm}

\textbf{Constants and comparison notation.}
The letter $C$ denotes a positive constant whose
value may change from one occurrence to another.
We indicate its dependence by writing, for example,
$C=C(\Lambda,\sigma)$.
The dimensions $n,N$ and the cutoff functions fixed
above are omitted from such lists.

For nonnegative quantities $A,B$, we write
$A\lesssim B$ if $A\leq CB$, and $A\sim B$ if both inequalities
hold. Subscripts indicate the permitted dependence
of the implicit constants, in addition to any
dependence explicitly stipulated for the section
or proof under consideration; for instance, we write $A\sim_{a, b, \dots} A'$ if for some constant $C(a, b, \dots)>1$, there holds $C(a, b, \dots)^{-1}A'\leq A\leq C(a, b, \dots)A'$. The same convention applies to constants implicit
in $O(\cdot)$. When no dependencies are displayed, they are those
specified in the surrounding argument, but
whenever uniformity in a varying metric or another parameter is required, this is stated explicitly.
In statements outside Appendix \ref{app:C^q},
we use explicit constants and specify their
dependence. In Appendix \ref{app:C^q}, we also
allow for asymptotic notation in statements, with dependencies specified by subscripts and the local conventions.

\section{A refined counting formula}\label{sec:Refine}

Here we compute the Fredholm index of the Jacobi operator of $\MSI$ in terms of the links of the cones at the corresponding singular points. This is a refinement, related to the necessity of considering superlinearly decaying sections, of \cite[Theorem 3.2]{CLW24}.

\

\textbf{General notation: the stabilizer of a cone.} Throughout this section, let $\Sigma$ be an
$\MSI$ in the fixed smooth Riemannian manifold
$(M,\bar g)$ and let $\rho_{\Sigma,\bar g}$ be a radius function thereof (with $r_0$ denoting its scale).
 
  Suppose $\Sing(\Sigma) = \{p_1, \dots p_Q\}$; $\mbfC_j\subset T_{p_j}M$ be the tangent cone of $\Sigma$ at $p_j$ (defined by blowing up $(\exp_{p_j}^{\bar g})^{-1}(\Sigma)$ centered at $0$), $\mathfrak{g}_j\subset \mathfrak{so}(T_{p_j}M)$ be the Lie algebra of the stabilizer $G_j\subset\Ogroup(T_{p_j}M)$ of $\mbfC_j$ (equivalently: $G_j$ is the subgroup leaving the link $\mbfC_j\cap\SSp^{N-1}$ invariant as a set) and $\mathfrak{g}_j^\perp$ be its orthogonal complement with respect to the metric 
  \[
  \langle X, Y\rangle := -\Tr(X\circ Y) \ \text{on} \ \mathfrak{so}(T_{p_j}M).
  \]

Recall that the Lie algebra $\mathfrak{so}(N)$ of $\Ogroup(N)$ (or, equivalently, of $\SOgroup(N)$) in Euclidean $\R^{N}$ is proven to correspond to the space of $N\times N$ skew-symmetric matrices and thus has dimension equal to $N(N-1)/2$. As a result, the dimension of $\mathfrak{g}_j^\perp$ is exactly equal to $N(N-1)/2-\dim \mathfrak{g}_j$.

(For instance, as we will see in Appendix \ref{sec:topology}, in the case of the spherical suspension of the Veronese embedding of $\R\mathbb{P}^2$ in $\SSp^4$ one has $\dim\mathfrak{g}_j=3$ for $j=1,2$ (corresponding to the singularities at the north and south poles) inside the $10$-dimensional linear space $\mathfrak{so}(5)$.) 

\

\textbf{Fredholm index of the Jacobi operator.} We will now proceed and sharpen the analysis carried out in \cite[Section 3]{CLW24} essentially by modifying (i.\,e. further enlarging) the domain of the Jacobi operator. In what follows we let $\mbfV$ denote the normal bundle of $\Sigma$ (regarded as an open, smoothly embedded submanifold) in $(M,\bar g)$.

 \begin{definition}\label{def:RotLikeSection}
  Let $\Sigma$ be an $\MSI$ in $(M,\bar g)$, we define a section $\psi\in C^{\ell}_{loc}(\Sigma;\mbfV)$ \textbf{rotation-like} if for every $q=p_j\in \Sing(\Sigma)$, there exist a neighborhood $U\ni q$ in $M$, an orthonormal basis $\left\{e_1,\ldots, e_N\right\}$ of $T_q M$ (with associated normal coordinates $\left\{x\right\}$ centered at $q$) and $\mbfA\in \mathfrak{g}_j^{\perp}$ such that, under the identifications induced by such coordinates, there holds
  \[
  \psi|_U(x) = \Pi^\perp_{x}(\mbfA x)
  \]
  where \[
    \Pi^\perp_x: T_xM \to T_xM
  \]  
  is the $\bar g$-orthogonal projection map onto $T_x^\perp \Sigma (:=\mbfV_x)$. 
  \end{definition}

  \begin{definition}\label{def:BiAugmentedSpaces} Let $\ell\geq 2$ and $1\leq p<\infty$.
  If $W^{\ell,p}_{loc}(\Sigma;\mbfV)$ denotes the space of
	 sections that are locally in the Sobolev space $W^{\ell,p}$ then one sets  
 \begin{multline*}
 \hat{W}^{\ell,p}_\tau(\Sigma;\mbfV) := \{u\in W^{\ell,p}_{loc}(\Sigma;\mbfV): \|u-\phi-\psi\|_{W^{\ell,p}_{\tau}}<+\infty \\
 \text{ for some translation-like function} \ \phi \
 \text{and for some rotation-like function} \ \psi \}\,
  \end{multline*}
  where the norm $\|\cdot\|_{W^{\ell,p}_\tau}$ and the corresponding weighted Sobolev space $W^{\ell,p}_{\tau}(\Sigma;\mbfV)$ are defined exactly as in \cite[Section 2.3]{CLW24}. (In particular, we are considering the same weight along all ends, as prescribed by the scalar $\tau\in\R$.)
  \end{definition}

  For $\tau\geq 1$ we simply have the direct-sum decomposition
  \be\label{eq:Augment}
 \hat{W}^{\ell,p}_\tau(\Sigma;\mbfV)=W_{\text{Tr}}(\Sigma;\mbfV)\oplus  W_{\text{Rot}}(\Sigma;\mbfV)\oplus {W}^{\ell,p}_\tau(\Sigma;\mbfV)
  \ee 
  for $W_{\text{Rot}}(\Sigma;\mbfV)$ defined as the span of
  \be\label{eq:BasisRotLikeFunctions}
\left\{\xi_{\Sigma,\bar g,r_0} \ \cdot \  \Pi^{\perp}_x(\mbfA^i_j x) \st j=1,\ldots, Q, \ \mbfA^i_j\in \mathfrak{g}_j^{\perp} \right\}
\ee
  
where $\xi_{\Sigma,\bar g,r_0}(x):=\xi(\rho_{\Sigma,\bar g}(x)/r_0)$ in terms of the smooth cutoff function introduced in the dedicated paragraph of Section \ref{sec:Prelim} and it is understood that $\mbfA^i_j$ varies in a basis for the linear space $\mathfrak{g}_j^{\perp}$ associated with the identification of $T_{p_j}M$ with $\R^N$ descending from a choice of orthonormal basis (see Definition \ref{def:RotLikeSection}). In particular, this space has dimension equal to
\[
\dim W_{\text{Rot}}(\Sigma;\mbfV)= \sum_{i=1}^{Q}\left(\dim \frak{so}(N)-\dim \mathfrak{g}_i\right)=\frac{N(N-1)Q}{2}-\sum_{i=1}^Q\dim \mathfrak{g}_i \,.
\]

To proceed, we need to prove a result analogous to \cite[Lemma 2.22]{CLW24}. In what follows, and throughout the present article, we let $L_{\Sigma,\bar g}$ denote the Jacobi (stability) operator of $\Sigma$ in $(M,\bar g)$; the reader may wish to consult \cite[Section 2]{CLW24} for background about the second variation of an $\MSI$.

\begin{lemma}\label{lem:JacobiWellPosedRotations}
Let $\Sigma$ be an $\MSI$ in $(M,\bar g)$.
For every $\phi\in W_{\textrm{Rot}}(\Sigma;\mbfV)$,
there holds
$L_{\Sigma,\bar g}\phi\in L^\infty(\Sigma;\mbfV)$.
More precisely, there exists
$C=C(\bar g,\Sigma)>0$ such that, near each
$q=p_j\in\Sing(\Sigma)$, in $\bar g$-normal
coordinates centered at $q$,
\[
\bigl|L_{\Sigma,\bar g}((\mbfA x)^\perp)(x)\bigr|
\leq C|x|
\]
for every $\mbfA\in\mathfrak g_j^\perp$ of unit norm.
Consequently,
\[
L_{\Sigma,\bar g}\phi
\in W^{0,p}_{1-\varepsilon}(\Sigma;\mbfV)
\]
for every $1\leq p<\infty$ and
$0<\varepsilon<1$.
\end{lemma}

\begin{proof}
Fix $q=p_j\in\Sing(\Sigma)$ and use
$\bar g$-normal coordinates centered at $q$.
For $\mbfA\in\mathfrak g_j^\perp$ of unit norm,
let $F_t$ be the locally defined diffeomorphism determined by
$F_t(x)=e^{t\mbfA}x$, and set
\[
g_t:=F_t^*\bar g,\qquad
h:=\left.\frac{d}{dt}\right|_{t=0}g_t
=\mathcal L_V\bar g,\qquad V(x):=\mbfA x.
\]
Naturality of the mean curvature vector gives
\[
H_{\Sigma,g_t}(x)
=(dF_t|_x)^{-1}
H_{F_t(\Sigma),\bar g}(F_t(x)),
\]
on the domain where these quantities are defined.
Since $\Sigma$ is minimal, differentiation
of this equation at $t=0$ yields
\[
L_{\Sigma,\bar g}(V^\perp)
=\left.\frac{d}{dt}\right|_{t=0}H_{\Sigma,g_t}.
\]
At each regular point, the tangential component
of $V$ induces only a local reparametrization
of $\Sigma$, whose contribution to the
linearized mean curvature vanishes since
$H_{\Sigma,\bar g}\equiv0$.

The metric variation formula for mean curvature
therefore implies
\[
|L_{\Sigma,\bar g}(V^\perp)|
\leq C\bigl(
|h|\,|\II_{\Sigma,\bar g}|+|\bar\nabla h|
\bigr),
\]
with all norms and covariant derivatives taken
with respect to $\bar g$.

Write $\bar g=g_{\eucl}+b$ in the chosen
normal coordinates. Smoothness of $\bar g$ gives
\[
|b(x)|\leq C|x|^2,\qquad
|Db(x)|\leq C|x|,\qquad
|D^2b(x)|\leq C.
\]
Since $\mbfA$ is skew-symmetric,
$\mathcal L_Vg_{\eucl}=0$, so $h=\mathcal L_Vb$.
Consequently,
\[
|h(x)|\leq C|x|^2,\qquad
|\bar\nabla h(x)|\leq C|x|.
\]
The smooth convergence to the regular tangent
cone on rescaled annuli also gives
$|\II_{\Sigma,\bar g}(x)|\leq C|x|^{-1}$.
Combining these estimates proves
\[
\bigl|L_{\Sigma,\bar g}((\mbfA x)^\perp)(x)\bigr|
\leq C|x|,
\]
uniformly over the finitely many singular points
and all unit-norm $\mbfA$.
For an arbitrary $\phi\in W_{\textrm{Rot}}$,
write $\phi=(\mbfA_jx)^\perp$ near each
$p_j\in\Sing(\Sigma)$. By linearity, the
preceding estimate gives
\[
|L_{\Sigma,\bar g}\phi|
\leq C\Bigl(\max_j|\mbfA_j|\Bigr)\rho
\]
near the singular set, with the same
$C=C(\bar g,\Sigma)$. Away from that set,
$L_{\Sigma,\bar g}\phi$ is smooth.
Thus it is globally bounded and belongs to
$W^{0,p}_{1-\varepsilon}(\Sigma;\mbfV)$
for every $1\leq p<\infty$ and
$0<\varepsilon<1$, by the definition of the
weighted norm and the conical volume growth
near each singular point.
\end{proof}

With this notation in place, we wish to study the mapping properties of the Jacobi operator
\[
L_{\Sigma,\bar g}: \hat{W}^{\ell,2}_\tau(\Sigma;\mbfV) \to W^{\ell-2,2}_{\tau-2}(\Sigma;\mbfV).
\]

We recall, for each $\mbfC=\mbfC_p(\Sigma)$ with $p\in\Sing(\Sigma)$, that $\Gamma(\mbfC)$ denotes the asymptotic spectrum of the cone $\mbfC$ (that is to say: the collection of all possible rates of radial decay/growth of homogeneous Jacobi fields along the cone in question) and further set
\be\label{eq:UpperThreshold}
   \gamma_-(\mbfC) := \sup (\Gamma(\mbfC)\cap \R_{<1})\in[0,1), \ \text{and} \ \gamma_+(\mbfC) := \min\left\{\inf (\Gamma(\mbfC)\cap \R_{>1}); 2\right\} 
    \ee
and, in turn,
\[
\gamma_-(\Sigma) := \sup_{p\in\Sing(\Sigma)}\gamma_{-}(\mbfC_p)\in[0,1), \ \text{and} \ \gamma_{+}(\Sigma)=\inf_{p\in\Sing(\Sigma)}\gamma_+(\mbfC_p)
\]
    whence the open interval(s)
\be
  I_{-}(\Sigma)=(\gamma_{-}(\Sigma),1)\subset (0,1), \quad 
  I_{+}(\Sigma)=(1,\gamma_{+}(\Sigma))\subset (1,2).  \label{Def_Intervals I_-, I_+}
\ee

\begin{definition}\label{def:EffMorsIndexNull}
Let $\Sigma_0$ denote a closed, embedded minimal submanifold in the round sphere of dimension $N-1$, i.\,e. $\SSp^{N-1}\subset \R^N$, and let $\mbfC$ denote the cone over $\Sigma_0$ having vertex at the origin $0\in\R^N$ (that is: $\mbfC:=\mathbf{0} \conetimes \Sigma_0$). Then we define the \textbf{effective Morse index} of $\Sigma_0$ as
\be\label{Equ_GenReg:EffectiveInd_Cone}
    \rmI(\mbfC) := (\,\text{Morse index of the Jacobi operator of } \mbfC\cap \mathbb{S}^{N-1}\,) - N \ee
and, denoted $\mathfrak{g}$ the Lie algebra of the stabilizer of $\Sigma_0$, the \textbf{effective nullity} of $\Sigma_0$ as
\be\label{Equ_GenReg:EffectiveNul_Cone}
    \rmJ (\mbfC) := (\,\text{nullity of the Jacobi operator of } \mbfC\cap \mathbb{S}^{N-1}\,) - \left( \frac{N(N-1)}{2}-\dim\mathfrak{g} \right). \ee

(Here it is to be stressed that we are not assuming $\Sigma_0$ to be connected, and note that -- if $\Sigma_0$ is not totally geodesic -- both numbers in question are non-negative.)    
\end{definition}

By following, almost verbatim \cite[Section 3]{CLW24} (but now considering a different range of values of $\tau$) we prove the following result:

\begin{proposition}
 \label{thm:CountRefined}
Let $(M, \bar g)$ be an $N$-dimensional closed Riemannian manifold, 
    and let $\Sigma$ be an $\MSI$ thereof. 
    Let $\ell\geq 2$ and 
    \be\label{eq:RangeTauExt}
      \tau \in  I_{+}(\Sigma)\,.
    \ee
    Then the Jacobi operator \be\label{eq:JacobiExtDomain}
      L_{\Sigma,\bar g} : \hat{W}^{\ell,2}_\tau(\Sigma;\mbfV) \to W^{\ell-2,2}_{\tau-2}(\Sigma;\mbfV)
    \ee
    is a Fredholm operator of index \[
      \widehat{\ind}_{\tau}(L_{\Sigma,\bar g}) = -\sum_{p\in \Sing(\Sigma)} \rmI(\mbfC_p) -\sum_{p\in \Sing(\Sigma)} \rmJ(\mbfC_p)\,.
    \]
   \end{proposition}

\begin{proof}(Sketch)
One first looks at the Jacobi operator on standard weighted Sobolev spaces, namely \[
      L_{\Sigma,\bar g} : W^{\ell,2}_\tau(\Sigma;\mbfV) \to W^{\ell-2,2}_{\tau-2}(\Sigma;\mbfV)\,;
    \]
    by the way $\gamma_{+}$ is defined, if we take $\tau$ in the range above we find 
    that the Fredholm index is equal to $-\sum_{p\in \Sing(\Sigma)}(\text{Morse index} \ + \text{nullity of} \ \mbfC_p(\Sigma)\cap\SSp^{N-1})$; this is a direct consequence of Remark 3.1 therein. Then, by virtue of the augmentation of the domain, the Fredholm index increases by $NQ$ due to $W_{\text{Tr}}(\Sigma;\mbfV)$ plus $\sum_{p\in\Sing(\Sigma)}(N(N-1)/2-\dim\Stab_{\;\SOgroup(N)}(\mbfC_p(\Sigma)\cap\SSp^{N-1}))$ due to $W_{\text{Rot}}(\Sigma;\mbfV)$. If we keep in mind the definitions of effective index and nullity and rearrange the terms we get the desired conclusion. 
\end{proof}

\begin{remark}\label{rem:Inclusions}
For later reference, keeping in mind the definition of weighted Sobolev spaces, let us note that there always holds $-2(\tau-2)-n<0$ if $n\geq 2$ and $\tau\in I_{+}\subset (1,2)$ and thus one has the inclusions $W^{\ell-2,2}_{\tau-2}(\Sigma;\mbfV)\subset W^{0,2}_{\tau-2}(\Sigma;\mbfV)\subset L^2(\Sigma;\mbfV)$.
\end{remark}

\begin{example}\label{ex:DisjSpheres}
Let $n\geq2$, set $N=2n$, and let
$\Sigma_0\subset\SSp^{2n-1}$ be the union of
the great $(n-1)$-spheres in two orthogonal
complementary $n$-planes.
Its Morse index is $N$
(cf. \cite[Remark 2.9]{CLW24}), whereas its
nullity is $2n^2$
(cf. \cite[Proposition 5.1.1]{Sim68}).
The stabilizer Lie algebra has dimension
$2\dim\mathfrak{so}(n)=n(n-1)$.
Hence the spherical suspension
$\Sigma\subset\SSp^{2n}$ has two singular
points $p_1,p_2$, with
$\rmI(\mbfC_{p_i})=0$ and
\be\label{eq:EffNullSpheres}
\rmJ(\mbfC_{p_i})
=2n^2-\bigl(n(2n-1)-n(n-1)\bigr)
=n^2,\qquad i=1,2.
\ee
By \ref{thm:CountRefined}, the operator in
\eqref{eq:JacobiExtDomain} therefore has
Fredholm index $-2n^2$.

This negative index should be interpreted in
relation to the prescribed singular models.
The local adjustments of Section \ref{sec:Diffeo}
include anisotropic linear maps, but these
compensate for the change of ambient metric:
the resulting tangent cones are isometric
to the original ones when measured in the
respective ambient metrics. In particular,
their two sheets retain their relative angles.

By contrast, applying the smooth perturbation
theory of \cite{Whi91} separately to the two
great $n$-spheres concerns their respective
isometry orbits, rather than a fixed pair
or its orbit under a common ambient isometry.
The relative position of the resulting minimal
sheets is therefore not constrained in the
same way. When they meet transversely, their
union has cross singularities whose tangent
cones need not be isometric to the prescribed
ones. Persistence in this broader sense is
thus compatible with the negative Fredholm
index computed above, and differs from the
persistence with prescribed singular models
established by our main theorem.
\end{example}

\section{Deformations and diffeomorphisms}\label{sec:Diffeo}

The goal of this manuscript is essentially to find (under appropriate assumptions) $\MSI$ near a given one -- or congruent copies thereof -- under small perturbations $g$ of $\bar g$. For that purpose, based on the linear theory presented in Section \ref{sec:Refine}, in order to minimize the obstruction space it is appropriate to allow (not only for graphical deformation, but also) for motion of singular points and ``rotations" of the associated tangent cones.

\subsection{Generalized graphs}\label{subs:GenGraph}

In what follows, we let $(M,\bar g)$ be as stipulated in Section \ref{sec:Prelim} and  let $\Sigma$ denote an $n$-dimensional $\MSI$ in $(M,\bar g)$.

We let $\Sing(\Sigma)=\left\{p_1,\ldots, p_{Q}\right\}$ denote the singular set of $\Sigma$, and we define vector spaces
  \begin{align*}
    \scT := \prod_{j=1}^Q T_{p_j}M\,, & &
    \scR := \prod_{j=1}^Q \mathfrak{g}_j^\perp
  \end{align*}
  each of which has a canonical metric induced by Riemannian metric $\bar g$. 

For the sake of convenience, we also assume to have fixed (once and for all), at each such point $p_i$, an orthonormal frame for $T_{p_i}M$ which determines geodesic normal coordinates in a neighborhood of $p_i$ (with respect to $\bar g$); as usual, $\mbfC_i$ denotes the tangent cone of $\Sigma$ at $p_i$.

\

As anticipated in the introduction, minimal cones in $\R^N\cong (T_p M,\bar g_p)$ do not need to be minimal with respect to $g_p$, no matter how close $g$ might be to $\bar g$; roughly speaking, such cones may need to be deformed by means of an anisotropic dilation to be minimal in metric $g$. With that issue in mind, we will now present the design of a suitable class of diffeomorphisms that allows to deal with this issue. Let us begin with the following definition:
  
  \begin{definition}\label{def:MetricsCoincide}
Let $(M,\bar g)$ be a Riemannian manifold. We will say that a metric $g\in \Gamma(q)$ \textbf{coincides with $\bar g$ on $Z\subset M$} if for every $p\in Z$, there holds $g|_p = \bar g_p$ (as symmetric $2$-tensors on $T_pM$). In particular, for $\Sigma$ a given $\MSI$ in $(M,\bar g)$ we will denote by
\[
\Gamma(q; \bar g,\Sigma)
\]
the subset of $\Gamma(q)$ consisting of those $C^q$ Riemannian metrics that coincide with $\bar g$ on $\Sing(\Sigma)$.
\end{definition}

Throughout this section we let $0<r_0(\Sigma,\bar g)<\injrad(\bar g,\Sigma)/10$ and $\delta\in (0,r_0/2)$. For constructions associated with the fixed
submanifold $\Sigma$, we choose $r_0$ once and
for all in terms of $(\bar g,\Sigma)$.
Its dependence may then be absorbed into that
on $(\bar g,\Sigma)$.

We define $M(Q,r_0)$ as the space of unordered $Q$-tuples \(Z=\{p_1,\ldots,p_Q\}\subset M\) whose pairwise distances are greater than \(10r_0\), endowed with its natural smooth structure as an open subset of \(M^Q/S_Q\) for $S_Q$ the set of permutations of $Q$ objects. The following statement, whose proof is given in Section \ref{subs:ConstrDiff}, provides the desired ``local adjustment" at the singular points:

  \begin{lemma}\label{lem:LocalAdjMap}
  (Setting as above.)
    There exist $\delta=\delta(r_0, Q,\bar g)$ and a $C^q$ map
    \[
    \Phi: \Gamma_{\delta}(q)\times M(Q,r_0) \to \Diff(M)
    \]
    such that
\begin{enumerate}[label=(\roman*)]

\item $\Phi(g,Z)$ equals the identity on $Z$ and on $M\setminus \cup_{p\in Z}(B_{2r_0}(p))$; 
\item 
    $(\Phi(g,Z))^{*}g$ coincides with $\bar g$ on $Z$;
    \item\label{item:Locality} if $g$ coincides with $\bar g$ on $Z$ then $\Phi(g,Z)$ equals the identity on $M$;
\item 
for any $\ell\in\mathbb{N}$ there holds the estimate
       \be\label{eq:DiffVsIndentity}
      \|\Phi(g,Z) - \id\|_{C^\ell} \leq C(\ell,Q,r_0,\bar g)\|g-\bar g\|_{C^{0}}\,.
    \ee
    \end{enumerate}
     
  \end{lemma}

 \emph{Notation: if $\Sigma$ is an $\MSI$ in $(M,\bar g)$ we will write 
  $\Phi^{\Sigma}(g)$ in place of $\Phi(g,\Sing(\Sigma))$.} 

\
 
  For fixed $\Sigma$, the construction depends smoothly
on $g$ only through its values at the finitely many
singular points of $\Sigma$. Consequently, the map
\begin{equation*}
\Gamma_{\delta}(q) \longrightarrow \operatorname{Diff}(M),
\qquad g \longmapsto \Phi^{\Sigma}(g),
\end{equation*}
is smooth and satisfies
\begin{equation*}
\bigl(\Phi^{\Sigma}(g)\bigr)^{*}g = \bar{g}
\quad \text{on } \operatorname{Sing}(\Sigma).
\end{equation*}

\

   In our setting, as described above, for every $p\in \Sing(\Sigma)$ and $\mbfx\in T_pM$ with $|\mbfx|\leq \delta$, let $\mbfP_\mbfx$ be the map defined as 
   \[
      \mbfP_\mbfx: T_p M \to T_p M, \quad v\mapsto \mbfx+v \,;
    \]
   note that $\mbfP_\mbfx(0)=\mbfx$ and $d\mbfP_\mbfx|_0$ is the identity map.
   Then there exists $\delta(\bar g, \Sigma)>0$ such that, 
\begin{equation}\label{eqn:def_Psi_p_x_A}
    \Psi_{p, \mbfx, \mbfA}(v) := \xi^{(r_0)}(|v|)\mbfP_{\mbfx}(e^\mbfA v) + (1-\xi^{(r_0)}(|v|))v
  \end{equation}
  is a diffeomorphism of $T_pM$ to itself, provided $\mbfx\in T_pM$, $\mbfA\in \mathfrak{so}(T_pM)$ with $|\mbfx|, |\mbfA|\leq \delta$. Indeed, direct differentiation gives
\[
\|d\Psi_{p,\mbfx,\mbfA}-\id\|_{C^0}
\leq C\bigl(r_0^{-1}|\mbfx|+|\mbfA|\bigr)<1
\]
for $\delta$ sufficiently small, where $C$
depends only on the fixed cutoff.
Since $\Psi_{p,\mbfx,\mbfA}$ equals the identity
outside $B_{2r_0}(0)$, the injectivity and
open-and-closed image argument used in the
proof of Lemma \ref{lem:LocalAdjMap} shows
that it is a diffeomorphism.
    
  Now, given $\Sigma$ an $\MSI$ in $(M,\bar g)$ (and adopting the notational conventions described above) for $\mbfx=(\mbfx_1, \dots, \mbfx_Q)\in \scT$ with $|\mbfx_j|\leq \delta$ and $\mbfA=(\mbfA_1, \dots, \mbfA_Q)\in \scR$ with $|\mbfA_j|\leq \delta$, define the diffeomorphism $\Psi_{\mbfx, \mbfA}$ by 
  \begin{equation}\label{eqn:def_Psi_x_A}
     \Psi_{\mbfx, \mbfA}(y):=
     \begin{cases}
\exp^{\bar g}_{p_j}\circ\,\Psi_{p_j, \mbfx_j, \mbfA_j}\circ (\exp^{\bar g}_{p_j})^{-1}(y), & \text{ if }y\in B_{2r_\circ}(p_j) \\
y & \text{ if }y\notin B_{2r_\circ}(p_j) 
\end{cases}
  \end{equation}
Each $\Psi_{p_j,\mbfx_j,\mbfA_j}$ maps
$B_{2r_0}(0)$ onto itself, since it is a
bijection fixing its complement pointwise.
Thus its conjugate by the exponential map
is a diffeomorphism of $B_{2r_0}(p_j)$.
These maps patch smoothly with the identity
to give a diffeomorphism of $M$, because
the balls are pairwise disjoint and the
cutoff is smooth. Moreover, for every
$\ell\in\mathbb{N}$ there holds
\be\label{eq:TranslRotVsIndentity}
\|\Psi_{\mbfx,\mbfA}-\id\|_{C^\ell}
\leq C(\ell,Q,r_0,\bar g)
\bigl(|\mbfx|+|\mbfA|\bigr).
\ee

  \begin{lemma}\label{lem:AdjustCoincideWith}
In the setting above, given $\Sigma$ an $\MSI$ in $(M,\bar g)$, for $\delta=\delta(r_0,Q,\bar g)>0$ small enough the $C^q$ map 
\be
\Xi: \Gamma(q)\oplus \scT \oplus\scR\supset ((\bar g,0,0))_{\delta}\longrightarrow \Diff(M) 
\ee
(where $((\bar g,0,0))_{\delta}$ stands for the metric $\delta$-neighborhood of $(\bar g,0,0)\in \Gamma(q)\oplus \scT \oplus\scR$)
defined by
\be\label{eq:Xi}
\Xi(g,\mbfx,\mbfA)=\Psi_{\mbfx,\mbfA}\circ  \Phi^\Sigma((\Psi_{\mbfx,\mbfA})^\ast g)\,.
\ee
satisfies the following properties:
\begin{enumerate}[label=(\roman*)]
\item\label{item:Xi} there holds \[
\Xi(g,0,0)=\id \ \hspace{2mm} (\forall g\in \Gamma(q;\bar g,\Sigma))\,;
\]
\item\label{item:BoundComposite} for any $\ell\in\mathbb{N}$ there holds the estimate
\be\label{eq:DiffBoundComposite}
      \|\Xi(g,\mbfx,\mbfA) - \id\|_{C^\ell} \leq C(\ell,\bar g,\Sigma)(\|g-\bar g\|_{C^{0}}+|\mbfx|+|\mbfA|)\,.
    \ee

    \item\label{item:XiAst} the map
\[
\Xi^\ast:
\Gamma(q)\oplus\scT\oplus\scR
\supset ((\bar g,0,0))_{\delta}
\longrightarrow \Gamma(q;\bar g,\Sigma),
\qquad
\Xi^\ast(g,\mbfx,\mbfA)
:=
\Xi(g,\mbfx,\mbfA)^\ast g,
\]
is continuous and is a retraction when restricted to
$\Gamma_\delta(q)\oplus0\oplus0
\cong\Gamma_\delta(q)$.
Moreover, for every integer $1\leq s\leq q-1$,
the same map, regarded as taking values in
$\Gamma(q-s;\bar g,\Sigma)$, is of class $C^s$.
\end{enumerate}
\end{lemma}

\emph{Notation: we will from now on systematically write $\Xi_{\mbfx, \mbfA}^{g}$ in place of $\Xi(g,\mbfx,\mbfA)$, and $g_{\mbfx,\mbfA}$ in place of  $(\Xi_{\mbfx, \mbfA}^{g})^*g$. When, in later sections, we wish to stress the dependence on the $\MSI$ in question, we will write  $\Xi_{\mbfx, \mbfA}^{\Sigma,g}$.}

  \begin{proof}

When $\mbfx=0$ and $\mbfA=0$ we note that -- straight from the corresponding definition -- $\Psi_{0,0}$ is the identity map on $M$, as is $\Phi^\Sigma$ as soon as $g\in \Gamma(q;\bar g,\Sigma)$ (this is item \ref{item:Locality} of Lemma \ref{lem:LocalAdjMap}), and so will their composition under our assumptions. This gives \ref{item:Xi}. Estimate \eqref{eq:DiffBoundComposite}, that is \ref{item:BoundComposite}, comes (via the triangle inequality) from \eqref{eq:DiffVsIndentity} and \eqref{eq:TranslRotVsIndentity}.

For \ref{item:XiAst}, write
$F=\Xi(g,\mbfx,\mbfA)$.
In smooth local coordinates, the pullback has components
\[
(F^\ast g)_{ij}(y)
=
g_{ab}(F(y))\,\partial_iF^a(y)\,\partial_jF^b(y).
\]
The continuous dependence of $F$ on its parameters
in every fixed $C^\ell$ topology, together with
continuity of composition and multiplication in
integer $C^j$ norms, proves continuity of $\Xi^\ast$
with values in $C^q$.

More generally, the map
$(g,\mbfx,\mbfA)\mapsto F$ is $C^q$ with values
in every fixed $C^\ell$ topology.
For $1\leq s\leq q-1$, differentiating the displayed
formula $s$ times with respect to the parameters
involves spatial derivatives of $g$ of order at most
$s$. Measuring the resulting expressions in
$C^{q-s}$ therefore requires spatial derivatives
of $g$ of order at most $q$.
The chain and product rules, justified by Taylor's
formula in local coordinates, yield continuous
Fr\'echet derivatives through order $s$ with values
in $C^{q-s}$.
A finite coordinate covering of the compact
manifold $M$ gives the asserted $C^s$ regularity.

By the definition of $\Xi$ and
Lemma \ref{lem:LocalAdjMap}, the metric $F^\ast g$
coincides with $\bar g$ on $\Sing(\Sigma)$.
Thus all the maps above take values in the
indicated affine spaces of metrics.
Finally, the retraction property follows from
\ref{item:Xi}.
  \end{proof}

\begin{definition}\label{def:GenNormGraph}(Generalized normal graph) Given $\Sigma$ an $\MSI$ in $(M,\bar g)$, let $\delta\in (0,1/4)$ be small enough as afforded by Lemma \ref{lem:AdjustCoincideWith}. Let $\Exp_\Sigma:\mbfV\to M$ be the normal exponential map of $\Sigma$ with respect to the metric $\bar g$; thus $\Exp_\Sigma(v)=\exp_x^{\bar g}(v)$ for $v\in\mbfV_x$. For
$k\geq 2, \tau>1$ and
$u\in W^{k,2}_\tau(\Sigma;\mbfV)$  with $\|u\|_{W^{k,2}_{\tau}(\Sigma)}<\delta$ (possibly taking $\delta>0$ smaller, as provided by the tubular neighborhood theorem in this context)
let \[
\Sigma(u):=\left\{\Exp_\Sigma(u(x)) \st x\in\Sigma\right\}\,.
\]

Then, for $(g,\mbfx, \mbfA)\in ((\bar g,0,0))_{\delta}\subset \Gamma(q)\oplus \scT \oplus\scR$ we let
\be\label{eq:GenNormGraph}
 \Sigma(g; u, \mbfx, \mbfA):= \Xi^{\Sigma,g}_{\mbfx, \mbfA}(\Sigma(u))\,.
\ee
 \end{definition}

  \begin{remark}
Tautologically, there holds the following equivalence:
$\Sigma(g; u, \mbfx, \mbfA)$ is minimal in $(M,g)$ if and only if $\Sigma(u)$ is minimal in $(M,g_{\mbfx,\mbfA})$. Note that on the one hand $\Sigma(u)$ is independent of $g$ and on the other hand $g_{\mbfx,\mbfA}$ coincides with $\bar g$ on $\Sing(\Sigma)$. 
  \end{remark}

  The following statement, whose proof is an elementary consequence of the bound \eqref{eq:DiffBoundComposite}, describes the ``velocity vector field" associated to the variations presented above, thereby connecting this discussion with the functional spaces that have been introduced in Section \ref{sec:Refine}.

    \begin{lemma}\label{lem:VelocityField}
     Let $\Sigma$ be an $\MSI$ in $(M,\bar g)$, let $t\in (-1,1)$ and set
     \begin{equation}\label{eq:TranslVelocity}
     \psi_\mbfx:=\left[\frac{d}{dt}\right]_{t=0} \Xi^{\Sigma,\bar g}_{t\mbfx, \orig}\,,
     \end{equation}
     as well as
     \begin{equation}\label{eq:RotVelocity}
     \psi_\mbfA:=\left[\frac{d}{dt}\right]_{t=0} \Xi^{\Sigma,\bar g}_{ \orig, t\mbfA}.
     \end{equation}
     Then, for every $j=1,\ldots, Q$, for every $y\in B_{r_\circ}(p_j)$ there holds 
     \[
     \psi_\mbfx (y)=\mbfx_j, \ \ \ \text{and} \ \ \ \psi_\mbfA (y)=\mbfA_j y \,;
     \]
     furthermore a basis of the vector space $W_{\text{Tr}}(\Sigma;\mbfV)$ (respectively: of $W_{\text{Rot}}(\Sigma;\mbfV)$)  for $\mbfV$ the normal bundle of $\Sigma$ in $(M,\bar g)$ is given by
     \[
     \left\{(\psi_{e^i_j})^{\perp} \ \st \ \left\{e^i_j\right\}_{i=1,\ldots N} \ \text{is an o.n. basis of $T_{p_j} M$} \right\}
     \]
     or, respectively, by
     \[
     \left\{(\psi_{\mbfA^i_j})^{\perp} \ \st \ \left\{\mbfA^i_j\right\}_{i=1,\ldots N} \ \text{is an o.n. basis of $\mathfrak{g}^{\perp}_j$} \right\}\,.
     \]
    \end{lemma}

    \textbf{Notation: link isomorphisms.} With the preceding notation in place, we note that -- associated to any given $\MSI$ in $(M,\bar g)$ -- there is a linear isomorphism
    \be\label{eq:mapP}
    Q_{\Sigma,\bar g}: W^{k,2}_\tau(\Sigma;\mbfV)\oplus \scT\oplus\scR \to \hat W^{k,2}_\tau(\Sigma;\mbfV), \ \ (u,\mbfx,\mbfA)\mapsto u+\psi^{\perp}_{\mbfx}+\psi^{\perp}_\mbfA
    \ee
whose inverse shall from now on be denoted by
$P_{\Sigma,\bar g}$. The isomorphism property
follows from \eqref{eq:Augment} and
Lemma \ref{lem:VelocityField}: injectivity of
the translation parametrization follows from
\cite[Lemma 4.7]{CLW24}, while injectivity of
the rotation parametrization follows from
the restriction to the orthogonal complement
of the stabilizer Lie algebra.

\subsection{Construction of local adjustments}\label{subs:ConstrDiff}

  \begin{lemma}\label{lem:LocalAdj}(Local Adjustment)
    There exists $\delta=\delta(n)\in (0,1/4)$ such that: if $(V, \bar g)$ is a Euclidean space of dimension $n$, $g$ is another (positive-definite) inner product on $V$ such that $\|g-\bar g\|\leq \delta$, then there exists a unique $\bar g$-self-adjoint, positive-definite map $T=T(V, \bar g, g): V\to V$ such that $T^*g =\bar g$ and it varies smoothly in $\bar g, g$. Moreover, there holds:
    \be\label{eq:AdjDiffNormBound}
    \|\id - T\|\leq \|g-\bar g\|.
    \ee
  (In this statement and throughout the following proof all norms $\|\cdot\|$ are understood with respect to the metric $\bar g$, and $\|\id-S\|$ denotes the operator norm with respect to $\bar g$).
  \end{lemma}

  \begin{proof}
 Let $S:V\to V$ be the linear, $\bar g$-self adjoint map defined by requiring
 \[
 g(v_1, v_2)=\bar g(Sv_1, v_2) \ \ \hspace{2mm} (\forall v_1, v_2\in V)\,.
 \]
This definition obviously relies upon Riesz' theorem; note that any such $S$ must always be a linear isomorphism under the assumptions on $g$.
 
 Then it is readily checked that the operator $T=S^{-1/2}$, namely the linear map determined by letting $T\in \text{End}_{\R}(V)$ as composite of the two (analytic) maps defined by the Neumann series
 \[
 S \mapsto \sum_{k=0}^\infty (\id -S)^k =:S', \ \ \ S' \mapsto \sum_{\ell=0}^{\infty}(-1)^{\ell} \binom{1/2}{\ell}(\id - S')^{\ell}
 \]
 that is indeed well-defined for (say) $\|\id -S\|<1/4$
 , satisfies the claimed pullback property.

Since $\id-S$ is $\bar g$-self-adjoint, the
spectral theorem and the defining identity
$g(v,w)=\bar g(Sv,w)$ give
\[
\|\id-S\|
=\sup_{v\in V\setminus\{0\}}
  \frac{\|v-Sv\|}{\|v\|}
=\sup_{v\in V\setminus\{0\}}
  \frac{|\bar g((\id-S)v,v)|}{\|v\|^2}
=\sup_{v\in V\setminus\{0\}}
  \frac{|\bar g(v,v)-g(v,v)|}{\|v\|^2}
\leq \|\bar g-g\|,
\]
where $\|\bar g-g\|$
denotes the Hilbert--Schmidt norm of the
symmetric bilinear form $\bar g-g$.
 In particular, if $\left\{v_1,\ldots, v_n\right\}$ is a $\bar g$-orthonormal basis of $V$ diagonalizing $S$ (whose existence is guaranteed by the spectral theorem) with $Sv_i=\lambda_i v_i$ (thus $\lambda_i\in (3/4,5/4)$ for all $i=1,\ldots, n$) then there holds
 \[
 Tv_i=\frac{1}{\sqrt{\lambda_i}}v_i \ \ (i=1,\ldots n)
 \]
 and it is clear that, on the other hand, this relation must hold for any $\bar g$-self adjoint positive-definite map that pulls back $g$ to $\bar g$. Thus the previously defined map $T$ does indeed enjoy the claimed uniqueness property. 

  Lastly, the bound \eqref{eq:AdjDiffNormBound} can be justified by comparing the eigenvalues of $S$ with those of $S^{-1/2}$, which boils down (under the norm assumptions above) to studying the function
  \[
  f(\lambda)=\frac{1-\sqrt{\lambda}}{\sqrt{\lambda}(1-\lambda)}=\frac{1}{\sqrt{\lambda}(1+\sqrt{\lambda})} \ \ \text{for} \ \lambda\in I:=\left(\frac{3}{4},\frac{5}{4}\right)\,;
  \]
  in particular $\sup_{\lambda\in I}f(\lambda)=\frac{4}{3+2\sqrt{3}}<1$ which implies \eqref{eq:AdjDiffNormBound}.
  \end{proof}

 \begin{proof}[Proof of Lemma \ref{lem:LocalAdjMap}]
Keeping in mind the smallness assumption on $r_0$ and working in $\bar g$-geodesic normal coordinates centered at each given point $p\in Z$, we let, for $x\in B_{3r_0}(p)$,
\[
\Phi(x)=\xi^{(r_0)}(|x|)T_{(p)}x+
(1-\xi^{(r_0)}(|x|))x,
\]
where
$T_{(p)}:=T (T_p M,\bar g_p,g_p)$ is as provided by Lemma \ref{lem:LocalAdj}. We extend this definition by requiring $\Phi$ to be the identity on the complement of
$\bigcup_{p\in Z}B_{2r_0}(p)$; notice that, by the definition of $M(Q,r_0)$, the balls $B_{3r_0}(p)$, for $p\in Z$, are pairwise disjoint.
Thus, for $x\in B_{2r_0}(p)$ and $v\in T_xM$, differentiation gives
\[
(d\Phi_x-\id)[v]
=\xi^{(r_0)}(|x|)(T_{(p)}-\id)v
+({\xi^{(r_0)}})'(|x|)
\bar g\left(\frac{x}{|x|},v\right)
(T_{(p)}-\id)x\,,
\]
and so it follows that
\[
\|d\Phi-\id\|_{C^0(B_{2r_0}(p))}
\leq C(r_0,\bar g)\|T_{(p)}-\id\|.
\]
By Lemma \ref{lem:LocalAdj},
$\|T_{(p)}-\id\|
\leq \|g_p-\bar g_p\|$
and therefore
\[
\|d\Phi-\id\|_{C^0(B_{2r_0}(p))}
\leq C(r_0,\bar g)\|g-\bar g\|_{C^0}.
\]
Upon decreasing $\delta$ if necessary, we may assume that the right-hand side is strictly smaller than $1$.
We claim that $\Phi$ is a diffeomorphism of $M$. To this end, let us write for brevity
$F:=\Phi-\id$. Since each $B_{3r_0}(p)$ is (geodesically) convex and $\|dF\|_{C^0}<1$, for any $x,y$ in that ball the fundamental theorem of calculus implies (via the reverse triangle inequality) \[
|\Phi(x)-\Phi(y)|
\geq
\bigl(1-\|dF\|_{C^0}\bigr)|x-y|.
\]
Thus $\Phi$ is injective on each such ball. Moreover, since $\Phi$ agrees with the identity on the annulus $B_{3r_0}(p)\setminus B_{2r_0}(p)$, injectivity implies that
$\Phi(B_{2r_0}(p))\subset B_{2r_0}(p). $
As the balls $B_{3r_0}(p)$ are pairwise disjoint and $\Phi$ is the identity outside their union, it follows that $\Phi$ is globally injective on $M$.
On the other hand, the estimate $\|d\Phi-\id\|_{C^0}<1$ shows that $d\Phi$ is a linear isomorphism at every point. Hence $\Phi$ is a local diffeomorphism. Since $M$ is compact, $\Phi(M)$ is compact and therefore closed in $M$; being the image of a local diffeomorphism, it is also open. Thus, being $M$ connected by assumption, $\Phi$ is a diffeomorphism.
By construction, $\Phi(p)=p$ for every $p\in Z$, while, since $\xi^{(r_0)}\equiv 1$ in a neighborhood of the origin,
$
d\Phi_p=T_{(p)}.$
Consequently,
$
(\Phi^*g)_p=T_{(p)}^*g_p=\bar g_p,$
which proves that $\Phi^*g$ coincides with $\bar g$ on $Z$.
Moreover, if $g$ already coincides with $\bar g$ on $Z$, then $g_p=\bar g_p$ for every $p\in Z$ and this implies $\Phi(g,Z)=\id$ on $M$.
Finally, for every $\ell\in\mathbb N$, differentiating the defining expression for $\Phi-\id$, and recalling that $r_0$ and the cutoff function are fixed, yields
\[
\|\Phi(g,Z)-\id\|_{C^\ell}
\leq
C(\ell,Q,r_0,\bar g)
\max_{p\in Z}\|T_{(p)}-\id\|.
\]
Recalling the bound given in the statement of Lemma \ref{lem:LocalAdj} we then get
\[
\|\Phi(g,Z)-\id\|_{C^\ell}
\leq
C(\ell,Q,r_0,\bar g)\|g-\bar g\|_{C^0},
\]
which is \eqref{eq:DiffVsIndentity}.
The $C^q$ dependence of $\Phi$ on $(g,Z)$ follows from the smooth dependence of $T_{(p)}$ on $(\bar g_p,g_p)$, together with the $C^q$ dependence of $g_p$ on $p$ (working locally with respect to exponential coordinates in metric $\bar g$). This concludes the proof.
 \end{proof}

\section{The analytic mean curvature map}\label{sec:AnMeanCurv}

In this section, we present a proper, \emph{ad hoc} setup that will be employed towards the development of the local aspects of the deformation theory and eventually towards a proof of Theorem \ref{thm:main}.  The core aspect is the introduction of the notion of \emph{analytic mean curvature map}. 

\subsection{Notation and general preliminaries}
Let $(M,\bar g)$ satisfy the standing assumptions in Section \ref{sec:Prelim} ane let $\Sigma$ denote a fixed  
$\MSI$ therein, with
$\mbfV$ denoting its normal bundle. Unless otherwise
stated, integration and differential operators are
understood with respect to $\bar g$.

We use the weighted Sobolev spaces and their doubly
augmented counterparts introduced in
Section \ref{sec:Refine}. Recall the linear isomorphism
\[
P:=P_{\Sigma,\bar g}:
\hat W^{k,2}_{\tau}(\Sigma;\mbfV)
\longrightarrow
W^{k,2}_{\tau}(\Sigma;\mbfV)
\oplus\scr T\oplus\scr R,
\qquad
P(\hat w)=(u,\mbfx,\mbfA).
\]
Since $\Sigma$ is fixed, we abbreviate
$\Xi^{\Sigma,g}_{\mbfx,\mbfA}$ to
$\Xi^g_{\mbfx,\mbfA}$.

In an attempt of adapting, to the extent possible, White's theory \cite{Whi91} to our setting of $\MSI$, we will henceforth conveniently adopt the following special notation, to be understood for $k\geq 2$ and $\tau\geq 1$:
\begin{align*}
\Gamma &:= \Gamma(q)\,, \\
X &:= \hat W^{k,2}_{\tau}(\Sigma;\mbfV)\,, \\
Y &:= W^{k-2,2}_{\tau-2}(\Sigma;\mbfV)\,,
\end{align*}
and then
\begin{align*}
X_I &:= W^{k,2}_{\tau}(\Sigma;\mbfV) \,.
\end{align*}

In particular, we recall that for this range of decay there holds
\be\label{eq:FinInfSplit}
X=X_{I}\oplus \underbrace{{W_{\text{Tr}}\oplus W_{\text{Rot}}}}_{X_F}
\ee
with the natural projectors $\pi_I:X\to X_I, \pi_F:X\to X_F$,
and the obvious (continuous) dense inclusions 
\[
X\subset Y\subset L^2(\Sigma;\mbfV)=:\cH
\]
based upon the fact that $n\geq 2$ throughout,
which allow to regard $X$ and $Y$ as linear subspaces of the ambient Hilbert space $(\cH,\langle\cdot,\cdot\rangle_{\cH})$.

We assume to have fixed two open connected subsets $\Gamma\supset \Gamma_0\ni \overline{g}$ and $X\supset X_0\ni 0$ (small enough) so that Lemma \ref{lem:AdjustCoincideWith} holds. We 
will write, for brevity, $\Gamma_{\Sigma}:=\Gamma(q; \bar g,\Sigma)$ for the set of $C^q$ metrics which coincide with $\bar g$ on the set $\Sing(\Sigma)$ as per Definition \ref{def:MetricsCoincide}; this is the intersection of a closed affine subspace with $\Gamma$.

\

\textbf{The area functional for normal graphs.} Throughout the paper we let $\check{A}$ denote the ``area functional" on normal graphs over $\Sigma$: more precisely, with the notation described above,
we set
\be\label{eq:AreaFunctNormGraphs}
\check A: (\Gamma_0\cap \Gamma_{\Sigma})\times (X_0\cap X_I) \to \mathbb{R}, \ \ \check A(g,u):=\scH^n(g,(\Sigma(u)))
\ee
where, we recall, $\Sigma(u)$ is the image through the $\bar g$-exponential map of the section $u$ of $\mbfV$. 
That being said, we can proceed towards the definition of \emph{analytic mean curvature map}. A key technical step is the following statement:

\begin{proposition} \label{Prop:C^qMSO}
 Let $\Sigma$ be an $n$-dimensional $\MSI$ in $(M,\bar g)$ under the general standing assumptions listed above, and let
 \be\label{eq:RangeOfParameters}
 n+3\leq k\leq q-1, \ \text{and} \ \tau \in [1, 2)\,.
 \ee
  Possibly by further shrinking the neighborhoods $\Gamma_0$ of $\bar g$, of $X_0$ of the zero section of the normal bundle $\mbfV$, there exists a unique well-defined $C^{q-k}$ map $\mathfrak{\check{h}}: (\Gamma_0\cap \Gamma_\Sigma)\times (X_I\cap X_0) \to Y$ such that \[
    \left[\frac{d}{dt}\right]_{t=0} \check A(g, u+tv) = \langle \mathfrak{\check{h}}(g, u), v\rangle_{\cH}\,,
  \] 
  for every $g\in \Gamma_0\cap \Gamma_\Sigma$, $u\in X_I\cap X_0$ and $v\in X_I$.
\end{proposition}

We stress here that the assumption $g\in\Gamma_{\Sigma}$ is actually \emph{necessary} for the map $\mathfrak{\check{h}}$ to take values in $Y$ as the mean curvature of non-minimal cones blows up with rate $-1$ on approach to the singularity. 

\begin{proof}
This is proven in Appendix \ref{app:C^q}.
\end{proof}

\

\textbf{The area functional for generalized graphs.}
 More generally, we will consider
the area functional on generalized graphs (as defined in Section \ref{subs:GenGraph}), for which we give the definition
\be\label{eq:WhiteAreaFunctional}
A:\Gamma_0\times X_0\subset \Gamma\times X \to \R, \ \ \ A(g,\hat w):=\scH^{n}(g, \Xi^g_{\mbfx,\mbfA}(\Sigma(u))) \ \ \
[P(\hat w)=(u,\mbfx,\mbfA)]
\ee
where 
$\Sigma(g; u,\mbfx,\mbfA):=\Xi^g_{\mbfx,\mbfA}(\Sigma(u))$
denotes a generalized graph as defined in Section \ref{sec:Diffeo}. 

By diffeomorphism invariance of the area functional, $A(g,\cdot): X\to \R$ satisfies the property
\be
\label{eq:Invariance}
    A(g,\hat w)=A(g, Q_{\Sigma,\overline{g}}(u,\mbfx,\mbfA)) 
= \check A((\Xi^g_{\mbfx,\mbfA})^\ast g, u)=\check A(g_{\mbfx,\mbfA},u)
\ee
for $g\in\Gamma_0$ (i.\,e. for all $g\in\Gamma$ sufficiently close to $\overline{g}$). In terms of the well-defined map $\Xi^*$ (cf. Lemma \ref{lem:AdjustCoincideWith}) possibly by taking $\Gamma'_0\subset \Gamma_0$ and $X'_0\subset X_0$ small enough that
\be\label{eq:MetMap}
\Gamma'_0\times (X'_0\cap X_F)\ni (g,(\mbfx,\mbfA)) \ \mapsto^{\Xi^*} \ g_{\mbfx,\mbfA}:=(\Xi^g_{\mbfx,\mbfA})^\ast g \in \Gamma_0\cap\Gamma_{\Sigma}
\ee
we will have, for all $g\in \Gamma'_0, \hat w\in X'_0$
\[
A(g,\hat w)=\check A(\Xi^*(g,Q(\pi_F \hat w)),Q(\pi_I\hat w))\,.
\]

For the sake of readability let us henceforth rename $\Gamma'_0, X'_0$ respectively $\Gamma_0, X_0$, so that $A(g,\hat w)$ satisfies the above identities, and let us further assume that such neighborhoods are small enough that for any $u\in X_0\cap X_I$ the map associating to $v\in X_I$ the normal projection of the variational vector field, $V^{\perp}$, along $\Sigma(u)$ of the family $\Sigma(u+tv)$ at $t=0$, with respect to the metric $g_{\mbfx,\mbfA}$, is invertible.

\begin{lemma}\label{Lem_LS_Full1stVar}
Under the standing assumptions, let
$u\in W^{k,2}_\tau(\Sigma;\mbfV)$ be sufficiently small, and set
\[
\Phi_u(x):=\exp_x^{\bar g}(u(x)).
\]
There exists a bundle map
\[
\Pi_{\Sigma,u}:TM|_{\Sigma(u)}\longrightarrow T^\perp\Sigma
\]
covering $\Phi_u^{-1}$, characterized as follows:
for every $x\in\Sigma$ and $\xi\in T_{\Phi_u(x)}M$,
$\zeta:=\Pi_{\Sigma,u}(\xi)\in T_x^\perp\Sigma$ is the unique
vector satisfying
\[
\xi-d(\exp_x^{\bar g})|_{u(x)}[\zeta]
\in T_{\Phi_u(x)}\Sigma(u).
\]
All norms below are measured with respect to $\bar g$.
The following properties hold:
\begin{enumerate}[label={\normalfont(\roman*)}]
\item\label{Item_LS_|Pi(xi)|<2|xi|}
For every $\xi\in TM|_{\Sigma(u)}$,
\[
|\Pi_{\Sigma,u}(\xi)|\leq2|\xi|.
\]

\item\label{Item_LS_|Pi(xi)-xi^perp|<|g-bar g||xi|}
For every global $C^1$ vector field $\xi$ on $M$,
\[
\left\|
\Pi_{\Sigma,u}(\xi|_{\Sigma(u)})
-(\xi|_\Sigma)^\perp
\right\|_{L^\infty(\Sigma)}
\leq
C\|u\|_{W^{k,2}_\tau}\|\xi\|_{C^1(M)},
\]
where sections along $\Sigma(u)$ are identified with sections
over $\Sigma$ through $\Phi_u$, and we let $C=C(\bar g,\Sigma,k,\tau)$.

\item\label{Item_LS_Full-1stVar}
Let $g\in\Gamma_0\cap\Gamma_\Sigma$, and let
$\{\varphi_t:\Sigma\to M\}_{|t|<\epsilon}$ be a family of
$C^1$ embeddings, depending $C^1$ on $t$ in the
$C^1_{\mathrm{loc}}$ topology, with $\varphi_0=\Phi_u$.
Assume that there exist constants $c_0,C_0>0$ such that,
uniformly for $|t|<\epsilon$, $x\in\Sigma$, and $a\in T_x\Sigma$,
\[
c_0|a|\leq|d\varphi_t(x)[a]|\leq C_0|a|,
\qquad
|\dot\varphi_t(x)|
+\rho(x)|\bar\nabla\dot\varphi_t(x)|
\leq C_0.
\]
Here $\rho$ is the fixed radius function, and
$\bar\nabla\dot\varphi_t$ is the covariant derivative along
$\Sigma$ of the section $\dot\varphi_t$ of $\varphi_t^*TM$.
Then
\[
\left.\frac{d}{dt}\right|_{t=0}
\scH^n(g,\varphi_t(\Sigma))
=
\left\langle
\check{\mathfrak h}(g,u),
\Pi_{\Sigma,u}(\dot\varphi_0)
\right\rangle_{\cH}.
\]
\end{enumerate}
\end{lemma}

\begin{proof}
All norms and orthogonal projections in the construction of
$\Pi_{\Sigma,u}$ are taken with respect to $\bar g$. Also, throughout this proof, $C$ denotes a positive constant depending
only on the fixed data $(M,\bar g,\Sigma,k,\tau)$ and the chosen
radius function, and may change from line to line; in the proof of \ref{Item_LS_Full-1stVar}, it may additionally
depend on the admissibility constants $c_0,C_0$ and uniform
bounds for $g$ in $C^1$ and $g^{-1}$ in $C^0$. After shrinking $\Gamma_0$, the required bounds
for $g$ in $C^1$ and $g^{-1}$ in $C^0$ are
uniform for $g\in\Gamma_0$.
Write
\[
\Phi_u(x):=\exp_x^{\bar g}(u(x)),
\qquad
B_u(x):=d(\exp_x^{\bar g})|_{u(x)}
             \big|_{T_x^\perp\Sigma}.
\]
Let $P_u(x):T_xM\to T_{\Phi_u(x)}M$ denote parallel transport
along $s\mapsto\exp_x^{\bar g}(su(x))$.

Set $\delta:=\|u\|_{W^{k,2}_\tau}$.
The weighted Sobolev embedding and the assumption $\tau\geq1$
give
\[
\sup_\Sigma\left(\rho^{-1}|u|+|\bar\nabla^\perp u|\right)
\leq C\delta.
\]
Moreover, the geometry of the strongly isolated singularities
gives $|\mathrm{II}_\Sigma|\leq C\rho^{-1}$.
Using the Weingarten formula to differentiate $u$ as an ambient
vector field along $\Sigma$, and then differentiating the
exponential map, we obtain
\[
\begin{aligned}
\bigl|P_u(x)^{-1}d\Phi_u(x)[a]-a\bigr|
&\leq C\delta |a|,
&& a\in T_x\Sigma,\\
\bigl|P_u(x)^{-1}B_u(x)[\zeta]-\zeta\bigr|
&\leq C\delta |\zeta|,
&& \zeta\in T_x^\perp\Sigma.
\end{aligned}
\]
The constants are uniform in $x$: near the singularities this
follows from the preceding weighted estimates, and away from
them it follows by compactness.

Consider the linear map
\[
\mathcal T_u(x):T_x\Sigma\oplus T_x^\perp\Sigma\longrightarrow T_xM,
\qquad
\mathcal T_u(x)(a,\zeta)
=
P_u(x)^{-1}\bigl(d\Phi_u(x)[a]+B_u(x)[\zeta]\bigr).
\]
Identifying $T_x\Sigma\oplus T_x^\perp\Sigma$ with $T_xM$,
the preceding estimates imply
\[
\|\mathcal T_u(x)-\mathrm{id}\|\leq C\delta.
\]
Thus, for sufficiently small $\delta$, these maps are
invertible, uniformly in $x$, and
\[
\|\mathcal T_u(x)^{-1}\|\leq2,
\qquad
\|\mathcal T_u(x)^{-1}-\mathrm{id}\|\leq C\delta.
\]
In particular,
\[
T_{\Phi_u(x)}M
=
T_{\Phi_u(x)}\Sigma(u)\oplus B_u(x)(T_x^\perp\Sigma).
\]
This proves existence and uniqueness in the definition of
$\Pi_{\Sigma,u}$. More explicitly, if $\pi_x^\perp$ is the
$\bar g$-orthogonal projection onto $T_x^\perp\Sigma$, then
\[
\Pi_{\Sigma,u}(\xi)
=
\pi_x^\perp\mathcal T_u(x)^{-1}P_u(x)^{-1}\xi,
\qquad \xi\in T_{\Phi_u(x)}M.
\]
Since parallel transport is an isometry, this formula proves
\ref{Item_LS_|Pi(xi)|<2|xi|}. It also gives
\begin{equation}\label{Equ_LS_Pi-transport-estimate}
\left|
\Pi_{\Sigma,u}(\xi)
-\pi_x^\perp P_u(x)^{-1}\xi
\right|
\leq C\delta|\xi|.
\end{equation}

If now $\xi$ is a global $C^1$ vector field on $M$, integration
of its covariant derivative along the same geodesic yields
\[
\left|
P_u(x)^{-1}\xi(\Phi_u(x))-\xi(x)
\right|
\leq |u(x)|\,\|\bar\nabla\xi\|_{L^\infty(M)}.
\]
Combining this with \eqref{Equ_LS_Pi-transport-estimate} and
$\|u\|_{L^\infty}\leq C\delta$, we obtain
\[
\left\|
\Pi_{\Sigma,u}(\xi|_{\Sigma(u)})
-(\xi|_\Sigma)^\perp
\right\|_{L^\infty(\Sigma)}
\leq C\delta\|\xi\|_{C^1(M)},
\]
which proves \ref{Item_LS_|Pi(xi)-xi^perp|<|g-bar g||xi|}.

We turn to \ref{Item_LS_Full-1stVar}, with $g$ in the domain
of $\check{\mathfrak h}$ and with the admissible variations
specified in the statement.
Let $d\mu=d\mathscr H^n_{\bar g}|_\Sigma$, and define $j_u>0$ by
\[
\Phi_u^*(d\mathscr H^n_g|_{\Sigma(u)})=j_u\,d\mu.
\]
For every compactly supported smooth normal section $v$ on
$\Sigma$, the variation $\Phi_{u+tv}$ has velocity $B_uv$.
The defining property of $\check{\mathfrak h}$ and the
classical first-variation formula therefore give
\[
\int_\Sigma
\langle\check{\mathfrak h}(g,u),v\rangle_{\bar g}\,d\mu
=
-\int_\Sigma
g\bigl(H_{\Sigma(u),g}\circ\Phi_u,B_uv\bigr)\,j_u\,d\mu.
\]
By testing with arbitrary such sections, we obtain the
pointwise identity
\begin{equation}\label{Equ_LS_h-pointwise}
\langle\check{\mathfrak h}(g,u)(x),\zeta\rangle_{\bar g}
=
-j_u(x)\,
g\bigl(H_{\Sigma(u),g}(\Phi_u(x)),B_u(x)\zeta\bigr)
\end{equation}
for $\zeta\in T_x^\perp\Sigma$.

Write $V(\Phi_u(x)):=\dot\varphi_0(x)$ and
\[
\eta(x):=\Pi_{\Sigma,u}\bigl(V(\Phi_u(x))\bigr).
\]
By definition of $\Pi_{\Sigma,u}$,
\[
V(\Phi_u(x))-B_u(x)\eta(x)
\in T_{\Phi_u(x)}\Sigma(u).
\]
Since the mean curvature vector is $g$-normal to $\Sigma(u)$,
\eqref{Equ_LS_h-pointwise} implies
\begin{equation}\label{Equ_LS_h-velocity}
-j_u\,g(H_{\Sigma(u),g}\circ\Phi_u,V\circ\Phi_u)
=
\langle\check{\mathfrak h}(g,u),\eta\rangle_{\bar g}.
\end{equation}
The velocity is bounded by admissibility, so part
\ref{Item_LS_|Pi(xi)|<2|xi|} implies $\eta\in L^\infty\cap\cH$.
Since $\check{\mathfrak h}(g,u)\in Y\subset\cH$, the
right-hand side is integrable.

We justify the first-variation formula across the isolated
singularities. Choose cutoffs $\chi_r$ on $\Sigma$ that vanish
where $\rho<r$, equal one where $\rho>2r$, and satisfy
$|\bar\nabla\chi_r|\leq C/r$.
The conical geometry and the uniform graph estimates give
\[
\mathscr H^n_g\bigl(\Phi_u(\{\rho<2r\})\bigr)\leq Cr^n,
\qquad
\left|\nabla_{\Sigma(u)}(\chi_r\circ\Phi_u^{-1})\right|
\leq C/r.
\]
Integration by parts on the regular part, with this compactly
supported cutoff, gives
\[
\begin{aligned}
\int_{\Sigma(u)}(\chi_r\circ\Phi_u^{-1})
                 \operatorname{div}_{\Sigma(u),g}V\,d\mu_g
={}&-\int_{\Sigma(u)}(\chi_r\circ\Phi_u^{-1})
                 g(H_{\Sigma(u),g},V)\,d\mu_g\\
&-\int_{\Sigma(u)}
g\bigl(\nabla_{\Sigma(u),g}(\chi_r\circ\Phi_u^{-1}),
       V^\top\bigr)\,d\mu_g.
\end{aligned}
\]
The absolute value of the last integral is at most
$C\|V\|_{L^\infty}r^{n-1}$ and hence tends to zero, since
$n\geq2$.
The admissibility bounds imply
$|\operatorname{div}_{\Sigma(u),g}V|\leq C\rho^{-1}$
after pullback to $\Sigma$.
This function is integrable by virtue of the same volume estimate.
Together with \eqref{Equ_LS_h-velocity}, dominated convergence
therefore yields
\[
\int_{\Sigma(u)}\operatorname{div}_{\Sigma(u),g}V\,d\mu_g
=
-\int_{\Sigma(u)}g(H_{\Sigma(u),g},V)\,d\mu_g.
\]

Finally, define the positive density $j_t$ by
\[
d\mu_{\varphi_t^*g}=j_t\,d\mu.
\]
Since $\varphi_t$ is an embedding,
\[
\scH^n(g,\varphi_t(\Sigma))=\int_\Sigma j_t\,d\mu.
\]
The uniform bounds on $d\varphi_t$ and its inverse on tangent
spaces, together with the velocity bounds, give
\[
|\partial_tj_t|\leq C\rho^{-1}.
\]

Thus differentiation under the integral is justified, and
the usual derivative of the area density gives
\[
\begin{aligned}
\left.\frac{d}{dt}\right|_{t=0}
\mathscr H^n(g,\varphi_t(\Sigma))
&=\int_{\Sigma(u)}
      \operatorname{div}_{\Sigma(u),g}V\,d\mu_g\\
&=-\int_{\Sigma(u)}g(H_{\Sigma(u),g},V)\,d\mu_g\\
&=\langle\check{\mathfrak h}(g,u),
         \Pi_{\Sigma,u}(\dot\varphi_0)\rangle_{\cH},
\end{aligned}
\]
as required.
\end{proof}

\begin{remark}\label{rem:AreaRegularity}
We now record the differentiability of the area functional.
For every integer $r\geq1$ such that $q\geq k+2r$,
the map
\[
A:\Gamma_0\times X_0\longrightarrow\mathbb R
\]
is of class $C^r$.

To see this, recall that
\[
A(g,\hat w)=\check A(g_{\mbfx,\mbfA},u),
\qquad P(\hat w)=(u,\mbfx,\mbfA).
\]
By Lemma \ref{lem:AdjustCoincideWith}\,\ref{item:XiAst},
with $s=r$, the map
$(g,\mbfx,\mbfA)\mapsto g_{\mbfx,\mbfA}$
is $C^r$ when its values are endowed with the
$C^{q-r}$ topology. This intermediate topology
accounts for the loss of derivatives under pullback.

Since $q-r\geq k+r$, the coordinate expression
for the area integrand, weighted Sobolev embedding
and finiteness of the area of $\Sigma$ justify
differentiation under the integral through order $r$.
 Thus $\check A$
is jointly $C^r$ with its metric variable endowed
with the $C^{q-r}$ topology, and the chain rule
gives the asserted regularity of $A$.
\end{remark}

\begin{corollary}[Analytic mean curvature and critical points]
\label{cor:CharCritPoints}
Under the standing assumptions, let $r\geq1$ be an integer
and assume $q\geq k+2r$.
Then there exists a unique $C^r$ map
\[
\mathfrak h:\Gamma_0\times X_0\to Y
\]
such that
\begin{equation}\label{Def_H_G_new}
D_2A(g,\hat w)[v]
=
\langle\mathfrak h(g,\hat w),v\rangle_{\cH}
\qquad(v\in X_I).
\end{equation}
Writing $P(\hat w)=(u,\mbfx,\mbfA)$, this map is given by
\begin{equation}\label{eq:AnalMeanCurv}
\mathfrak h(g,\hat w)
=
\check{\mathfrak h}(g_{\mbfx,\mbfA},u).
\end{equation}
Moreover, the following conditions are equivalent:
\begin{enumerate}[label=(\roman*)]
\item\label{Item_A1stvariation_H_G=0}
$\mathfrak h(g,\hat w)=0$;
\item\label{Item_A1stvariation_CritPt}
$\hat w$ is a critical point of $A(g,\cdot)$.
\end{enumerate}
If $\tau>1$, these conditions imply that the generalized graph
\[
\Sigma(g;u,\mbfx,\mbfA)
=
\Xi^g_{\mbfx,\mbfA}(\Sigma(u))
\]
is an $\MSI$ in $(M,g)$.
\end{corollary}

\begin{proof}
By \eqref{eq:Invariance} and Proposition \ref{Prop:C^qMSO},
for every $v\in X_I$,
\[
D_2A(g,\hat w)[v]
=
D_2\check A(g_{\mbfx,\mbfA},u)[v]
=
\langle\check{\mathfrak h}(g_{\mbfx,\mbfA},u),v\rangle_{\cH},
\qquad P(\hat w)=(u,\mbfx,\mbfA).
\]
Thus \eqref{eq:AnalMeanCurv} defines a map with the required
first-variation identity, and uniqueness follows from the
density of $X_I$ in $\cH$.

To justify its regularity, apply
Lemma \ref{lem:AdjustCoincideWith}\,\ref{item:XiAst}
with $s=r$.
After shrinking the neighborhoods if necessary, the map
\[
\Gamma_0\times(X_0\cap X_F)
\ni (g,Q(0,\mbfx,\mbfA))
\longmapsto g_{\mbfx,\mbfA}
\in \Gamma(q-r;\bar g,\Sigma)
\]
is of class $C^r$.
Applying Proposition \ref{Prop:C^qMSO} with $q-r$
in place of $q$ gives a $C^{q-r-k}$ normal-graph
analytic mean curvature map on a neighborhood of
$(\bar g,0)$ at this lower metric regularity.
By the first-variation characterization and uniqueness,
this map agrees with $\check{\mathfrak h}$ wherever
both are defined.
Since $q-r-k\geq r$, the chain rule applied to
\eqref{eq:AnalMeanCurv} proves that
\[
\mathfrak h\in C^r(\Gamma_0\times X_0;Y).
\]

To prove
\ref{Item_A1stvariation_H_G=0}
$\Rightarrow$
\ref{Item_A1stvariation_CritPt},
fix $\hat z=Q(v,\mbfy,\mbfB)\in X$ and consider
\[
\varphi_t
=
(\Xi^g_{\mbfx,\mbfA})^{-1}
\circ\Xi^g_{\mbfx+t\mbfy,\mbfA+t\mbfB}
\circ\Phi_{u+tv}.
\]
Then $\varphi_0=\Phi_u$ and
\[
A(g,\hat w+t\hat z)
=
\scH^n(g_{\mbfx,\mbfA},\varphi_t(\Sigma)).
\]
The weighted Sobolev estimates for $u,v$ and the $C^1$
dependence of $\Xi$ on its finite-dimensional parameters
give the admissibility bounds required by
Lemma \ref{Lem_LS_Full1stVar}, for sufficiently small $|t|$.
Applying part \ref{Item_LS_Full-1stVar} of that lemma yields
\[
D_2A(g,\hat w)[\hat z]
=
\left\langle
\mathfrak h(g,\hat w),
\Pi_{\Sigma,u}(\dot\varphi_0)
\right\rangle_{\cH}
=0.
\]
Conversely,
\ref{Item_A1stvariation_CritPt}
$\Rightarrow$
\ref{Item_A1stvariation_H_G=0}
follows by restricting \eqref{Def_H_G_new} to $v\in X_I$
and using its density in $\cH$.

Finally, suppose $\tau>1$ and the equivalent conditions hold.
The pointwise identity \eqref{Equ_LS_h-pointwise},
applied with the metric $g_{\mbfx,\mbfA}$, implies that
$H_{\Sigma(u),g_{\mbfx,\mbfA}}=0$.
Indeed, the area density is positive, and the vectors
$d(\exp_x^{\bar g})|_{u(x)}[\zeta]$,
$\zeta\in T_x^\perp\Sigma$, span a complement to
$T_{\Phi_u(x)}\Sigma(u)$.
A $g_{\mbfx,\mbfA}$-normal vector orthogonal to this
complement must therefore vanish.

The weighted Sobolev embedding gives $u=o(\rho)$ and
$\bar\nabla^\perp u=o(1)$ near the singular set.
Consequently, $\Sigma(u)$ has the same singular points and
the same multiplicity-one regular tangent cones as $\Sigma$.
Since $g_{\mbfx,\mbfA}$ coincides with $\bar g$ at these
points, these cones are also minimal for the corresponding
tangent metrics.
The vanishing of $H_{\Sigma(u),g_{\mbfx,\mbfA}}$, together
with the cutoff argument in Lemma \ref{Lem_LS_Full1stVar},
shows that the associated varifold is stationary across
the isolated singularities. Hence $\Sigma(u)$ is an $\MSI$
in $(M,g_{\mbfx,\mbfA})$.
The final assertion follows by applying the isometry
$
\Xi^g_{\mbfx,\mbfA}:
(M,g_{\mbfx,\mbfA})\longrightarrow(M,g).$
\end{proof}

For the remainder of this section, fix an integer
$r\geq1$ and assume $q\geq k+2r$, as in
Corollary \ref{cor:CharCritPoints}.
In particular, $\mathfrak h$ is of class $C^r$,
and we set
\begin{align}
J(g,\hat w)
:=
D_2\mathfrak h(g,\hat w):X\longrightarrow Y.
\label{Equ_Def_J(g,hat w)}
\end{align}
The map $(g,\hat w)\mapsto J(g,\hat w)$ is of class
$C^{r-1}$ with values in $\mathcal L(X,Y)$.

\begin{lemma}\label{lem:IdentifJacobi operator}When $g=\bar g$ and $\hat w=0$, the operator $J$
equals the negative of the geometric Jacobi operator
of $\Sigma$ with respect to $\bar g$, namely
\[
J(\bar g,0)=-L_{\Sigma,\bar g}:X\to Y.
\]

  In particular, possibly by further shrinking the open sets $X_0$ and $\Gamma_0$, under the hypotheses of Proposition \ref{thm:CountRefined} (in particular, for $\tau\in I_{+}(\Sigma))$ the operator $J(g,\hat w)$ is a Fredholm operator of index zero as soon as $g\in\Gamma_0, \hat w\in X_0$, whenever $\Sigma$ is an $\MSI$ with vanishing effective Morse index and effective nullity.
\end{lemma}

\begin{proof}
Write $J_0:=J(\bar g,0)$ and $L:=L_{\Sigma,\bar g}$.
Fix $\hat w=Q(u,\mbfx,\mbfA)\in X$ and
$v\in C_c^\infty(\Sigma;\mbfV)\subset X_I$.
By the definition of the analytic mean curvature map,
\[
\langle J_0\hat w,v\rangle_{\cH}
=
\left.\frac{d}{ds}\right|_{s=0}
\left.\frac{d}{dt}\right|_{t=0}
A(\bar g,s\hat w+tv).
\]
To compute this mixed variation, consider, on
$\Sigma$, the family
\[
F_{s,t}(p)
:=
\Xi^{\bar g}_{s\mbfx,s\mbfA}
\bigl(\Exp_\Sigma(su(p)+tv(p))\bigr).
\]
Its image is the generalized graph defining
$A(\bar g,s\hat w+tv)$, and its variational vector fields
at $(s,t)=(0,0)$ satisfy
\[
F_{0,0}=\id_\Sigma,\qquad
\left(\left.\partial_sF_{s,0}\right|_{s=0}\right)^{\perp}
=
u+\psi_{\mbfx}^{\perp}
+\psi_{\mbfA}^{\perp}
=\hat w,\qquad
\left.\partial_tF_{0,t}\right|_{t=0}=v\,,
\]
where $\perp=\perp_{\bar g}$.
For each (sufficiently small) fixed $s$, the $t$-variation is supported away from $\Sing(\Sigma)$. We may therefore
compute the mixed variation on a relatively compact
smooth domain containing $\operatorname{supp}v$.

Since $\Sigma$ is minimal, the mixed second variation
formula, with our convention for the geometric
Jacobi operator, gives
\[
\left.\frac{d}{ds}\right|_{s=0}
\left.\frac{d}{dt}\right|_{t=0}
A(\bar g,s\hat w+tv)
=
-\langle L\hat w,v\rangle_{\cH}.
\]
No boundary term arises because the $t$-variation
is compactly supported away from $\Sing(\Sigma)$.

Consequently,
\[
\langle J_0\hat w+L\hat w,v\rangle_{\cH}=0
\qquad
\text{for every }
v\in C_c^\infty(\Sigma;\mbfV).
\]
Both $J_0\hat w$ and $L\hat w$ belong to
$Y\subset\cH$.
Density of the compactly supported test sections
in $\cH$ therefore implies $J_0\hat w=-L\hat w$.
Since $\hat w\in X$ was arbitrary, this proves
\[
J(\bar g,0)=-L_{\Sigma,\bar g}:X\longrightarrow Y.
\]

Under the hypotheses of Proposition \ref{thm:CountRefined},
including the stipulated range of $\tau$ and the
vanishing of the effective index and effective nullity,
the operator $L_{\Sigma,\bar g}:X\to Y$ is Fredholm
of index zero. Hence $J_0=-L_{\Sigma,\bar g}$ is also Fredholm
of index zero. Moreover, Corollary \ref{cor:CharCritPoints} and
$r\geq1$ imply that
\[
(g,\hat w)\longmapsto
J(g,\hat w)=D_2\mathfrak h(g,\hat w)
\]
is continuous with values in $\mathcal L(X,Y)$.
The openness of the Fredholm set and local constancy
of its index therefore imply, after shrinking
$\Gamma_0$ and $X_0$, that every $J(g,\hat w)$
is Fredholm of index zero.
\end{proof}

\begin{lemma}\label{lem:J_SelfAdjoint}
In the setting above, $J_0:=J(\bar g,0)$ is symmetric
with respect to the $\cH$ pairing on its domain $X$, namely there holds 
\[
\langle J_0 \hat u,\hat v \rangle_{\cH}=\langle \hat u, J_0\hat v\rangle_{\cH}, \ \ (\forall \hat u, \hat v\in X)\,.
\]
\end{lemma}

\begin{proof}
Let $f:=A(\bar g,\cdot):X_0\to\R$.
For $\hat w=Q(u,\mbfx,\mbfA)\in X_0$ and
$\hat z=Q(v,\mbfy,\mbfB)\in X$, write
$\Psi:=\Xi^{\bar g}_{\mbfx,\mbfA}$,
$g_{\hat w}:=\Psi^*\bar g$ and
$\Phi_u(x):=\exp_x^{\bar g}(u(x))$.
Consider the variation
\[
\varphi_t
:=
\Psi^{-1}\circ
\Xi^{\bar g}_{\mbfx+t\mbfy,\mbfA+t\mbfB}
\circ\Phi_{u+tv},
\]
so that $\varphi_0=\Phi_u$ and
$f(\hat w+t\hat z)
=\scH^n(g_{\hat w},\varphi_t(\Sigma))$.
For sufficiently small $|t|$, weighted Sobolev embedding
and the smooth dependence of $\Xi^{\bar g}$ on its
parameters give uniform bounds
\[
C^{-1}|a|\leq|d\varphi_t(x)[a]|\leq C|a|,
\qquad
|\dot\varphi_t(x)|
+\rho(x)|\bar\nabla\dot\varphi_t(x)|
\leq C\|\hat z\|_X.
\]
Thus Lemma \ref{Lem_LS_Full1stVar} applies and yields
\begin{equation}\label{eq:FullVariationOperator}
Df(\hat w)[\hat z]
=
\left\langle
\mathfrak h(\bar g,\hat w),
\Pi_{\Sigma,u}(\dot\varphi_0)
\right\rangle_{\cH},
\end{equation}
where we used
$\check{\mathfrak h}(g_{\hat w},u)=\mathfrak h(\bar g,\hat w)$. We compare the section in this pairing with $\hat z$. Define the ambient vector field
\[
V_{\mbfx,\mbfA;\mbfy,\mbfB}
:=
\left.\frac{d}{dt}\right|_{t=0}
\left(
(\Xi^{\bar g}_{\mbfx,\mbfA})^{-1}
\circ
\Xi^{\bar g}_{\mbfx+t\mbfy,\mbfA+t\mbfB}
\right).
\]
Differentiating $\varphi_t$ and using the defining property
of $\Pi_{\Sigma,u}$ gives
\[
\Pi_{\Sigma,u}(\dot\varphi_0)
=
v+\Pi_{\Sigma,u}
\bigl(V_{\mbfx,\mbfA;\mbfy,\mbfB}|_{\Sigma(u)}\bigr).
\]
Smooth dependence on the parameters implies
\[
\begin{aligned}
\|V_{\mbfx,\mbfA;\mbfy,\mbfB}\|_{C^1(M)}
&\leq C(|\mbfy|+|\mbfB|),\\
\|V_{\mbfx,\mbfA;\mbfy,\mbfB}
      -V_{0,0;\mbfy,\mbfB}\|_{C^0(M)}
&\leq C(|\mbfx|+|\mbfA|)(|\mbfy|+|\mbfB|).
\end{aligned}
\]
By Lemma \ref{lem:VelocityField} and the definition of $Q$,
$(V_{0,0;\mbfy,\mbfB}|_\Sigma)^\perp=Q(0,\mbfy,\mbfB)$.
Consequently, Lemma \ref{Lem_LS_Full1stVar}\,
\ref{Item_LS_|Pi(xi)-xi^perp|<|g-bar g||xi|}
and the boundedness of $P=Q^{-1}$ yield
\begin{equation}\label{eq:FullVariationVelocityEstimate}
\left\|
\Pi_{\Sigma,u}(\dot\varphi_0)-\hat z
\right\|_{\cH}
\leq C\|\hat w\|_X\|\hat z\|_X.
\end{equation}

Since $\Sigma$ is minimal,
$\mathfrak h(\bar g,0)=0$ and $Df(0)=0$.
Differentiability of $\mathfrak h$ and the continuous
inclusion $Y\hookrightarrow\cH$ give
\[
\mathfrak h(\bar g,\hat w)
=
J_0\hat w+r(\hat w),
\qquad
\|r(\hat w)\|_{\cH}=o(\|\hat w\|_X).
\]
Define $\mathcal B\in\mathcal L(X,X^*)$ by
$(\mathcal B\hat u)[\hat v]
:=\langle J_0\hat u,\hat v\rangle_{\cH}$.
Using \eqref{eq:FullVariationOperator},
\eqref{eq:FullVariationVelocityEstimate} and
$X\hookrightarrow\cH$, we obtain
\[
\|Df(\hat w)-Df(0)-\mathcal B\hat w\|_{X^*}
\leq
C\|r(\hat w)\|_{\cH}
+
C\|\mathfrak h(\bar g,\hat w)\|_{\cH}\|\hat w\|_X
=o(\|\hat w\|_X).
\]
Hence $Df$ is Fr\'echet differentiable at $0$, with
\[
D^2f(0)[\hat u,\hat v]
=
\langle J_0\hat u,\hat v\rangle_{\cH}
\qquad(\hat u,\hat v\in X).
\]
Symmetry of the second Fr\'echet derivative proves the claim.
\end{proof}

\subsection{The constraint theorem - linear local version}\label{subs:Constraint}

In the setting above, we denote the kernel of $J=J(g,\hat w)$ by
\be\label{eq:Kernel}
  K(g,\hat w)\subset X
\ee
which we may naturally regard (also) as a finite-dimensional (hence: closed) subspace of $\cH$. Denote for simplicity $K=K(\bar g, 0)$ and $K^{\perp}\subset \cH$ be the corresponding orthogonal complement therein; further, let $c:=\dim_{\R}(K)$. We will denote the associated orthogonal projectors, respectively by $\pi: \cH\to K$ and $\pi^{\perp}:\cH\to K^{\perp}$. In particular, in view of the inclusions (cf. Remark \ref{rem:Inclusions})
\[
  X\subset Y\subset \cH
\]
such operators may be restricted to the smaller domains $X$ and $Y$ when appropriate. In what follows, we will tacitly abuse notation and write e.\,g. $\pi^{\perp}:Y\to K^{\perp}$ (a posteriori: $\pi^{\perp}:Y\to Y\cap K^{\perp}$, see below) in place of $\pi^{\perp}_{|Y}$.

\begin{lemma}\label{lem:LocSurj_new}
  Under the standing assumptions, with $r\geq1$ and
$q\geq k+2r$, suppose that $J(\bar g,0)$ is Fredholm
of index zero.
Then, at $\hat w=0$, the map
$\pi^\perp\circ\mathfrak h(\bar g,\cdot):
X\to Y\cap K^\perp$
has surjective differential
\[
  \pi^{\perp}\circ J(\bar g,0):
  X\to Y\cap K^{\perp}.
\]
  
  Hence, there exist open neighborhoods $X_0$ (resp. $K_0$) of $0\in X$ (resp. in $K$), $\Gamma_0$ of $\bar g \in\Gamma$ and
  a $C^{r}$ map $\varphi: \Gamma_0\times K_0\to X\cap K^{\perp}$ such that for any $g\in \Gamma_0$ and $\kappa\in K_0$ the section $\mathfrak{h}(g,\kappa+\varphi(g,\kappa))\in Y$ lies in the subspace $Y\cap K$;  in particular, for every $g\in\Gamma_0$, this gives a $c$-dimensional differentiable open manifold of solutions to the problem $\pi^{\perp}\circ \mathfrak{h}=0$. Conversely, if $g\in\Gamma_0$ and $\hat w\in X_0$
satisfy $\mathfrak h(g,\hat w)\in K$, then there
exists a unique $\kappa\in K_0$ such that
$\hat w=\kappa+\varphi(g,\kappa)$.

  There holds $\varphi(\bar g, 0)=0$ and $D_2 \varphi(\bar g, 0)\kappa=0$ (for any $\kappa\in K$) and so, as a result,
  \be\label{eq:Taylor1SelectFunction}
    \varphi(g,\kappa)=O(\|g-\bar g\|_{\Gamma})+o(\|\kappa\|_X),  \ \ \ (g\in\Gamma_0, \kappa\in K_0)\,. 
  \ee
\end{lemma}

\begin{proof}
Since $K\subset X\subset Y\subset\cH$ and $K$
is finite-dimensional, the $\cH$-orthogonal
projection onto $K$ restricts to bounded
projections on $X$ and $Y$. Hence
\[
X=K\oplus(X\cap K^\perp),
\qquad
Y=K\oplus(Y\cap K^\perp),
\]
with closed summands. In particular,
$\pi^\perp\circ\mathfrak h$ takes values in
$Y\cap K^\perp$.

From Lemma \ref{lem:J_SelfAdjoint}, there holds for $J_0:=J(\bar g,0)$ 
\[
J_0 X\subset Y\cap K^{\perp}\,.
\]
by directly applying the assumption that $J_0$ has Fredholm index 0, since the closed subspace $Y\cap K^{\perp}$ has codimension $c$ in $Y$, we get at once that the previous inclusion is actually an equality (that is: $J_0X= Y\cap K^{\perp}$) and so $\pi^{\perp}\circ J_0: X\to Y\cap K^{\perp}$ is surjective, as claimed.

By Corollary \ref{cor:CharCritPoints}, the map
\[
(g,\kappa,\psi)\longmapsto
\pi^\perp\mathfrak h(g,\kappa+\psi),
\qquad
\kappa\in K,\quad \psi\in X\cap K^\perp,
\]
is of class $C^r$ near $(\bar g,0,0)$.
Its differential in $\psi$ at that point is the
bounded linear isomorphism
\[
J_0|_{X\cap K^\perp}:
X\cap K^\perp\longrightarrow Y\cap K^\perp.
\]
The implicit function theorem therefore gives the
claimed $C^r$ map $\varphi$ and its local uniqueness.
Moreover, $\varphi(\bar g,0)=0$, and differentiation
in $\kappa$ at $(\bar g,0)$ gives
$D_2\varphi(\bar g,0)=0$, since $J_0\kappa=0$
for $\kappa\in K$.
The stated estimate follows from these identities
and the $C^1$ regularity of $\varphi$.
\end{proof}

\subsection{Local moduli and persistence}\label{subs:LocStruct}
Next, we shall describe the local moduli space of $\MSI$ as the ambient metric varies and derive persistence results under suitable restrictions on the metric (covering, in particular, the ``unobstructed" case).

\begin{proposition}[cf. Theorem 1.2 in \cite{Whi91}]
\label{pro:SurjectionCriterion}
Under the standing assumptions, with $r\geq1$ and
$q\geq k+2r$, suppose that $J(\bar g,0)$ is Fredholm
of index $-d$, where $d\geq0$.
After possibly shrinking the neighborhoods
$\Gamma_0$ of $\bar g$ and $X_0$ of $0$, consider
\be\label{eq:LocModuliSpace}
Z:=\left\{(g,\hat w)\in\Gamma_0\times X_0
\ \st\ \mathfrak h(g,\hat w)=0\right\}.
\ee
Then the following assertions hold:
\begin{enumerate}[label=(\roman*)]
\item
$Z$ is a $C^r$ Banach submanifold of
$\Gamma_0\times X_0$, with
\[
T_{(g,\hat w)}Z=\ker D\mathfrak h(g,\hat w)
\qquad\text{for every }(g,\hat w)\in Z.
\]
\item
The projection onto the first factor
$\Pi:Z\to\Gamma_0$ is a $C^r$ Fredholm map
of index $-d$. Its differential satisfies
\[
\ker D\Pi_{(g,\hat w)}
=
\left\{(0,\hat z):\hat z\in\ker J(g,\hat w)\right\}
\]
and
\[
\operatorname{Ran}D\Pi_{(g,\hat w)}
=
\left\{
h\in T_g\Gamma:
D_1\mathfrak h(g,\hat w)[h]
\in\operatorname{Ran}J(g,\hat w)
\right\}.
\]
In particular,
\[
\dim\ker D\Pi_{(\bar g,0)}
=
\dim\ker L_{\Sigma,\bar g}.
\]
\end{enumerate}
\end{proposition}

\begin{proof}
Let
\[
\mathcal E:=C^q(M;S^2T^*M),
\qquad
J_0:=J(\bar g,0),
\qquad
B_0:=D_1\mathfrak h(\bar g,0).
\]
We identify $T_g\Gamma$ with $\mathcal E$.
By Corollary \ref{cor:CharCritPoints},
$\mathfrak h$ is of class $C^r$ and
\[
D\mathfrak h(\bar g,0)[h,\hat z]
=
B_0[h]+J_0\hat z.
\]

We first show that metric variations detect the full
cokernel of $J_0$. Fix
$v\in C_c^\infty(\Sigma;\mbfV)$.
A tubular neighborhood construction gives
$\lambda\in C^\infty(M;\mathbb R)$, supported away
from $\Sing(\Sigma)$, such that
\[
\lambda|_\Sigma=0,
\qquad
(\bar\nabla\lambda)|_\Sigma=v.
\]
Set $h=\lambda\bar g$ and
$g_t=\bar g+th=(1+t\lambda)\bar g$.
For sufficiently small $|t|$, the metric $g_t$
coincides with $\bar g$ near $\Sing(\Sigma)$.
Hence Lemma \ref{lem:AdjustCoincideWith}\,
\ref{item:Xi} gives $\Xi^{g_t}_{0,0}=\id$.

For every $z\in C_c^\infty(\Sigma;\mbfV)$,
the defining first-variation identity for
$\mathfrak h$, together with the conformal
variation of area, yields
\begin{align*}
\langle B_0[\lambda\bar g],z\rangle_{\cH}
&=
\left.\frac{d}{dt}\right|_{t=0}
\left.\frac{d}{ds}\right|_{s=0}
A(g_t,sz)\\
&=
\frac n2
\left.\frac{d}{ds}\right|_{s=0}
\int_{\Sigma(sz)}\lambda\,d\scH^n(\bar g)\\
&=
\frac n2
\int_\Sigma\bar g(v,z)\,d\scH^n(\bar g).
\end{align*}
Here the mixed derivatives may be interchanged
because the variations under consideration are
smooth and supported away from the singularities.
The last equality uses $\lambda|_\Sigma=0$.
Since both $B_0[\lambda\bar g]$ and $v$ belong to
$Y\subset\cH$, it follows that
\[
B_0[\lambda\bar g]=\frac n2v.
\]

Let $\mathcal C_0:=Y/\operatorname{Ran}J_0$ and
let $\pi_0:Y\to\mathcal C_0$ be the quotient map.
Since $J_0$ is Fredholm, $\mathcal C_0$ is
finite-dimensional. The preceding construction gives
\be\label{eq:CondSurjDH}
\pi_0B_0:\mathcal E\longrightarrow\mathcal C_0
\quad\text{is surjective}.
\ee
Indeed, its image contains
$\pi_0(C_c^\infty(\Sigma;\mbfV))$, which is dense
in $\mathcal C_0$ because $C_c^\infty(\Sigma;\mbfV)$
is dense in $Y$ and $\pi_0$ is continuous and surjective.
Every linear subspace of $\mathcal C_0$ is closed,
so the image is all of $\mathcal C_0$.

We next verify that $D\mathfrak h(\bar g,0)$
has a bounded right inverse.
Choose a closed complement $X'$ of $\ker J_0$ in $X$,
and a finite-dimensional subspace $F\subset\mathcal E$
such that
\[
\pi_0B_0|_F:F\longrightarrow\mathcal C_0
\]
is an isomorphism.
Then
\[
F\oplus X'\longrightarrow Y,
\qquad
(h,\hat z)\longmapsto B_0[h]+J_0\hat z
\]
is a bounded linear isomorphism:
the quotient equation determines $h$ uniquely,
and the remaining equation determines $\hat z$
uniquely because
$J_0|_{X'}:X'\to\operatorname{Ran}J_0$
is an isomorphism.
Its inverse provides a bounded right inverse
for $D\mathfrak h(\bar g,0)$.

By continuity, after shrinking $\Gamma_0$ and $X_0$,
the restriction of $D\mathfrak h(g,\hat w)$ to
the same subspace $F\oplus X'$ remains an isomorphism
onto $Y$.
Thus $D\mathfrak h(g,\hat w)$ is surjective
with complemented kernel throughout this neighborhood.
The Banach submersion theorem proves assertion (i).

The map $\Pi$ is $C^r$ because it is the restriction
of the projection onto the first factor.
Fix $(g,\hat w)\in Z$ and write
$B:=D_1\mathfrak h(g,\hat w)$ and $J:=J(g,\hat w)$.
Assertion (i) gives
\[
T_{(g,\hat w)}Z
=
\{(h,\hat z)\in\mathcal E\times X:
B[h]+J\hat z=0\}.
\]
Since $D\Pi_{(g,\hat w)}(h,\hat z)=h$,
the stated descriptions of its kernel and range follow.

After a further shrinking of the neighborhoods,
$J$ is Fredholm of index $-d$, by continuity of
$J(g,\hat w)$ and local constancy of the Fredholm index.
Let $\pi_J:Y\to Y/\operatorname{Ran}J$ be the
quotient map.
Surjectivity of $D\mathfrak h(g,\hat w)$ implies
surjectivity of $\pi_JB$.
Consequently,
\[
\operatorname{Ran}D\Pi_{(g,\hat w)}
=
\ker(\pi_JB)
\]
is closed and has codimension
$\dim\operatorname{coker}J$.
Therefore $D\Pi_{(g,\hat w)}$ is Fredholm and
\[
\operatorname{ind}D\Pi_{(g,\hat w)}
=
\dim\ker J-\dim\operatorname{coker}J
=
-d.
\]
Finally, Lemma \ref{lem:IdentifJacobi operator}
gives $J_0=-L_{\Sigma,\bar g}$, proving the
last assertion.
\end{proof}

Applying the inverse function theorem to the projection \(\Pi\) in the case of trivial Jacobi kernel yields the following consequence.

\begin{corollary}\label{cor:ExistenceUnobstructedCase}
Let $M$ be a compact, boundaryless manifold endowed
with a smooth Riemannian metric $\bar g$, and let
$\Sigma$ be an $\MSI$ in $(M,\bar g)$.
Fix integers $k\geq n+3$, $r\geq1$ and $q\geq k+2r$,
and let $\tau\in I_+(\Sigma)$.
Assume that there are no non-trivial Jacobi fields in
$\hat W^{k,2}_\tau(\Sigma;\mbfV)$, and set
\[
d:=-\operatorname{ind}
\bigl(L_{\Sigma,\bar g}:X\to Y\bigr)\geq0,
\]
as computed in Proposition \ref{thm:CountRefined}.

Then, after possibly shrinking $\Gamma_0$ and $X_0$,
there exists a $C^r$ submanifold
$\mathcal Z\subset\Gamma_0$ of codimension $d$,
containing $\bar g$, such that $(M,g)$ contains
an $\MSI$ represented by a generalized graph over
$\Sigma$ with section $\hat w\in X_0$
if and only if $g\in\mathcal Z$.
For each $g\in\mathcal Z$, the representing section
$\hat w(g)\in X_0$ is unique, depends $C^r$ on $g$,
satisfies $\hat w(\bar g)=0$, and obeys
\[
\|\hat w(g)\|_X
\leq C\|g-\bar g\|_\Gamma,
\qquad g\in\mathcal Z,
\]
where $C=C(\bar g,\Sigma,q,k,\tau)$.

Moreover,
\[
T_{\bar g}\mathcal Z
=
\left\{
h\in C^q(M;S^2T^*M):
D_1\mathfrak h(\bar g,0)[h]
\in\operatorname{Ran}J(\bar g,0)
\right\}.
\]

In particular, if the effective index and effective
nullity at the links of all singular points vanish,
then $d=0$ and $\mathcal Z=\Gamma_0$; hence the
existence and uniqueness conclusion holds for
every $g\in\Gamma_0$.
\end{corollary}

\begin{proof}
Set $J_0:=J(\bar g,0)$,
$B_0:=D_1\mathfrak h(\bar g,0)$ and
$R:=\operatorname{Ran}J_0$.
By Lemma \ref{lem:IdentifJacobi operator}
and Proposition \ref{thm:CountRefined}, $J_0$ is
Fredholm of index $-d$ with trivial kernel.
Hence $R$ is closed and has codimension $d$ in $Y$,
and $J_0:X\to R$ is a bounded linear isomorphism.

Choose a complement $F$ of $R$ in $Y$, and let
$\pi_R:Y\to R$ be the corresponding projection.
The implicit function theorem applied to
$\pi_R\mathfrak h(g,\hat w)=0$ gives a $C^r$ map
$\psi:\Gamma_0\to X_0$, with $\psi(\bar g)=0$,
such that, after shrinking the neighborhoods,
\[
\pi_R\mathfrak h(g,\hat w)=0
\quad\Longleftrightarrow\quad
\hat w=\psi(g).
\]
The remaining equation is encoded by the $C^r$ map
\[
\mathcal O:\Gamma_0\to F,
\qquad
\mathcal O(g):=(\mathrm{id}_Y-\pi_R)
\mathfrak h(g,\psi(g)).
\]
Since $(\mathrm{id}_Y-\pi_R)J_0=0$, we have
\[
D\mathcal O(\bar g)[h]
=
(\mathrm{id}_Y-\pi_R)B_0[h].
\]
This differential is surjective by
\eqref{eq:CondSurjDH}.
The submersion theorem therefore shows that,
after further shrinking $\Gamma_0$,
$\mathcal Z:=\mathcal O^{-1}(0)$ is a $C^r$
submanifold of codimension $\dim F=d$, with
\[
T_{\bar g}\mathcal Z
=
\ker D\mathcal O(\bar g)
=
\{h:B_0[h]\in R\}.
\]

By construction,
\[
\mathfrak h(g,\hat w)=0
\quad\Longleftrightarrow\quad
g\in\mathcal Z
\ \text{and}\ 
\hat w=\psi(g).
\]
Corollary \ref{cor:CharCritPoints}, together with
the first-variation characterization of
$\mathfrak h$, identifies these solutions with
the generalized graphs that are $\MSI$ in $(M,g)$.
Thus $\hat w(g):=\psi(g)$ for $g\in\mathcal Z$
has the asserted existence, uniqueness and
regularity properties.

Finally, $\psi(\bar g)=0$ and the $C^1$ regularity
of $\psi$ give the stated estimate on a sufficiently
small ball about $\bar g$.
If $d=0$, then $F=\{0\}$, so $\mathcal O$ vanishes
identically and $\mathcal Z=\Gamma_0$.
\end{proof}

\begin{remark}[Persistence with a normalization]
When nontrivial Jacobi fields are present, the preceding
argument still gives existence and uniqueness if we restrict to a finite-codimension submanifold of ambient metrics.

More precisely, set
$K:=\ker J(\bar g,0)$, $c:=\dim K$ and
$d:=-\operatorname{ind}J(\bar g,0)$.
We impose the normalization
\[
\langle\hat w,\kappa\rangle_{\cH}=0
\qquad\text{for every }\kappa\in K,
\]
or equivalently $\pi\hat w=0$.
This amounts to $c$ scalar conditions on $\hat w$.

Under the same regularity assumptions as in
Corollary \ref{cor:ExistenceUnobstructedCase},
there exists, after shrinking the neighborhoods,
a $C^r$ submanifold $\mathcal Z_K\subset\Gamma_0$
through $\bar g$, of codimension $d+c$, such that
$(M,g)$ contains an $\MSI$ represented by a generalized
graph over $\Sigma$ with section $\hat w\in X_0$
satisfying $\pi\hat w=0$ if and only if
$g\in\mathcal Z_K$.
For each $g\in\mathcal Z_K$, this representing section
is unique, depends $C^r$ on $g$, satisfies
$\hat w(\bar g)=0$, and obeys
\[
\|\hat w(g)\|_X\leq C\|g-\bar g\|_\Gamma,
\]
where $C=C(\bar g,\Sigma,q,k,\tau)$.
Moreover,
\[
T_{\bar g}\mathcal Z_K
=
\left\{
h\in C^q(M;S^2T^*M):
D_1\mathfrak h(\bar g,0)[h]
\in\operatorname{Ran}J(\bar g,0)
\right\}.
\]

Indeed, the restriction
\[
J(\bar g,0)|_{X\cap K^\perp}:
X\cap K^\perp\longrightarrow
\operatorname{Ran}J(\bar g,0)
\]
is a bounded linear isomorphism.
Thus the proof of the preceding corollary applies
with $X\cap K^\perp$ in place of $X$.
The normalization removes the kernel but leaves
the range unchanged, so the obstruction space
has dimension
\[
\dim\operatorname{coker}J(\bar g,0)=d+c.
\]

Existence and uniqueness here concern only sections
satisfying the normalization. Metrics outside
$\mathcal Z_K$ may still admit $\MSI$ represented
by nearby generalized graphs with $\pi\hat w\neq0$.
\end{remark}

\section{The Lyapunov-Schmidt
reduction}\label{sec:LSreduction}
Throughout this and the next section, fix
$\tau\in I_+(\Sigma)$ and assume that the effective
index and nullity at every singular point of $\Sigma$
vanish. Assume also that every Jacobi field in
$\hat W^{k,2}_\tau(\Sigma;\mbfV)$ is integrable,
meaning that it is the normal component of an ambient
Killing vector field. Such fields generate global
one-parameter groups of isometries because $M$ is closed.

With $G:=\Iso(M,\bar g)$ and $H:=\Stab_G(\Sigma)$
as above, integrability gives
\[
\dim(G/H)=\dim K=c.
\]

  Furthermore, fix integers $r\geq2$
and $q\geq k+3r$, with $k\geq n+3$.
We apply Corollary \ref{cor:CharCritPoints} and
Lemma \ref{lem:LocSurj_new} with metric regularity $q-r$.
Since $q-r\geq k+2r$, this gives a $C^r$ map
\[
\varphi:
\Gamma_0(q-r)\times K_0
\longrightarrow X\cap K^\perp,
\]
where $\Gamma_0(q-r)$ is a sufficiently small neighborhood
of $\bar g$ in $\Gamma(q-r)$.
By the uniqueness assertion in Lemma \ref{lem:LocSurj_new},
its restriction to more regular metrics agrees with the
map constructed directly at that degree of regularity.
We choose $\Gamma_0(q-r)$ to be invariant under pullback by $G$,
and choose a sufficiently small, likewise $G$-invariant
neighborhood
\[
\Gamma_0(q)\subset
\Gamma(q)\cap\Gamma_0(q-r).
\]
Such choices are possible because $G$ consists of
$\bar g$-isometries.
In this section, $\Gamma_0$ without an indicated regularity
means $\Gamma_0(q)$, while $\varphi$ is understood to be
defined on the larger domain displayed above. Note that the pullback map
\[
\Gamma_0(q)\times G\longrightarrow\Gamma_0(q-r),
\qquad (g,\gamma)\longmapsto\gamma^*g,
\]
is of class $C^r$: in local coordinates, each
differentiation in $\gamma$ costs at most one
spatial derivative of $g$.
Since $q-r\geq k+2r$, Remark \ref{rem:AreaRegularity}
also gives $C^r$ regularity of $A$ at metric
regularity $q-r$.
Consequently,
\[
(g,\gamma)\longmapsto
A\bigl(\gamma^*g,\varphi(\gamma^*g,0)\bigr)
\]
is of class $C^r$.

We define an action of $H$ on $\mathscr{T}$ and $\mathscr{R}$, as follows: for every $\gamma \in H$, $\mbfx \in \mathscr{T}$ and $\mbfA \in \mathscr{R}$, and $p \in \Sing \Sigma$,
\begin{equation}\label{eqn:gamma_act_on_T_R}
      (\gamma. \mbfx)_{\gamma(p)} := d\gamma_p(\mbfx_p), \quad (\gamma. \mbfA)_{\gamma(p)} := (d\gamma_p) \mbfA_p (d\gamma_p)^{-1}\,.
  \end{equation}
A key identity for such an action will be proven in Lemma \ref{lem:ActionIdentity}. But first, let us present and prove the main result of this section.

  \begin{proposition}\label{Prop:Main_LS_Red}
    Consider the $H$-actions on $\scT$ and $\scR$
defined in \eqref{eqn:gamma_act_on_T_R}, and suppose
that for every $\gamma\in H$ there holds
    \begin{equation}\label{Equ_LS_gamma-action-scT+scR}
\Xi^g_{\gamma.\mbfx,\gamma.\mbfA}\circ\gamma
=
\gamma\circ\Xi^{\gamma^*g}_{\mbfx,\mbfA},
\qquad
\forall (g,\mbfx,\mbfA)\in((\bar g,0,0))_\delta.
\end{equation}
    For every $\gamma\in H$ and $g\in \Gamma_0$, we have 
    \begin{align}
      \gamma(\Sigma(\gamma^*g, \varphi(\gamma^*g, 0))) = \Sigma(g, \varphi(g, 0)) \,.  \label{Equ_LS-reduction}
    \end{align}
In particular, the $C^r$ map
\[
\tilde{\mathfrak a}:\Gamma_0\times G\longrightarrow\mathbb R,
\qquad
\tilde{\mathfrak a}(g,\gamma)
:=
A\bigl(\gamma^*g,\varphi(\gamma^*g,0)\bigr)
\]
descends to a $C^r$ map
$\mathfrak a:\Gamma_0\times(G/H)\to\mathbb R$.
    
    Moreover, 
    \begin{enumerate}[label={\normalfont(\roman*)}]
     \item\label{Item_LS_a-critical-iff-minimal}
For every $g\in\Gamma_0$, the coset $[\gamma]\in G/H$
is a critical point of $\mathfrak a(g,\cdot)$
if and only if
\[
\Sigma_g(\gamma)
:=
\gamma\bigl(
\Sigma(\gamma^*g,\varphi(\gamma^*g,0))
\bigr)
\]
is an $\MSI$ in $(M,g)$.

\item\label{Item_LS_a-generic-nondegenerate}
For a residual subset of $\Gamma_0(q)$,
the function $\mathfrak a(g,\cdot)$ has only
nondegenerate critical points.
     \item\label{Item_LS_geometric-distinctness}
For every $g\in\Gamma_0(q)$, the map
\[
G/H\ni[\gamma]\longmapsto
\Sigma_g(\gamma)
=
\gamma\bigl(
\Sigma(\gamma^*g,\varphi(\gamma^*g,0))
\bigr)
\]
is injective.
In particular, distinct critical points of
$\mathfrak a(g,\cdot)$ determine distinct $\MSI$
in $(M,g)$.
    \end{enumerate}
  \end{proposition}
  
  \begin{proof}
  Throughout this proof, $C$ denotes a positive constant
that may change from line to line and may depend on
$(\bar g,\Sigma,q,k,\tau)$ and the auxiliary
choices fixed in the construction.
These constants are uniform for metrics in the chosen
sufficiently small neighborhood of $\bar g$
and for elements of $G$.
The auxiliary choices are fixed in terms of
the displayed data, and their dependence is
suppressed.  

We first prove \eqref{Equ_LS-reduction} by examining
how $A$, $\mathfrak h$ and $\varphi$ transform
under the action of $H$. For $\gamma\in H$, let $\gamma_\#$ denote the pushforward
of normal sections:
\[
(\gamma_\#u)(x)
=
d\gamma_{\gamma^{-1}(x)}u(\gamma^{-1}(x)).
\]
This is an isometry of $\cH$ and a bounded linear
automorphism of $W^{k,2}_\tau(\Sigma;\mbfV)$.
Moreover, $\gamma(\Sigma(u))=\Sigma(\gamma_\#u)$.

For the actions on $\scT$ and $\scR$ defined in
\eqref{eqn:gamma_act_on_T_R}, the construction of $Q$
gives
\[
\gamma_\#Q(u,\mbfx,\mbfA)
=
Q(\gamma_\#u,\gamma.\mbfx,\gamma.\mbfA).
\]
Indeed, the radial cutoffs and normal projections
commute with the ambient isometries preserving $\Sigma$.
Thus pushforward also preserves $X$.
We choose the neighborhoods $X_0$ and $K_0$
to be $H$-invariant, as is possible by compactness
of $H$.

By \eqref{eq:WhiteAreaFunctional} and
\eqref{Equ_LS_gamma-action-scT+scR}, for
$g\in\Gamma_0$ and
$\hat w=Q(u,\mbfx,\mbfA)\in X_0$, we have
   \begin{align*}
      \Sigma(g, \gamma_\#\hat w) = \Xi^g_{\gamma.\mbfx, \gamma.\mbfA}\circ \gamma(\Sigma(u)) = \gamma\circ \Xi^{\gamma^*g}_{\mbfx, \mbfA}(\Sigma(u)) = \gamma(\Sigma(\gamma^*g, \hat w))\,;
    \end{align*}
    \begin{align*}
      A(g, \gamma_\# \hat w) 
      = \scH^n((\Xi^g_{\gamma.\mbfx, \gamma.\mbfA})^*g, \gamma(\Sigma(u))) 
      & = \scH^n((\Xi^g_{\gamma.\mbfx, \gamma.\mbfA}\circ \gamma)^*g, \Sigma(u)) \\
      & = \scH^n((\Xi^{\gamma^*g}_{\mbfx, \mbfA})^*\circ \gamma^*g, \Sigma(u))
      = A(\gamma^*g, \hat w) \,.
    \end{align*}
    Also by \eqref{Def_H_G_new}, for $v\in W^{k,2}_\tau(\Sigma)$, 
    \begin{align*}
      \langle\mathfrak{h}(g, \gamma_\#\hat w), \gamma_\# v\rangle_\cH 
      & = \left[\frac{d}{dt}\right]_{t=0} A(g, \gamma_\#(\hat w + tv)) \\
      & = \left[\frac{d}{dt}\right]_{t=0} A(\gamma^*g, \hat w + tv) 
      = \langle\mathfrak{h}(\gamma^*g, \hat w), v\rangle_\cH 
      = \langle\gamma_\#(\mathfrak{h}(\gamma^*g, \hat w)), \gamma_\# v\rangle_\cH\,.
    \end{align*}
    Therefore, 
    \begin{align*}
      \mathfrak{h}(g, \gamma_\#\hat w) = \gamma_\#(\mathfrak{h}(\gamma^*g, \hat w)) \,.
    \end{align*}
    \begin{align*}
      J(g, \gamma_\# \hat w)[\gamma_\# \hat z] 
      = \left[\frac{d}{dt}\right]_{t=0} \mathfrak{h}(g, \gamma_\#(\hat w + t\hat z))
      = \left[\frac{d}{dt}\right]_{t=0} \gamma_\#(\mathfrak{h}(\gamma^*g, \hat w+t\hat z)) 
      = \gamma_\#(J(\gamma^*g, \hat w)[\hat z]) \,. 
    \end{align*}
    In particular, $J(\bar g, 0)$ commutes with $\gamma_\#$, and thus $K=\ker J(\bar g, 0),\, K^\perp$ are preserved under every $\gamma\in H$.

    Now for every $\kappa\in K$ small and $g\in \Gamma(q)$ close to $\bar g$, denote for simplicity $\psi:=\varphi(\gamma^*g, \kappa)$, then \[
      \mathfrak{h}(g, \gamma_\# (\kappa+\psi)) = \gamma_\# (\mathfrak{h}(\gamma^*g, \kappa+\psi)) \in K \,.
    \]
    By the uniqueness of $\kappa\mapsto \varphi(g, \kappa)$ in Lemma \ref{lem:LocSurj_new}, we conclude $\varphi(g, \gamma_\#\kappa)=\gamma_\# (\varphi(\gamma^*g, \kappa))$ and \[
      \Sigma(g, \varphi(g, 0)) = \Sigma\!\left(g, \gamma_\#(\varphi(\gamma^*g, 0))\right) = \gamma\!\left(\Sigma(\gamma^*g, \varphi(\gamma^*g, 0)) \right) \,.
    \]
    This proves \eqref{Equ_LS-reduction}, and hence
$\tilde{\mathfrak a}$ descends to a $C^r$ function
$\mathfrak a:\Gamma_0\times(G/H)\to\mathbb R$.

    Now to prove \ref{Item_LS_a-critical-iff-minimal}, first observe that for every $\gamma, \vartheta\in G$, 
    \begin{align}
\tilde{\mathfrak a}(\gamma^*g,\vartheta)
&=
A\bigl(
\vartheta^*\gamma^*g,
\varphi(\vartheta^*\gamma^*g,0)
\bigr)\notag\\
&=
\tilde{\mathfrak a}(g,\gamma\vartheta).
\label{Equ_LS_G-inv-a}
\end{align}
    Therefore, it suffices to prove \ref{Item_LS_a-critical-iff-minimal} for $\gamma=\id_M$. 
    We fix a metric on $\frak g:=T_{\id_M}G$, and let $(\frak g)^{\perp}\subset \frak g$ be the orthogonal complement of the Lie algebra of $H$. 
    Note that integrability of $\Sigma$ implies that the following is a linear isomorphism: \[
      (\frak g)^{\perp}\to K, \quad V\mapsto (\xi_V|_\Sigma)^\perp\,,
    \] 
    where $\xi_V\in \cX(M)$ is the Killing vector field induced by $V$.
    
    Let $t\mapsto \gamma_t$ be an arbitrary smooth path in $G$ such that $\gamma_0=\id_M$ and that $0\neq \dot\gamma_0\in (\frak g)^{\perp}$. And let $(u_t, \mbfx_t, \mbfA_t):=P_{\Sigma, \bar g}(\varphi(\gamma_t^*g, 0))$. Set $\xi:=\xi_{\dot\gamma_0}$.
The curve $t\mapsto\gamma_t^*g$ is differentiable
with values in $\Gamma(q-r)$.
Since $\varphi$ is $C^r$ on the corresponding
metric neighborhood, the chain rule gives
\[
\left.\frac{d}{dt}\right|_{t=0}
\varphi(\gamma_t^*g,0)
=
D_1\varphi(g,0)[\mathcal L_\xi g]
=
D_1\varphi(g,0)[\mathcal L_\xi(g-\bar g)],
\]
where the last equality uses
$\mathcal L_\xi\bar g=0$. Then \[
      \Sigma_t:= \gamma_t(\Sigma(\gamma_t^*g, \varphi(\gamma_t^*g, 0))) = \Xi^g_{\mbfx_0, \mbfA_0}\circ 
      \underbrace{\left( (\Xi^g_{\mbfx_0, \mbfA_0})^{-1}\circ \gamma_t\circ \Xi^{\gamma_t^*g}_{\mbfx_t, \mbfA_t}\circ \Phi_{u_t}\right)}_{=:\varphi_t}(\Sigma) \,.
    \]
    By Lemma \ref{Lem_LS_Full1stVar} \ref{Item_LS_Full-1stVar}, we have 
\begin{align}
D_2\tilde{\mathfrak a}(g,\id_M)[\dot\gamma_0]
&=
\left.\frac{d}{dt}\right|_{t=0}
\scH^n(g,\Sigma_t) \notag\\
&=
\left\langle
\check{\mathfrak h}(g_{\mbfx_0,\mbfA_0},u_0),
\Pi_{\Sigma,u_0}(\dot\varphi_0)
\right\rangle_{\cH} \notag\\
&=
\left\langle
\mathfrak h(g,\varphi(g,0)),
\Pi_{\Sigma,u_0}(\dot\varphi_0)
\right\rangle_{\cH}.
\label{Equ_LS_da[gamma']=<H, phi'>}
\end{align}

 The first-variation formula shows that if
$\Sigma(g,\varphi(g,0))$ is minimal, then $\id_M$
is a critical point of $\tilde{\mathfrak a}(g,\cdot)$.
Conversely, if this submanifold is not minimal,
Corollary \ref{cor:CharCritPoints} implies that
$\mathfrak h(g,\varphi(g,0))\in K$ is nonzero.
Choose $\gamma_t$ above so that
$(\xi_{\dot\gamma_0}|_\Sigma)^\perp
=\mathfrak h(g,\varphi(g,0))$.
We claim that
\begin{equation}\label{Equ_LS_|Pi(phi')-xi^perp|-small}
\left\|
\Pi_{\Sigma,u_0}(\dot\varphi_0)
-(\xi_{\dot\gamma_0}|_\Sigma)^\perp
\right\|_{\cH}
\leq
C\|g-\bar g\|_{C^q}
\|(\xi_{\dot\gamma_0}|_\Sigma)^\perp\|_{\cH}.
\end{equation}
Together with
\eqref{Equ_LS_da[gamma']=<H, phi'>}
and our choice of $\gamma_t$, this gives
\[
D_2\tilde{\mathfrak a}(g,\id_M)[\dot\gamma_0]
\geq
\bigl(1-C\|g-\bar g\|_{C^q}\bigr)
\|\mathfrak h(g,\varphi(g,0))\|_{\cH}^2
>0,
\]
after shrinking $\Gamma_0$ so that
$C\|g-\bar g\|_{C^q}<1$.
     Thus, if $\id_M$ is a critical point of
$\tilde{\mathfrak a}(g,\cdot)$, then
$\Sigma(g,\varphi(g,0))$ is minimal.

    To prove \eqref{Equ_LS_|Pi(phi')-xi^perp|-small}, by Leibniz rule, \[
      \dot\varphi_0 = \underbrace{\left[\frac{d}{dt}\right]_{t=0} \left((\Xi^g_{\mbfx_0, \mbfA_0})^{-1}\circ \gamma_t\circ \Xi^g_{\mbfx_0, \mbfA_0}\circ \Phi_{u_0}\right)}_{\xi_1}
      + \underbrace{\left[\frac{d}{dt}\right]_{t=0} \left((\Xi^g_{\mbfx_0, \mbfA_0})^{-1}\circ \Xi^{\gamma_t^*g}_{\mbfx_t, \mbfA_t}\circ \Phi_{u_t}\right)}_{\xi_2} \,.
    \]

Since $\varphi(\bar g,0)=0$ and
$\Xi^{\bar g}_{0,0}=\id_M$, we have
\[
\|\Xi^g_{\mbfx_0,\mbfA_0}-\id_M\|_{C^1(M)}
\leq
C\bigl(\|g-\bar g\|_{C^q}
       +|\mbfx_0|+|\mbfA_0|\bigr)
\leq C\|g-\bar g\|_{C^q}.
\]
Writing $\xi:=\xi_{\dot\gamma_0}$ as above, it follows that
\[
\|\xi_1-\xi\circ\Phi_{u_0}\|_{L^\infty(\Sigma)}
\leq C\|g-\bar g\|_{C^q}\|\xi\|_{C^1(M)}.
\]
Using parts \ref{Item_LS_|Pi(xi)|<2|xi|} and
\ref{Item_LS_|Pi(xi)-xi^perp|<|g-bar g||xi|}
of Lemma \ref{Lem_LS_Full1stVar}, together with
$\|u_0\|_{W^{k,2}_\tau}\leq C\|g-\bar g\|_{C^q}$,
we obtain
\[
\begin{aligned}
\|\Pi_{\Sigma,u_0}(\xi_1)-(\xi|_\Sigma)^\perp\|_{\cH}
&\leq C\|g-\bar g\|_{C^q}\|\xi\|_{C^1(M)}\\
&\leq C\|g-\bar g\|_{C^q}
       \|(\xi|_\Sigma)^\perp\|_{\cH}.
\end{aligned}
\]
The last inequality follows from the finite-dimensional
isomorphism
$(\frak g)^\perp\to K$,
$V\mapsto(\xi_V|_\Sigma)^\perp$.

   After shrinking $\Gamma_0(q-r)$, the operators
$D_1\varphi(g,0)$ are uniformly bounded from
$C^{q-r}(M;S^2T^*M)$ to $X$.
The preceding chain-rule identity and the boundedness
of $P_{\Sigma,\bar g}$ therefore give
\[
\begin{aligned}
\|\dot u_0\|_{W^{k,2}_\tau}
+|\dot{\mbfx}_0|+|\dot{\mbfA}_0|
&\leq C\|\mathcal L_\xi(g-\bar g)\|_{C^{q-r}}\\
&\leq C\|g-\bar g\|_{C^{q-r+1}}
          \|\xi\|_{C^{q-r+1}}\\
&\leq C\|g-\bar g\|_{C^q}
          \|(\xi|_\Sigma)^\perp\|_{\cH}.
\end{aligned}
\]
Here $q-r+1\leq q$, and the last inequality follows
from the finite-dimensional isomorphism
$(\frak g)^{\perp}\to K$, $V\mapsto(\xi_V|_\Sigma)^\perp$.
  
    The $C^1$ dependence of $\Xi$ on the metric and
the finite-dimensional parameters, together with
Lemma \ref{Lem_LS_Full1stVar}
\ref{Item_LS_|Pi(xi)|<2|xi|}, yields
\[
\begin{aligned}
\|\Pi_{\Sigma,u_0}(\xi_2)\|_{\cH}
&\leq C\bigl(
\|\mathcal L_\xi g\|_{C^{q-r}}
+\|\dot u_0\|_{W^{k,2}_\tau}
+|\dot{\mbfx}_0|+|\dot{\mbfA}_0|
\bigr)\\
&\leq C\|g-\bar g\|_{C^q}
       \|(\xi|_\Sigma)^\perp\|_{\cH}.
\end{aligned}
\]
    Combining these proves \eqref{Equ_LS_|Pi(phi')-xi^perp|-small}, and finishes the proof of \ref{Item_LS_a-critical-iff-minimal}.

    To prove \ref{Item_LS_a-generic-nondegenerate}, following the idea of \cite{Whi91}, it suffices to prove that there is a sufficiently small neighborhood $\Gamma_0(q)$ of $\bar g\in \Gamma(q)$ such that let \[
      \cM:= \{(g, [\gamma])\in \Gamma_0(q)\times (G/H): D_2\frak a(g, [\gamma])=0\} \subset \Gamma_0(q)\times (G/H)\,,
    \]
    then
    \begin{enumerate} [label={\normalfont(\alph*)}]
      \item\label{Item_LS_D_1D_2a-surj} for every $(g, [\gamma])\in \cM$, the following is a surjective map \[
        T_g\Gamma_0(q) \to T^*_{[\gamma]}(G/H), \quad h\mapsto D_1(D_2\frak{a}(g, [\gamma]))[h] \,.
      \] 
      In particular, $\cM$ is a $C^{r-1}$-Banach submanifold of $\Gamma_0(q)\times (G/H)$;
      \item\label{Item_LS_Pi:cM->Gamma-Fredholm} the projection map $\Pi: \cM\to \Gamma_0(q)$, $(g, [\gamma])\mapsto g$ is Fredholm with index $0$.
    \end{enumerate}

    Since $\mathfrak a$ is $C^r$, its differential
$D_2\mathfrak a$, regarded as a section of the
pullback of $T^*(G/H)$ to $\Gamma_0(q)\times(G/H)$,
is $C^{r-1}$.
Assertion \ref{Item_LS_D_1D_2a-surj} therefore
makes its zero set $\cM$ a $C^{r-1}$ Banach
submanifold, and the projection $\Pi$ is
$C^{r-1}$.
Since $r\geq2$ and $\Pi$ has Fredholm index zero,
the Sard--Smale theorem applies and gives a
residual set of regular values.
At every such value $g$, all critical points of
$\mathfrak a(g,\cdot)$ are nondegenerate.

 Claim \ref{Item_LS_Pi:cM->Gamma-Fredholm} is an immediate consequence of the following standard linear algebra fact (whose proof is left to readers):
    If $E$ is a Banach space, $V$ is a closed subspace of $E\times \RR^m$ of codimension $m$, then the projection map $V\to E$, $(x, y)\mapsto x$ is Fredholm with index $0$. 

    To prove \ref{Item_LS_D_1D_2a-surj}, note that since a linear map to a finite-dimensional vector space being surjective is an open condition, and by \eqref{Equ_LS_G-inv-a}, \[
      \frak a(\gamma^*g, [\vartheta]) = \frak a(g, [\gamma\vartheta]), \qquad \forall\, \gamma, \vartheta\in G \,.
    \] 
    Thus it suffices to probe it for $g=\bar g, \gamma=\id_M$. 

    Set $J_0:=J(\bar g,0)$ and
$B_0:=D_1\mathfrak h(\bar g,0)$.
For $V\in(\frak g)^\perp$, write
$\kappa_V:=(\xi_V|_\Sigma)^\perp\in K$.
Differentiating
\eqref{Equ_LS_da[gamma']=<H, phi'>}
in the metric variable at $\bar g$ gives
\[
\bigl(D_1D_2\mathfrak a(\bar g,[\id_M])[h]\bigr)[V]
=
\left\langle
B_0[h]+J_0D_1\varphi(\bar g,0)[h],
\kappa_V
\right\rangle_{\cH}.
\]
Here we used $\mathfrak h(\bar g,0)=0$ and the
fact that the projected velocity at $\bar g$
is $\kappa_V$.
By Lemma \ref{lem:J_SelfAdjoint} and
$J_0\kappa_V=0$, this simplifies to
\[
\bigl(D_1D_2\mathfrak a(\bar g,[\id_M])[h]\bigr)[V]
=
\langle B_0[h],\kappa_V\rangle_{\cH}.
\]

By Lemma \ref{lem:LocSurj_new},
$\operatorname{Ran}J_0=Y\cap K^\perp$.
Thus the surjectivity onto the cokernel established
in \eqref{eq:CondSurjDH} implies that
$\pi B_0:C^q(M;S^2T^*M)\to K$ is surjective.
Since $V\mapsto\kappa_V$ identifies
$T_{[\id_M]}(G/H)$ with $K$, the preceding pairing
realizes every covector in $T^*_{[\id_M]}(G/H)$.
This proves \ref{Item_LS_D_1D_2a-surj}.

    To prove \ref{Item_LS_geometric-distinctness},
we may further shrink $\Gamma_0(q)$ as needed by the following, additional local uniqueness argument. Write $B=G/H$ and, for $f\in C^\infty(M)$, set
\[
b_f(g,[\gamma])
:=
\int_{\Sigma_g(\gamma)}f\,d\mathscr H^n_{\bar g}.
\]
These functions are well-defined by
\eqref{Equ_LS-reduction} and jointly $C^1$ by the
weighted area-density argument of
Remark \ref{rem:AreaRegularity} and the regularity
of $\varphi$.

Fix $b\in B$ and put $S=b(\Sigma)$.
The differentials $d_b b_f(\bar g,\cdot)$ span $T_b^*B$:
indeed, a nonzero tangent vector is represented by
a Killing field $\xi$ with nonzero normal component
$v$ on $S$.
Choosing $f$ supported near a relatively compact
regular patch of $S$, with $f|_S=0$ and
$\bar\nabla f|_S=\chi v$ for a suitable nonnegative
cutoff $\chi$, gives
\[
d_b b_f(\bar g,\cdot)[\xi]
=
\int_S\chi|v|^2\,d\mathscr H^n_{\bar g}>0.
\]
Thus, for $c=\dim B$, suitable $f_1,\ldots,f_c$
make the differential of
\[
(g,b')\longmapsto
\bigl(g,b_{f_1}(g,b'),\ldots,b_{f_c}(g,b')\bigr)
\]
invertible at $(\bar g,b)$.
The inverse function theorem gives local injectivity
in $b'$, uniformly for $g$ near $\bar g$.

If global injectivity failed arbitrarily close to
$\bar g$, compactness would give sequences
$g_j\to\bar g$ and distinct $b_j,b'_j\in B$ with
$\Sigma_{g_j}(b_j)=\Sigma_{g_j}(b'_j)$ and
$b_j\to b$, $b'_j\to b'$.
Continuity of all the weighted area integrals implies
that $b(\Sigma)$ and $b'(\Sigma)$ have the same area
measure, hence $b=b'$.
The local injectivity just proved gives a contradiction.
(The case $\dim B=0$ follows directly from this
compactness argument.)
  \end{proof}

  \begin{lemma}\label{lem:ActionIdentity}
    For the actions defined in
\eqref{eqn:gamma_act_on_T_R}, for every $\gamma \in H$, $(g,\mbfx,\mbfA)\in ((\bar g, 0, 0))_\delta$, $\mbfx \in \mathscr{T}$ and $\mbfA \in \mathscr{R}$,
    \begin{align*}
       \Xi(g, \gamma.\mbfx, \gamma.\mbfA)\circ \gamma = \gamma\circ \Xi(\gamma^*g, \mbfx, \mbfA)\,.
    \end{align*}
  \end{lemma}

\begin{proof}Fix $\gamma\in H$ and $p\in\Sing(\Sigma)$.
For $v\in T_pM$ sufficiently small, since $\gamma$
is a $\bar g$-isometry,
\[
\gamma\circ\exp_p^{\bar g}
=
\exp_{\gamma(p)}^{\bar g}\circ d\gamma_p,
\qquad
|d\gamma_p(v)|=|v|.
\]
Moreover, the actions defined in
\eqref{eqn:gamma_act_on_T_R} give
\[
(\gamma.\mbfx)_{\gamma(p)}
+e^{(\gamma.\mbfA)_{\gamma(p)}}d\gamma_p(v)
=
d\gamma_p(\mbfx_p+e^{\mbfA_p}v).
\]
Substituting these identities into the definition
\eqref{eqn:def_Psi_x_A}, and using the radial character
of the cutoff, yields
\begin{equation}\label{eqn:Psi_gamma}
\Psi_{\gamma.\mbfx,\gamma.\mbfA}\circ\gamma
=
\gamma\circ\Psi_{\mbfx,\mbfA}.
\end{equation}
This identity also holds outside the cutoff balls,
where both $\Psi$ maps are the identity and $\gamma$
preserves their union. Consequently,
\begin{equation}\label{eqn:Psi_gamma_g}
\Psi_{\gamma.\mbfx,\gamma.\mbfA}^*g
=
(\gamma^{-1})^*
\bigl(\Psi_{\mbfx,\mbfA}^*(\gamma^*g)\bigr).
\end{equation}

Set $p'=\gamma(p)$ and write
\[
\begin{aligned}
T_p
&:=T\bigl(T_pM,\bar g_p,
          (\Psi_{\mbfx,\mbfA}^*(\gamma^*g))_p\bigr),\\
T_{p'}
&:=T\bigl(T_{p'}M,\bar g_{p'},
          (\Psi_{\gamma.\mbfx,\gamma.\mbfA}^*g)_{p'}\bigr),
\end{aligned}
\]
as in Lemma \ref{lem:LocalAdj}.
By \eqref{eqn:Psi_gamma_g}, the map
$d\gamma_pT_p(d\gamma_p)^{-1}$ pulls back
$(\Psi_{\gamma.\mbfx,\gamma.\mbfA}^*g)_{p'}$
to $\bar g_{p'}$.
It is also positive definite and self-adjoint
with respect to $\bar g_{p'}$, since $d\gamma_p$
is an isometry. Uniqueness in
Lemma \ref{lem:LocalAdj} therefore gives
\begin{equation}\label{eqn:Tq_gamma_Tp}
T_{p'}=d\gamma_pT_p(d\gamma_p)^{-1}.
\end{equation}

The construction of $\Phi^\Sigma$ uses these linear
maps, the $\bar g$-exponential maps and the same radial
cutoff. Hence \eqref{eqn:Tq_gamma_Tp} implies
\[
\Phi^\Sigma\bigl(
\Psi_{\gamma.\mbfx,\gamma.\mbfA}^*g
\bigr)\circ\gamma
=
\gamma\circ
\Phi^\Sigma\bigl(
\Psi_{\mbfx,\mbfA}^*(\gamma^*g)
\bigr).
\]
As before, the identity holds globally because
$\Phi^\Sigma$ is the identity outside the cutoff balls.
Combining this with \eqref{eqn:Psi_gamma} and the
definition of $\Xi$, we obtain
\[
\begin{aligned}
\Xi(g,\gamma.\mbfx,\gamma.\mbfA)\circ\gamma
&=
\Psi_{\gamma.\mbfx,\gamma.\mbfA}
\circ\Phi^\Sigma\bigl(
\Psi_{\gamma.\mbfx,\gamma.\mbfA}^*g
\bigr)\circ\gamma\\
&=
\gamma\circ\Psi_{\mbfx,\mbfA}
\circ\Phi^\Sigma\bigl(
\Psi_{\mbfx,\mbfA}^*(\gamma^*g)
\bigr)\\
&=
\gamma\circ\Xi(\gamma^*g,\mbfx,\mbfA),
\end{aligned}
\]
as claimed.
\end{proof}

\section{Proof of the main theorems}\label{sec:MainProofs}

We first state a precise version of
Theorem \ref{thm:main} in the Introduction.

\begin{theorem}\label{thm:mainPrecise}
Let $M$ be a smooth, compact, boundaryless manifold
endowed with a smooth Riemannian metric $\bar g$,
and let $\Sigma$ be an $n$-dimensional $\MSI$
in $(M,\bar g)$.
Fix an integer $k\geq n+3$ and a weight
$\tau\in I_+(\Sigma)$.
Assume that:
\begin{enumerate}[label=(\roman*)]
\item
the effective index and effective nullity
at the links of all singular points vanish;
\item
every Jacobi field in
$\hat W^{k,2}_\tau(\Sigma;\mbfV)$ is the normal
component along $\Sigma$ of a Killing vector
field on $(M,\bar g)$.
\end{enumerate}
Set $G:=\Iso(M,\bar g)$ and $H:=\Stab_G(\Sigma)$.
Then the following conclusions hold.

\begin{enumerate}[label=(\alph*)]
\item
For every integer $q\geq k+6$, there exists
a $C^q$-open neighborhood
$\Gamma_0(q)\subset\Gamma(q)$ of $\bar g$
such that every $g\in\Gamma_0(q)$ admits
at least $LSN(G/H)$ geometrically distinct
$\MSI$ modeled on $\Sigma$.
For a residual subset of $\Gamma_0(q)$,
with respect to the $C^q$ topology, there
are at least $MN(G/H)$ such submanifolds.

\item
The set
\[
\Gamma_0(\infty):=\Gamma_0(q)\cap\Gamma(\infty)
\]
is a neighborhood of $\bar g$ in the space
of smooth metrics endowed with the $C^\infty$
topology.
Every $g\in\Gamma_0(\infty)$ admits at least
$LSN(G/H)$ geometrically distinct $\MSI$
modeled on $\Sigma$.
For a residual subset of $\Gamma_0(\infty)$,
with respect to the $C^\infty$ topology,
there are at least $MN(G/H)$ such submanifolds.
\end{enumerate}
More precisely, there exists an open neighborhood
$X_0$ of $0$ in
$X=\hat W^{k,2}_\tau(\Sigma;\mbfV)$ such that
each of the submanifolds counted above admits
a representation
\[
\Sigma'
=\gamma\bigl(\Sigma(\gamma^*g;\hat w)\bigr),
\qquad
\gamma\in G,\quad \hat w\in X_0,
\]
with
\[
\|\hat w\|_X
\leq C\|g-\bar g\|_{C^q},
\qquad C=C(\bar g,\Sigma,q,k,\tau).
\]
Here, writing
$P_{\Sigma,\bar g}(\hat w)=(u,\mbfx,\mbfA)$,
we set
\[
\Sigma(\gamma^*g;\hat w)
:=
\Sigma(\gamma^*g;u,\mbfx,\mbfA)
=
\Xi^{\gamma^*g}_{\mbfx,\mbfA}(\Sigma(u)),
\]
in accordance with Definition \ref{def:GenNormGraph}.
\end{theorem}

\begin{proof}
Fix $q\geq k+6$ and take $r=2$.
The regularity assumptions of
Proposition \ref{Prop:Main_LS_Red} are satisfied.
In particular, the reduced function
\[
\mathfrak a:\Gamma_0(q)\times(G/H)\to\mathbb R
\]
is of class $C^2$.

By Proposition \ref{Prop:Main_LS_Red}\,
\ref{Item_LS_a-critical-iff-minimal},
the critical points of
$\mathfrak a_g:=\mathfrak a(g,\cdot)$
determine $\MSI$ in $(M,g)$.
By \ref{Item_LS_geometric-distinctness}
of the same proposition, distinct critical
points determine geometrically distinct
submanifolds. The stated representation follows from the
construction, with
$\hat w=\varphi(\gamma^*g,0)$; its estimate follows
from $\varphi(\bar g,0)=0$, the $C^1$ regularity
of $\varphi$, and the uniform boundedness of
pullback by $\gamma\in G$ in the $C^q$ norm.
The critical-point bounds defining $LSN(G/H)$
and $MN(G/H)$ therefore give (a), using
\ref{Item_LS_a-generic-nondegenerate}
for the generic assertion.

For (b), the existence bound follows directly
from (a).
To justify genericity in the smooth topology,
consider the set
\[
\Gamma_{\text{Morse}}(q):=
\left\{
g\in\Gamma_0(q):
\mathfrak a_g\text{ is a Morse function on }G/H
\right\}.
\]
 By Proposition \ref{Prop:Main_LS_Red}, the set
$\Gamma_{\mathrm{Morse}}(q)$ is dense in $\Gamma_0(q)$.
It is also open, since the Morse property is
stable under $C^2$ perturbations on the compact manifold $G/H$.

For every integer $s\geq q$, its restriction
$\Gamma_{\text{Morse}}(q)\cap\Gamma(s)$ is dense in
$\Gamma_0(q)\cap\Gamma(s)$.
To see this, fix $s\geq q$ and consider the
restriction
\[
\mathfrak a\big|_{(\Gamma_0(q)\cap\Gamma(s))\times(G/H)},
\]
where $\Gamma_0(q)\cap\Gamma(s)$ is endowed
with the $C^s$ topology.
The transversality argument in
Proposition \ref{Prop:Main_LS_Red} relies on
the surjectivity of the derivative in the
metric variable onto $T^*_{[\gamma]}(G/H)$.
This surjectivity still holds when metric
variations are restricted to $C^s$ tensors:
smooth tensors are dense among $C^q$ tensors,
so their images are dense in the target;
being a linear subspace of a finite-dimensional
space, this image is also closed and hence
equals the whole target.
The same transversality argument therefore
gives density in the $C^s$ topology.

The approximation argument in the proof of
\cite[Lemma 4.30]{CLW24}, applied within
$\Gamma_0(q)$, shows that
$\Gamma_{\mathrm{Morse}}(q)\cap\Gamma_0(\infty)$
is open and dense, hence residual, in
$\Gamma_0(\infty)$ for the $C^\infty$ topology.
For each metric in this set, the Morse bound
and the geometric distinctness assertion of
Proposition \ref{Prop:Main_LS_Red} give at least
$MN(G/H)$ submanifolds.
This proves (b).\end{proof}

\begin{proof}[Proof of Corollary \ref{cor:main}]
We need to verify the integrability condition
(\ref{MainThm:hp2}) of Theorem \ref{thm:main}.
Write $n-1:=\dim\Sigma_0$.
Since $\Sigma_0$ is minimal and not totally
geodesic, $n-1\geq2$.
Let $\mbfC_0\subset\R^N$ be the regular cone
over $\Sigma_0$, and let
\[
\mbfC:=\mbfC_0\conetimes\R\subset\R^{N+1}
\]
be the cone over its spherical suspension
$\Sigma$. Thus $\mbfC$ is a cylindrical cone
with singular set $\{0\}\times\R$.
We denote the corresponding Jacobi operators
on their regular parts by $L_{\mbfC_0}$
and $L_{\mbfC}$, respectively.

 We will verify the claimed integrability condition by appealing to the classification of finite (Dirichlet) energy Jacobi fields as spelled out in \cite[Proposition 3.19, item (ii)]{CMSW25}, proven\footnote{Although stated for hypersurfaces, the relevant
decomposition extends to arbitrary codimension: one expands
normal sections in an orthonormal basis of eigensections
of the normal Jacobi operator on the link. The resulting
scalar coefficients satisfy the same equations in the
radial and axial variables as in the hypersurface case,
so the separation-of-variables argument applies unchanged.} in \cite[Section 3]{Sim21}; loosely speaking, this provides the precise relation between the Jacobi fields on $\mbfC$ and those on $\mbfC_0$. The strict stability of $\mbfC_0$ required in that result
follows from the vanishing of the effective index of
$\Sigma_0$ and the inequality $n\geq 3$.

To that purpose, let us remark that: first, there is an isomorphism between finite energy Jacobi fields on $\Sigma$ and \emph{homogeneous degree 1} finite energy Jacobi fields on $\mbfC\cap \BB(1)$, 
and (ii) every section in
$\hat W^{\ell,2}_{\tau}(\Sigma;\mbfV)$ has finite energy
when $\ell\geq1$ and $\tau>1$. Indeed, this follows
from the weighted norm for the unaugmented part,
while the translation- and rotation-like sections
have finite energy because $n\geq3$.

Hence, by \cite[Proposition 3.19, item (ii)]{CMSW25} (applied for $d=1$) there must hold $\gamma\in 1-\Z_{\geq 0}$ therein, which splits into two cases:
\begin{enumerate}[label=(\alph*)]
\item if $\gamma=1$ (and so $q=0$) one finds Jacobi fields on $\mbfC$ that are induced (by 1-homogeneous extension in the cylindrical coordinate $r$) by Jacobi fields on $\Sigma_0$, which we know (by the assumption that the effective
\emph{nullity} vanishes) are all determined by isometries of $\SSp^{N-1}$; thus, the corresponding Jacobi fields on $\Sigma$
are the normal components of ambient Killing vector
fields on $\SSp^N$ that vanish at both poles;
\item else, if $\gamma<1$  one finds Jacobi fields on $\mbfC$ that are induced (through multiplication by the ``height variable" $y$) by vector fields on $\Sigma_0$ contributing to the Morse index thereof, which we know (by the assumption that the effective \emph{index} vanishes) are all determined by translations of $\R^N$ (hence, a posteriori, necessarily $\gamma=0$ and $q=1$); the corresponding Jacobi fields on $\Sigma$
are the normal components of the ambient Killing
vector fields
\[
(x,y)\longmapsto (yv,-\langle v,x\rangle),
\qquad v\in\R^N,
\]
on $\SSp^N\subset\R^N\times\R$.
\end{enumerate}
Since sums of ambient Killing vector fields are again
Killing, every Jacobi field of $\Sigma$ is the normal
component of an ambient Killing vector field on
$\SSp^N$. Thus, the integrability hypothesis is
satisfied, and Theorem \ref{thm:main} yields the
desired conclusion.
The second part of the statement then comes -- for the case $\mathbb{K}=\mathbb{R}$ -- as a straightforward consequence of the first thanks to Lemma \ref{lem:RauSem} and the calculations of the topological invariants presented in Appendix \ref{sec:topology}, see in particular Corollary \ref{cor:MorseNumber} and Corollary \ref{cor:LSnumber} therein.
\end{proof}

\appendix

\section{Analytic mean curvature as a $C^q$ map}\label{app:C^q}
  The goal of this appendix is to prove Proposition \ref{Prop:C^qMSO}, which we shall obtain as a consequence of general results on ``composition maps" determined by (suitably regular) bundle maps.
  We present our analysis in an abstract framework in Section \ref{Subsec_BundleMap}, and then proceed with the proof of the differentiability of the analytic mean curvature map in Section \ref{Subsec_Pf Lem C^qMSO}.

\subsection{Bundle map and regularity} \label{Subsec_BundleMap}
  Throughout this subsection, we fix the following setup. Let  $q>n\geq 2$ be integers, $\Lambda>1$, and
  \begin{enumerate} [label=(\alph*)]
    \item\label{Item_AppSetup_Sigma, g} $\Sigma$ be an $n$-dimensional smooth manifold, $\bar g$ be a smooth metric on $\Sigma$ (not necessarily complete), $\bar\nabla$ be the covariant derivative with respect to $\bar g$;
    \item\label{Item_AppSetup_r bd} $\mfr: \Sigma\to \RR_{>0}$ be a smooth function such that $\mfr\leq \Lambda$, $\mfr^{\ell-1}|\bar\nabla^\ell \mfr|_{\bar g}\leq \Lambda$ for every $1\leq \ell\leq q$;
    \item\label{Item_AppSetup_Coord Bd} $\{U_\alpha\}$ be an open cover of $\Sigma$, $\varphi_\alpha: \BB_1^n \to U_\alpha$ be diffeomorphisms, $\mfr_\alpha:= \mfr\circ\varphi_\alpha(0)\in \RR$ such that, 
    \begin{align*}
      \Lambda^{-1}\leq \mfr_\alpha^{-1}\cdot  \mfr|_{U_\alpha} \leq \Lambda\,, \qquad
      g_\alpha:=\mfr_\alpha^{-2}\cdot \varphi_\alpha^* \bar g \ \text{ is $\Lambda$-$C^q$ equivalent to $g_\eucl$}\,, \qquad
      \sum_\alpha\id_{U_\alpha} \leq \Lambda \,.
    \end{align*}
  \end{enumerate}
\emph{Throughout this subsection, constants may depend
on the fixed base data and their quantitative
bounds $(\bar g,q,\Lambda)$.
This dependence may be suppressed in comparison
subscripts. All local estimates are uniform
in the chart index $\alpha$.
Dependence on a bounded vector bundle includes
its rank and the bounds for its metric,
connection and chosen trivializations.}

 \

  Let $E \xrightarrow{\pi} \Sigma$ be a (smooth) real vector bundle over $\Sigma$ of rank $m\geq 1$, with a smooth fiberwise Euclidean metric $h$ and a connection $\nabla$, not necessarily compatible with $h$. For $\Lambda' > 1$, we will call the triple $(E, h, \nabla)$ \textbf{$\Lambda'$-$C^q$ bounded}, if for each $\alpha$, there exists a smooth map $X_\alpha: E_{U_\alpha}:=\pi^{-1}(U_{\alpha}) \to \RR^m$ which is fiberwise linear isomorphic, and such that
  \begin{align}
    & \Lambda'^{-1}|v|_h \leq |X_\alpha(v)|_{\RR^m} \leq \Lambda'|v|_h\,,  & & \forall\, v\in E_{U_\alpha} \,,
    \label{Equ_App_|X_alpha(v)| sim |v|} \\  
    & \sum_{1\leq \ell\leq q}\mfr^\ell|\nabla^\ell s_{\alpha, i}|_{\bar g, h} \leq \Lambda'\ \ \text{ on }U_\alpha\,,  & & \forall\, 1\leq i\leq m,\ \forall\alpha  \,,
    \label{Equ_App_Conn bded}
  \end{align}
  where $s_{\alpha, i}\in \Gamma(E_{U_\alpha})$ is the local section determined by $X_\alpha\circ s_{\alpha, i} = e_i\in \RR^m$. Note that $\Phi_\alpha:= (\varphi_\alpha^{-1}\circ\pi, X_\alpha): E_{U_\alpha}\to \BB_1\times \RR^m$ is a local trivialization. 
  We may simply call $(E, h, \nabla)$ \textbf{$C^q$ bounded} if it is $\Lambda'$-$C^q$ bounded for some $\Lambda'>1$.

 Given $k\in\mathbb{N}$ we let $C^{k}(E)$ denote the subset of $C^k_{\textrm{loc}}$ sections of the bundle $E$, and $W^{k,2}_\tau(E)$ the closure of $C^{k}_c(E)$ with respect to the natural Sobolev weighted norm as determined by the Riemannian metrics $\bar g, h$.
 The following lemma allows us to estimate the norms of sections of a bounded bundle in terms of local trivializations over $U_\alpha$.
  
  \begin{lemma} \label{Lem_|v|_W^k,2 sim sum_alpha |v_alpha|_W^k,2}
    Let $(E, h, \nabla)$ be a $\Lambda'$-$C^q$ bounded vector bundle over $\Sigma$ of rank $m$. Then for every $1\leq k\leq q$, $v\in C^k(E)$ and $\alpha$, we have 
    \begin{align}
      \left(\sum_{j=0}^k \mfr^j|\nabla^j v|_{\bar g, h} \right)\circ \varphi_\alpha \sim_{\Lambda',m} \sum_{j=0}^k |D^j v_\alpha| \qquad \text{ on } \BB_1^n \,,  \label{Equ_Ptwise Est [v]_W^k sim [v_alpha]_W^k}
    \end{align}
    where $v_\alpha:= X_\alpha\circ v\circ \varphi_\alpha: \BB_1^n\to \RR^m$. In particular, for every $\tau\in \RR$ we have, 
    \begin{align}
      \|v\|_{W^{k,2}_\tau}^2 \sim_{\Lambda', m, \tau} \sum_\alpha \mfr_\alpha^{-2\tau} \|v_\alpha\|_{W^{k,2}}^2 \,.  \label{Equ_|v|_W^k sim sum_alpha |v_alpha|_W^k}
    \end{align}
  \end{lemma}
  \begin{proof}
    Fix an $\alpha$. On $\BB_1^n$ we have, 
    \begin{align*}
      \left(\sum_{j=0}^k \mfr^j|\nabla^j v|_{\bar g, h} \right)\circ \varphi_\alpha 
      = \left(\sum_{j=0}^k (\mfr\cdot \mfr_\alpha^{-1})^j |\nabla^j v|_{\mfr_\alpha^{-2}\bar g, h} \right)\circ \varphi_\alpha 
      \sim_{\Lambda', m} \sum_{j=0}^k |D^j v_\alpha|_{g_{\eucl}, h} \,;
    \end{align*}
    where the equivalence follows from \ref{Item_AppSetup_Coord Bd}, \eqref{Equ_App_Conn bded}, and the fact that $v = \sum_iv_\alpha^i\, s_{\alpha,i}$. This proves \eqref{Equ_Ptwise Est [v]_W^k sim [v_alpha]_W^k}.
    
    To prove \eqref{Equ_|v|_W^k sim sum_alpha |v_alpha|_W^k}, combining \eqref{Equ_Ptwise Est [v]_W^k sim [v_alpha]_W^k} with the assumptions in \ref{Item_AppSetup_Coord Bd} yields for any $v\in C^{k}_c(E)$ \[
      \|v\|_{W^{k,2}_\tau}^2 
      \sim \sum_\alpha \int_{U_\alpha} \sum_{j=0}^k \left(\mfr^{j-\tau} |\nabla^j v|_{\bar g, h} \right)^2 \mfr^{-n}\ d\scH^n_{\bar g}
      \sim_{\Lambda',m, \tau} \sum_\alpha \int_{\BB_1^n} \mfr_\alpha^{-2\tau}\sum_{j=0}^k |D^j v_\alpha|_{g_\eucl} ^2 \ d\scH^n_{g_{\eucl}} \,,
    \] 
    at which stage the conclusion follows by density.
  \end{proof}

  Let $(E_i, h_i, \nabla_i)$ be $C^q$-bounded vector bundles of rank $m_i>0$ over $\Sigma$, for $i=1,2$. We define $\breve\scC^q_{\tau, \tau'}(E_1, E_2)$ to be the space of fiber-preserving $C^q$ maps $\psi: E_1\to E_2$ such that
  \begin{align}
    \|\psi\|_{\breve\scC^q_{\tau, \tau'}} := \sup_\alpha \mfr_\alpha^{\tau-\tau'} \|\psi_\alpha\|_{C^q} <+\infty \,,\label{Equ_Def_breve scC^q bundle map}
  \end{align}  
  and define $\scC^q_{\tau, \tau'}(E_1, E_2)$ to be the subspace of $\breve\scC^q_{\tau, \tau'}(E_1, E_2)$ such that
  \begin{align}
    \|\psi\|_{\scC^q_{\tau, \tau'}} := \sup_\alpha \mfr_\alpha^{\tau-\tau'} \|\psi_\alpha\|_{C^q} + \|\psi(0)\|_{W^{q,2}_{\tau'}}  < +\infty \,, \label{Equ_Def_scC^q bundle map}
  \end{align}
  here $\psi_\alpha:= X_{2,\alpha}\circ \psi\circ \Phi_{1,\alpha}^{-1}: \BB_1^n \times \RR^{m_1}\to \RR^{m_2}$ is the bundle map under trivializations $\Phi_{i,\alpha}$. 
  The reason to include the second term in defining $\|\psi\|_{\scC^q_{\tau, \tau'}}$ is to guarantee that the map $W^{k,2}_\tau(E_1)\ni f\mapsto \psi\circ f$ takes value in $W^{k,2}_{\tau'}(E_2)$ when $k\leq q-1$ (see the discussion in Lemma \ref{Lem_|psi(f_1-f_2)|<|psi||f_1-f_2|}).  
  Note that $\scC^q_{\tau, \tau'}(E_1, E_2)$ is naturally a Banach space, where the fiberwise addition and scalar multiplication give the vector space structure.

  \begin{lemma} \label{Lem_|psi(f_1-f_2)|<|psi||f_1-f_2|}
    Let $n<k\leq q-1$, $E_1, E_2$ be $C^q$-bounded vector bundles over $\Sigma$, $\psi\in \scC^q_{\tau, \tau'}(E_1, E_2)$, where $\tau\geq 0$, $\tau'\in \RR$. Then,
    \begin{enumerate} [label=(\roman*)]
      \item\label{Item_psi(f) in W^k,2_tau'} for every $f\in W^{k,2}_\tau(E_1)$ there holds $\psi\circ f\in W^{k,2}_{\tau'}(E_2)$;
      \item\label{Item_|psi(f_1)-psi(f_2)| bd} for every $f_\pm\in W^{k,2}_\tau(E_1)$ with $\|f_\pm\|_{W^{k,2}_\tau}\leq M$, we have 
      \begin{align}
        \|\psi\circ f_+ - \psi\circ f_-\|_{W^{k,2}_{\tau'}(E_2)} \lesssim_{E_1, E_2,\tau, \tau'} \|\psi\|_{\breve\scC^q_{\tau, \tau'}} \cdot (1+\Lambda^\tau M)^{k}\|f_+ - f_-\|_{W^{k,2}_\tau(E_1)} \,.  \label{Equ_|psi(f_1-f_2)|<|psi||f_1-f_2|}
      \end{align}
    \end{enumerate}
  \end{lemma}
  \begin{proof}
    Note that \ref{Item_psi(f) in W^k,2_tau'} follows from \ref{Item_|psi(f_1)-psi(f_2)| bd} and \eqref{Equ_Def_scC^q bundle map} by taking $f_+=f$, $f_-=0$ and employing the triangle (Minkowski) inequality in $W^{k,2}_{\tau'}$ based on the very definition of (the norm on) $\scC^q_{\tau, \tau'}(E_1, E_2)$.
    
    To prove \ref{Item_|psi(f_1)-psi(f_2)| bd}, first assume without loss of generality $f_\pm\in C^k_{c}(E_1)$.
    By Lemma \ref{Lem_|v|_W^k,2 sim sum_alpha |v_alpha|_W^k,2}, it suffices to bound 
    \begin{align*}
      &(\psi\circ f_+ - \psi\circ f_-)_\alpha(x) = X_{2,\alpha}\circ (\psi\circ f_+ - \psi\circ f_-)\circ \varphi_\alpha(x) \\
      & \qquad = (X_{2,\alpha}\circ \psi \circ \Phi_{1,\alpha}^{-1})\circ (\Phi_{1,\alpha}\circ f_+\circ \varphi_\alpha)(x) - (X_{2,\alpha}\circ \psi \circ \Phi_{1,\alpha}^{-1})\circ (\Phi_{1,\alpha}\circ f_-\circ \varphi_\alpha)(x) \\
      & \qquad = \psi_\alpha(x, f_{+,\alpha}(x)) - \psi_\alpha(x, f_{-,\alpha}(x)) \\
      & \qquad = \int_0^1 \partial_z\psi_\alpha \Big(x,\ tf_{+,\alpha}(x)+(1-t)f_{-,\alpha}(x)\Big)\cdot (f_{+,\alpha}-f_{-,\alpha})(x)\ dt\,;
    \end{align*}
    where $\psi_\alpha(x,z)$ is the map on $\BB_1^n\times \RR^{m_1}$ to $\RR^{m_2}$ defined above. By the chain rule and the Leibniz (product) rule, there holds 
    \begin{align*}
      & \quad \|(\psi\circ f_+ - \psi\circ f_-)_\alpha\|_{W^{k,2}(\BB_1^n)} \\
      & \lesssim \|\psi_\alpha\|_{C^{k+1}} \Big[ (1+\|f_{+,\alpha}\|_{C^{\lfloor k/2\rfloor}} + \|f_{-,\alpha}\|_{C^{\lfloor k/2\rfloor}})^k\|(f_+-f_-)_\alpha\|_{W^{k,2}} \\
      & \qquad\qquad\quad + \, (1+\|f_{+,\alpha}\|_{W^{k,2}} + \|f_{-,\alpha}\|_{W^{k,2}})^k \|(f_+-f_-)_\alpha\|_{C^{\lfloor k/2\rfloor}} \Big] \\
      & \lesssim \mfr_\alpha^{\tau'-\tau} \|\psi\|_{\breve\scC^q_{\tau, \tau'}}(1 + \Lambda^\tau M)^k \|(f_+-f_-)_\alpha\|_{W^{k,2}(\BB_1^n)} \,.
    \end{align*}
    Here in the second inequality, we used \eqref{Equ_Def_scC^q bundle map}, the Sobolev embedding of $W^{k,2}(\BB_1^n)\hookrightarrow C^{\lfloor k/2\rfloor}(\BB_1^n)$ since $k>n$, and the assumption $\tau\geq 0$ to bound $\|f_{\pm,\alpha}\|_{W^{k,2}}$ by $\mfr_\alpha^\tau\|f_\pm\|_{W^{k,2}_\tau}\leq \Lambda^\tau M$. 
    Squaring both sides and applying \eqref{Equ_|v|_W^k sim sum_alpha |v_alpha|_W^k} proves \ref{Item_|psi(f_1)-psi(f_2)| bd}. 
  \end{proof}

  Given $C^q$-bounded vector bundles $(E_i, h_i, \nabla_i)$ over $\Sigma$, $i=1,2$ and an integer $k>0$ it is easy to check that $\Hom(E_1^{\otimes k}, E_2)$ admits a natural induced metric $h$ and connection $\nabla$, and it is also $C^q$-bounded in the sense defined above. Moreover, we have a natural isomorphism $\Hom(E_1, \Hom(E_1^{\otimes k}, E_2))=\Hom(E_1^{\otimes (k+1)}, E_2)$. When $k=0$ it is understood that $\Hom(E_1^{\otimes k}, E_2)\equiv E_2$. The reader may want to compare our approach, and in particular the next statement, with that in \cite[Appendix]{Whi87}.

  \begin{lemma}\label{Lem_Differentiable}
    Let $q-2\geq k>n$, $\tau\geq 0$, $\tau'\in \RR$. For every integer $0\leq \ell \leq q-k$, there exists a bounded linear map \[
      \cJ^\ell: \scC^q_{\tau, \tau'}(E_1, E_2) \to \breve\scC^{q-\ell}_{\tau, \tau'}(E_1, \Hom(E_1^{\otimes \ell},E_2))  
    \] 
    such that the following holds for every $\psi\in \scC^q_{\tau, \tau'}(E_1, E_2)$ (denote for simplicity $\psi^{(\ell)}:= \cJ^\ell(\psi)$): 
    \begin{itemize}
      \item $\psi^{(0)}=\psi$;
      \item for every $0\leq \ell\leq q-k-2$ and every $\Lambda''$-bounded $f, h_0, h_1, \dots, h_\ell\in W^{k,2}_\tau(E_1)$, we have 
      \begin{align}
       \begin{split}
         & \left\|(\psi^{(\ell)}\circ (f+h_0)-\psi^{(\ell)}\circ f)(h_1\otimes\dots\otimes h_\ell) - (\psi^{(\ell+1)}\circ f)(h_0\otimes h_1\otimes\dots\otimes h_\ell) \right\|_{W^{k,2}_{\tau'}(E_2)} \\
         & \qquad \lesssim_{E_1, E_2,\tau,\tau', \Lambda''} \|\psi\|_{\scC^q_{\tau, \tau'}(E_1, E_2)} \cdot \|h_0\|_{W^{k,2}_{\tau}(E_1)}^2\cdot \|h_1\|_{W^{k,2}_{\tau}(E_1)}\dots \|h_\ell\|_{W^{k,2}_{\tau}(E_1)} \,.
       \end{split} \label{Equ_|psi^l(f+h)-psi^l(f) -psi^(l+1)(f)h| Bd}
      \end{align}
    \end{itemize}
    In particular, the following composition map is $C^{q-k-1}$: \[
      \scC^q_{\tau, \tau'}(E_1, E_2) \times W^{k,2}_\tau(E_1) \to W^{k,2}_{\tau'}(E_2) \,, \qquad (\psi, f)\mapsto \psi\circ f \,.
    \]
  \end{lemma}
  \begin{proof}
    Define $\psi^{(\ell)}$ inductively by requiring that for every $f, h_0, \dots, h_\ell\in C_c^q(E_1)$, \[
      (\psi^{(\ell+1)}\circ f)(h_0\otimes \dots \otimes h_\ell) := \left.\frac{d}{ds}\right|_{s=0} (\psi^{(\ell)}\circ (f+sh_0))(h_1\otimes \dots \otimes h_\ell) \,.
    \]
    Since under local trivialization $\Phi_{1,\alpha}, \Phi_{2,\alpha}$, this reads as \[
      (\psi^{(\ell+1)})_\alpha(x, z)(z_0\otimes z_1\otimes \dots \otimes z_\ell) = \left.\frac{d}{ds}\right|_{s=0} (\psi^{(\ell)})_\alpha(x, z+sz_0)(z_1 \otimes \dots \otimes z_\ell)\,;
    \]
    inductively, it is clear from the definition \eqref{Equ_Def_scC^q bundle map} that $\psi^{(\ell)}\in \breve\scC^{q-\ell}_{\tau, \tau'}(E_1, \Hom(E_1^{\otimes \ell},E_2))$, and that $\cJ^\ell$ is bounded linear. 

    Then by Lemma \ref{Lem_|psi(f_1-f_2)|<|psi||f_1-f_2|}, when $0\leq \ell \leq q-k-2$, for every $f, h_0, \dots, h_\ell\in C_c^q(E_1)$, 
    \begin{align*}
      & \left\|(\psi^{(\ell)}\circ (f+h_0)-\psi^{(\ell)}\circ f)(h_1\otimes\dots\otimes h_\ell) - (\psi^{(\ell+1)}\circ f)(h_0\otimes\dots\otimes h_\ell) \right\|_{W^{k,2}_{\tau'}(E_2)} \\
      & \quad = \left\|\int_0^1 \left(\frac{d}{ds}(\psi^{(\ell)}\circ (f+ sh_0))(h_1\otimes\dots\otimes h_\ell) - (\psi^{(\ell+1)}\circ f)(h_0\otimes\dots\otimes h_\ell) \right)\ ds \right\|_{W^{k,2}_{\tau'}(E_2)} \\
      & \quad \leq \int_0^1 \left\|(\psi^{(\ell+1)}\circ (f+ sh_0))(h_0\otimes \dots\otimes h_\ell) - (\psi^{(\ell+1)}\circ f)(h_0 \otimes\dots\otimes h_\ell) \right\|_{W^{k,2}_{\tau'}(E_2)}\ ds  \\
      & \quad \lesssim_{E_1, E_2,\tau,\tau',\Lambda''} \|\psi\|_{\scC^q_{\tau, \tau'}(E_1, E_2)} \cdot \|h_0\|_{W^{k,2}_{\tau}(E_1)}\cdot \|h_0\otimes \dots \otimes h_\ell\|_{W^{k,2}_{\tau}(E_1^{\otimes (\ell+1)})} \\
      & \quad \lesssim_{E_1, E_2,\tau,\tau',\Lambda''} \|\psi\|_{\scC^q_{\tau, \tau'}(E_1, E_2)} \cdot \|h_0\|_{W^{k,2}_{\tau}(E_1)}^2\cdot \|h_1\|_{W^{k,2}_{\tau}(E_1)}\dots \|h_\ell\|_{W^{k,2}_{\tau}(E_1)} \,.
    \end{align*}
    where we use Lemma \ref{Lem_|psi(f_1-f_2)|<|psi||f_1-f_2|} in the second inequality, and the product rule in the last inequality. Then, a standard approximation argument allows us to obtain the same inequality for $f, h_0, \dots, h_\ell\in W^{k,2}_\tau(E_1)$.

    To prove $C^{q-k-1}$ regularity of the composition map, note that $(\psi, f)\mapsto \psi\circ f$ depends linearly on $\psi$. Hence it suffices to prove the $C^{q-k-1}$ regularity of $\psi_\#: f\mapsto \psi\circ f$ for every fixed $\psi\in \scC^q_{\tau, \tau'}$. 
    Now, for every $0\leq \ell\leq q-k-1$ and every $f, h_1, \dots, h_\ell\in W^{k,2}_\tau(E_1)$, \eqref{Equ_|psi^l(f+h)-psi^l(f) -psi^(l+1)(f)h| Bd} shows (inductively) that the $\ell$-th derivative of $\psi_\#$ exists and is given by \[
      D^\ell(\psi_\#)|_f(h_1\otimes \cdots \otimes h_\ell) = (\psi^{(\ell)}\circ f)(h_1\otimes \cdots \otimes h_\ell) \,; 
    \]
    which is a well-defined map $D^\ell(\psi_\#): W^{k,2}_\tau(E_1)\to \Hom(W^{k,2}_{\tau}(E_1)^{\otimes \ell}, W^{k,2}_{\tau'}(E_2))$ since $\psi^{(\ell)}\in \breve\scC^{q-\ell}$. That being said, the continuity of $D^\ell(\psi_\#)$ (that is, the continuity of $D^\ell(\psi_\#)|_f$ with respect to $f\in W^{k,2}_{\tau}(E_1)$) follows from \eqref{Equ_|psi(f_1-f_2)|<|psi||f_1-f_2|}.
  \end{proof}

\subsection{Proof of Proposition \ref{Prop:C^qMSO}.} \label{Subsec_Pf Lem C^qMSO}
  We start with the following existence lemma of good local coordinates and bundle trivializations.
  \begin{lemma}\label{Lem_MSI w Loc Coord}
    Let $M$ be a compact, boundaryless manifold and let $\Sigma$ be an $\MSI$ under a smooth metric $\bar g$, $\mfr$ be a radius function for $\Sigma$ (as defined in Section \ref{subs:DefPrelim}); $\bar q>0$ be an integer. Then there exists $\Lambda_1(\Sigma, \bar g,\bar q)>1$ such that for every $p\in \Sigma$, there exist an open neighborhood $\hat U_p\subset M$ of $p$ and a diffeomorphism $\hat\varphi_p: \BB^n_1\times \BB^{N-n}_1 \to \hat U_p$ 
    such that
    \begin{enumerate} [label=(\roman*)]
      \item\label{Item_Coord_r bded} $\mfr\leq \Lambda_1$, $\mfr^{\ell-1}|\bar\nabla^\ell (\mfr|_{\Sigma})|_{\bar g} \leq \Lambda_1$ for every $1\leq \ell\leq \bar q$;
      \item\label{Item_Coord_g_p bded} $\Lambda_1^{-1}\leq \mfr(p)^{-1}\cdot \mfr|_{\hat U_p}\leq \Lambda_1$, $\bar g_p:= \mfr(p)^{-2}\hat\varphi_p{}^*\bar g$ is $\Lambda_1$-$C^{\bar q}$ equivalent to $g_\eucl$ on $\BB^n_1\times \BB^{N-n}_1$;
      \item\label{Item_Coord_Sigma = B^n} $\hat\varphi_p(\BB^n_1\times\{0\}) = \Sigma \cap \hat U_p =: U_p$, $\hat\varphi_p(0,0)=p$;
      \item\label{Item_Coord_Y_p:T^perp to R^(N-n)} for every $(x, v)\in (T^\perp\Sigma)|_{U_p}$, if we let $Y_p(x, v):= d(\hat \varphi_p^{-1})|_x(v)\in \{0\}\times \RR^{N-n}$, which defines a fiberwise linear map $Y_p: (T^\perp\Sigma)|_{U_p} \to \RR^{N-n}$, one has $|Y_p(x, v)|_{g_\eucl} = \mfr(p)^{-1}|v|_{\bar g}$;
      \item\label{Item_Coord_exp(u)=(x, u_p(x))} for every $u\in \Gamma(T^\perp \Sigma)$ such that $|u|\leq \Lambda_1^{-1}\mfr$ and every $x\in \BB^n_1$, we have \[
        \hat \varphi_p^{-1}\circ E_u \circ \hat\varphi_p(x, 0) 
        = \big(x, Y_p\circ u\circ \hat\varphi_p(x, 0) \big) 
        =: (x, u_p(x)) \,,
      \]
      where $E_u(x'):= \exp^{\bar g}_{x'}(u(x'))\in M$ for every $x'\in \Sigma$,  which defines a local representative $u_p: \BB^n_1\to \RR^{N-n}$ of $u$ in $\hat U_p$. (This is the map denoted by $\Phi_u$ in the
main text; here we write $E_u$ to distinguish
it from the bundle trivializations above.)
    \end{enumerate}
    
    Furthermore, there exists $\Lambda_2(\Sigma, \bar g, \bar q)>\Lambda_1^2$ such that,  
    \begin{enumerate} [label=(\roman*)] \setcounter{enumi}{5}
      \item\label{Item_Coord_U_p cap U_q => r(q) sim r(p)} for every $p\in \Sigma$, we have $\{\tilde p\in \Sigma: U_{\tilde p}\cap U_p\neq \emptyset\}\subset \{\Lambda_2^{-1}\mfr(p)\leq \mfr \leq \Lambda_2\,\mfr(p)\}$;
      \item\label{Item_Coord_U'_p cap U'_q => q in U_p} set $U_p':= \hat\varphi_p(\BB^n_{1/\Lambda_2}\times \{0\})\subset \Sigma$, then for every $p_1, p_2\in \Sigma$ with $U_{p_1}'\cap U_{p_2}'\neq \emptyset$, we have $p_2\in U_{p_1}$.
    \end{enumerate}
  \end{lemma}
  \begin{proof}
    Since $\Sigma$ is an $\MSI$, by its very definition we may decompose it into two parts
    \[
      \Sigma = \Sigma^1 \sqcup \Sigma^2,
    \]
    where each connected component of $\Sigma^1$ lies near a singular point and is smoothly modeled on a unique minimal regular cone in $\mathbb{R}^N$, while $\Sigma^2$ is a smooth compact submanifold with smooth boundary.
For $p \in \Sigma^1$, we take a geodesic ball $B_p^\Sigma$ in $\Sigma$ centered at $p$ with radius $r_p$ proportional to $\mfr(p)$, and define $\hat U_p$ to be the $r_p$-normal tubular neighborhood of $B_p^\Sigma$. The ratio $r_p/\mfr(p)$ is chosen to be a constant determined by the minimal cone modeling the connected component containing $p$. Properties (i)--(v) hold because each component is modeled on a regular minimal cone, and there are only finitely many connected components (as $\Sing(\Sigma)$ is a finite set). For $p \in \Sigma^2$, we take a geodesic ball $B^{\bar g}_p$ in $M$ centered at $p$ with radius $r_p < \mfr(p)$ such that there exists a $2$-$C^{\bar q}$ diffeomorphism $\psi_p : (\mathbb{B}^N_\varepsilon, \mathbb{B}^n_\varepsilon) \to (B^{\bar g}_p, B^{\bar g}_p \cap \Sigma)$
    for some $\varepsilon > 0$. We then define $U_p := \psi_p(\mathbb{B}^n_{\varepsilon/2} \times \mathbb{B}^{N-n}_{\varepsilon/2})$, and let $\hat{\varphi}_p$ be a suitable rescaling of $\psi_p$. In this case, properties (i)--(v) hold since $\Sigma^2$ is a smooth compact submanifold with smooth boundary.
  Finally, by the construction above, there exists a constant $C$ such that the ratio between the maximum and minimum values of $\mfr$ in each $U_p$ is bounded by $C$. Hence, properties (vi) and (vii) follow immediately by choosing $\Lambda_2$ sufficiently large.
  \end{proof}

  By Lemma \ref{Lem_MSI w Loc Coord}, we can take a maximal collection of pairwise disjoint sets $\{U_{p_\alpha}'\}_\alpha$ of open subsets in $\{U_p': p\in \Sigma\}$, and then by \ref{Item_Coord_U'_p cap U'_q => q in U_p}, $\{U_{p_\alpha}\}_\alpha$ forms an open cover of $\Sigma$, and by \ref{Item_Coord_r bded}-\ref{Item_Coord_U_p cap U_q => r(q) sim r(p)}, $(\Sigma, \bar g|_{\Sigma}, \mfr|_{\Sigma}, \{U_\alpha:= U_{p_\alpha}, \varphi_\alpha:= \hat \varphi_{p_\alpha}|_{B^n_1\times \{0\}}\})$ satisfy properties \ref{Item_AppSetup_Sigma, g}-\ref{Item_AppSetup_Coord Bd} listed at the beginning of this appendix for some $\Lambda=\Lambda( \Sigma, \bar g, \bar q)>\Lambda_2>1$. In what follows, we fix integers \be\label{eq:IntegersRange}
  \bar q\geq 2q, \ \ \text{and} \ \  
  q-1\geq k\geq n+3\,;
  \ee
  we shall use $\mfr_\alpha, \hat\varphi_\alpha, \bar g_\alpha, Y_\alpha, u_\alpha$ in place of $\mfr(p_\alpha),\hat\varphi_{p_\alpha}, \bar g_{p_\alpha}, Y_{p_\alpha}, u_{p_\alpha}$ to save notations. Note that, by the way this cover has been selected there holds $\sum_\alpha \mfr^{\sigma}_\alpha<+\infty$ for any $\sigma>0$.

  The analytic mean curvature will be viewed as a bundle map from the 2-jet bundle. Towards that goal, we shall prioritize two bundles (namely, two triples of data as presented before): we let
  \begin{itemize}
  \item the normal bundle $T^\perp\Sigma$ be endowed with the scaling-invariant fiberwise metric $h_0$ such that $|z|_{h_0}=\mfr^{-1}|z|_{\bar g}$, and the normal connection $\bar\nabla^{(0)}\equiv \bar\nabla$;
  \item the 2-jet bundle  \[
    J^2\Sigma:= T^\perp\Sigma\oplus (T^\perp\Sigma\otimes T^*\Sigma) \oplus (T^\perp\Sigma\otimes \cT^{\otimes 2}(T^*\Sigma))\,,
  \] be endowed with the scaling-invariant fiberwise metric $h_2$: \[
    |z\oplus \xi \oplus \eta|_{h_2}^2 := \mfr^{-2}|z|_{\bar g}^2 + |\xi|_{\bar g}^2 + \mfr^2|\eta|_{\bar g}^2 \,,
  \] and with the induced connection $\bar\nabla^{(2)}$ as determined by $\bar\nabla$ (that is, from $\bar g$).
  \end{itemize}

  \begin{lemma} \label{Lem_App_X_alpha, Y_alpha C^2q Bd}
    $(T^\perp \Sigma, h_0, \bar\nabla^{(0)})$ is $C^{\bar q}$-bounded through the trivializations $Y_\alpha: (T^\perp\Sigma)|_{U_\alpha}\to \RR^{N-n}$; 
    $(J^2\Sigma, h_2, \bar\nabla^{(2)})$ is also $C^{\bar q-2}$ bounded through the following trivialization $X_\alpha: (J^2\Sigma)|_{U_\alpha} \to \RR^{N-n}\oplus \RR^{n(N-n)} \oplus \RR^{n^2(N-n)}$: for $z_1, z_2, z_3\in T^\perp \Sigma$, $\omega\in T^*\Sigma$ and $\vartheta\in \cT^{\otimes 2}(T^*\Sigma)$, \[
      X_\alpha(z_1\oplus (z_2\otimes\omega)\oplus (z_3\otimes \vartheta)):= Y_\alpha(z_1)\oplus (Y_\alpha(z_2)\otimes \varphi_\alpha^* \omega)\oplus (Y_\alpha(z_3)\otimes \varphi_\alpha^* \vartheta) \,.
    \]
    Moreover, for every integer $2\leq k\leq \bar q$ and $\tau\in \RR$, the following is a bounded linear map: 
    \begin{align}
      \mathfrak{J}^2: W^{k,2}_\tau(T^\perp \Sigma) \to W^{k-2,2}_\tau(J^2 \Sigma), \quad u\mapsto u\oplus\bar\nabla u\oplus\bar\nabla^2 u \,.  \label{Equ_App_J^2(u)=(u,du,D^2u)}
    \end{align}
  \end{lemma}
  \begin{proof}
    \eqref{Equ_App_|X_alpha(v)| sim |v|} for $Y_\alpha$ follows directly from Lemma \ref{Lem_MSI w Loc Coord} \ref{Item_Coord_Y_p:T^perp to R^(N-n)} and the definition of $h_0$. While for every $z_1, z_2, z_3\in T^\perp \Sigma$, $\omega\in T^*\Sigma$ and $\vartheta\in \cT^{\otimes 2}(T^*\Sigma)$, we have 
    \begin{align*}
      |X_\alpha(z_1\oplus (z_2\otimes\omega)\oplus (z_3\otimes \vartheta))|^2 
      & = |Y_\alpha(z_1)|^2 + |Y_\alpha(z_2)|^2|\varphi_\alpha^*\omega|^2 + |Y_\alpha(z_3)|^2|\varphi_\alpha^*\vartheta|^2 \\
      & \sim_\Lambda |z_1|_{h_0}^2 + |z_2|_{h_0}^2\cdot |\varphi_\alpha^*\omega|_{\mfr_\alpha^2\bar g_\alpha}^2\cdot \mfr_\alpha^2 + |z_3|_{h_0}^2\cdot |\varphi_\alpha^*\vartheta|_{\mfr_\alpha^2\bar g_\alpha}^2\cdot \mfr_\alpha^4 \\
      & \sim_\Lambda \mfr^{-2}|z_1|_{\bar g}^2 + |z_2\otimes \omega|_{\bar g}^2 + \mfr^2|z_3\otimes \vartheta|_{\bar g}^2 \\
      & = |z_1\oplus (z_2\otimes\omega)\oplus (z_3\otimes \vartheta)|_{h_2}^2 \,.
    \end{align*}
    This proves \eqref{Equ_App_|X_alpha(v)| sim |v|} for $X_\alpha$.

    To prove \eqref{Equ_App_Conn bded} for $Y_\alpha$ and $X_\alpha$, we pull back everything by $\varphi_\alpha$, note that by Lemma \ref{Lem_MSI w Loc Coord} \ref{Item_Coord_g_p bded},
    \begin{align*}
      (\mfr^\ell |\bar\nabla^\ell s_{\alpha, i}|_{h_0})\circ\varphi_\alpha 
      & = (\mfr^{\ell-1}|\bar\nabla^\ell s_{\alpha, i}|_{\bar g})\circ \varphi_\alpha \\
      & = \mfr^{\ell-1}\circ \varphi_\alpha\cdot |(\nabla_{\mfr_\alpha^2\bar g_\alpha})^\ell ((\hat\varphi_\alpha^{-1})_* s_{\alpha, i})|_{\mfr_\alpha^2\bar g_\alpha}) \\
      & \sim_\Lambda |(\nabla_{\bar g_\alpha})^\ell e_i|_{\bar g_\alpha} \leq C(q, \Lambda)\,,
    \end{align*}
    for every $1\leq i\leq N-n$ and $1\leq \ell \leq \bar q$; similarly,
    \begin{align*}
      (\mfr^\ell |\bar\nabla^\ell (s_{\alpha, i}\otimes dx^J)|_{h_2})\circ\varphi_\alpha 
      & = (\mfr^{\ell-1+|J|}|\bar\nabla^\ell (s_{\alpha, i}\otimes dx^J)|_{\bar g})\circ \varphi_\alpha \\
      & = \mfr^{\ell-1+|J|}\circ \varphi_\alpha\cdot |(\nabla_{\mfr_\alpha^2\bar g_\alpha})^\ell ((\hat\varphi_\alpha^{-1})_* (s_{\alpha, i}\otimes dx^J))|_{\mfr_\alpha^2\bar g_\alpha}) \\
      & \sim_\Lambda |(\nabla_{\bar g_\alpha})^\ell (e_i\otimes dx^J)|_{\bar g_\alpha} \leq C(q, \Lambda)\,,
    \end{align*}    
    for $|J|\leq 2$, $1\leq i\leq N-n$ and $1\leq \ell \leq \bar q-2$. This proves \eqref{Equ_App_Conn bded} for $Y_\alpha$ and $X_\alpha$.

    To show that the linear map $\mathfrak{J}^2$ in \eqref{Equ_App_J^2(u)=(u,du,D^2u)} is bounded, note that since $u = u_\alpha^i\cdot  s_{\alpha,i}$, we have 
    \begin{align*}
      \bar\nabla u = s_{\alpha,i}\otimes (du_\alpha^i + u_\alpha^j\,\omega_{\alpha}{}^i_j)\,; & &
      \bar\nabla^2 u = s_{\alpha,i}\otimes (\bar\nabla (du_\alpha^i + u_\alpha^j\,\omega_{\alpha}{}^i_j) + \omega_\alpha{}^i_k\otimes (du_\alpha^k + u_\alpha^j\,\omega_{\alpha}{}^k_j))
    \end{align*}
    where $\omega_\alpha{}^i_j$ is the connection form of $\bar\nabla$ under $\{s_{\alpha,i}\}$. Thus, there exist $C^{\bar q-2}$ $2$-tensors $\Omega_\alpha{}^i_j, \Theta_\alpha{}^{ik}_j$ on $B^n_1$ with $\|\omega_\alpha{}^i_j\|_{C^{\bar q-1}}, \|\Omega_\alpha{}^i_j\|_{C^{\bar q-2}}, \|\Theta_\alpha{}^{ik}_j\|_{C^{\bar q-2}} \leq C(\bar q,\Lambda)$, such that $\mathfrak{J}^2(u)$ has the following form under trivialization: 
    \begin{align}
      (\mathfrak{J}^2(u))_\alpha = u_\alpha\oplus \left(du_\alpha + u_\alpha^j\,e_i\otimes\omega_{\alpha}{}^i_j \right) \oplus \left( D^2u_\alpha + \partial_k u_\alpha^j\,e_i\otimes\Theta_{\alpha}{}^{ik}_j + u_\alpha^j\,e_i\otimes\Omega_{\alpha}{}^i_j \right)\,.  \label{Equ_App_J^2(u) loc form}
    \end{align}
    Combining this with \eqref{Equ_|v|_W^k sim sum_alpha |v_alpha|_W^k} proves the boundedness of $\mathfrak{J}^2$.
  \end{proof}

  The following Lemma gives an explicit expression of the analytic mean curvature of a normal graph under the local coordinates.
  \begin{lemma} \label{Lem_App_Loc Form of H(g,u)}
    For every $u\in C^{\bar q}_{loc}(T^\perp \Sigma)$ with $|u|\leq \Lambda^{-1}\mfr$ and every $C^q$ metric $g$ on $M$, there exists a normal section $\mathfrak{\check{h}}(g, u)\in C^{q-1}_{loc}(T^\perp \Sigma)$ such that (cf. \eqref{eq:AreaFunctNormGraphs}) \[
      \left[\frac{d}{dt}\right]_{t=0} \check A(g,u+tw) = \int_\Sigma \langle \mathfrak{\check{h}}(g, u), w\rangle_{\bar g}\ d\scH^n_{\bar g}\,, 
      \qquad \forall\, w\in C^\infty_c(T^\perp\Sigma) \,,
    \]
    where $\Sigma(u+tw)=E_{u+tw}(\Sigma)$ for the exponential map understood in metric $\bar g$ (see \ref{Item_Coord_exp(u)=(x, u_p(x))} in Lemma \ref{Lem_MSI w Loc Coord}).
    
    For every $\alpha$, $Y_\alpha\circ \mathfrak{\check{h}}(g, u)\circ \varphi_\alpha = \mathfrak{\check{h}}(g, u)_\alpha$, where $\mathfrak{\check{h}}(g, u)_\alpha: \BB^n_1\to \RR^{N-n}$ is given by 
     \begin{align}
     \begin{split}
      & \mathfrak{\check{h}}(g, u)_\alpha(x) \\
      & \quad = \mfr_\alpha^{-2} \left[ - \bar G_\alpha(x)^{-1} \partial_i\left(\bar G_\alpha(x) \cdot \partial_{\xi_i}f_{g,\alpha}(x, u_\alpha(x), du_\alpha(x)) \right) 
      + \partial_z f_{g,\alpha}(x, u_\alpha(x), du_\alpha(x))\right]\,;
     \end{split} \label{Equ_App_MeanCurv_LocDef}
    \end{align}
    and $f_{g, \alpha} = f_{g, \alpha}(x, z, \xi)\in C^q_{loc}(\BB^n_1\times \BB^{N-n}_1\times \RR^{n(N-n)})$ is determined by
    \begin{align}
      f_{g, \alpha}(x, u_\alpha(x), du_\alpha(x)) = \bar G_\alpha(x)^{-1} \det[(E_{u,\alpha}{}^*g_\alpha)_{ij}]^{1/2}\,, \label{Equ_App_Area Density_LocDef}
    \end{align}
    with $\bar G_\alpha(x):= \det^{1/2}[(\bar g_\alpha)_{ij}]$, $E_{u,\alpha}(x):=(x, u_\alpha(x))\in \RR^N$, $g_\alpha:= \mfr_\alpha^{-2}\hat\varphi_\alpha{}^*g$.
  \end{lemma}
  \begin{proof}
    Recall that, by the area formula 
\[
    \check A(g,u)
= \int_\Sigma d\mu_{E_u^*g}
= \int_\Sigma f_g[u]\,d\scH^n_{\bar g}\,,
\]
    where $f_g[u]\in C^q_{loc}(\Sigma;\R)$ and in local charts there holds
    \begin{align}
      f_g[u]\circ \varphi_\alpha = \det[(\varphi_\alpha{}^*\bar g)_{ij}]^{-1/2}\det[((E_u\circ \varphi_\alpha)^*g)_{ij}]^{1/2} = f_{g, \alpha}(\cdot, u_\alpha, du_\alpha)\,,  \label{Equ_App_Area density}
    \end{align}
    here the second equation follows from Lemma \ref{Lem_MSI w Loc Coord} \ref{Item_Coord_g_p bded}, \ref{Item_Coord_exp(u)=(x, u_p(x))}. Hence, if $\spt w\subset U_\alpha$, then
    \begin{align*}
      \left[\frac{d}{dt}\right]_{t=0} \check A(g,u+tw) 
      & = \int_{\BB^n_1} \left[\frac{d}{dt}\right]_{t=0} f_g[u+tw]\circ \varphi_\alpha\ d\scH^n_{\varphi_\alpha^*\bar g} \\
      & = \int_{\BB^n_1} \left(\langle \partial_z f_{g, \alpha}, w_\alpha\rangle_{g_\eucl} + \sum_i\langle \partial_{\xi_i}f_{g,\alpha}, \partial_i w_\alpha \rangle_{g_\eucl}\right)\ \mfr_\alpha^n \bar G_\alpha d\scH^n \\
      & = \int_{\BB^n_1}\mfr_\alpha^2\langle \mathfrak{\check{h}}(g, u)_\alpha, w_\alpha\rangle_{g_\eucl}\ \mfr_\alpha^n \bar G_\alpha d\scH^n \\
      & = \int_{\Sigma} \langle \mathfrak{\check{h}}(g, u), w\rangle_{\bar g}\ d\scH^n_{\bar g} \,.
    \end{align*}
    Here the second equality follows from \eqref{Equ_App_Area density} and Lemma \ref{Lem_MSI w Loc Coord} \ref{Item_Coord_g_p bded}, the third equality follows from the definition of $\mathfrak{\check{h}}(g, u)_\alpha$ and the divergence theorem, and the last equality follows from Lemma \ref{Lem_MSI w Loc Coord} \ref{Item_Coord_Y_p:T^perp to R^(N-n)}, \ref{Item_Coord_exp(u)=(x, u_p(x))}, which gives the local expression of $\mathfrak{\check{h}}(g, u)$. Since from the explicit definition \eqref{Equ_App_Area Density_LocDef} (cf. also \cite[Appendix A]{CLW24}), $\partial_zf_{g,\alpha}\in C^{q-1}, \partial_{\xi}f_{g,\alpha}\in C^{q}$, we have indeed $\mathfrak{\check{h}}(g, u)_\alpha \in C^{q-1},\ \mathfrak{\check{h}}(g, u)\in C_{loc}^{q-1}$.
  \end{proof}

  For the following statement, and subsequent proofs recall from the main body of the article that $\Gamma(q; \bar g,\Sigma)$ denotes the set of $C^q$ metrics which coincide with $\bar g$ on the set $\Sing(\Sigma)$ as per Definition \ref{def:MetricsCoincide}.

  \begin{corollary} \label{Cor_App_cH_g(J^2u)=H_G(g,u)}
    For every $\tau< 1$, there exists a $C^{\bar q-q-1}$ map \[
      \cH: \Gamma(q; \bar g,\Sigma) \to \scC^{q-1}_{\tau, \tau-2}(J^2\Sigma, T^\perp \Sigma)\,, \quad g\mapsto \cH_g\,,
    \] 
    such that for every $g\in \Gamma(q; \bar g,\Sigma)$ and $u\in W^{k,2}_{loc}(T^\perp \Sigma)$ for $k$ in the range \eqref{eq:IntegersRange} with $|\frak{J}^2u|_{h_2}<\Lambda^{-1}$, we have in $W^{k-2,2}_{loc}$ 
    \be\label{eq:Factorization}
      \cH_g\circ\mathfrak{J}^2(u) = \mathfrak{\check{h}}(g, u) \,.
    \ee
  \end{corollary}
  \begin{proof}
    The existence of such a $C^{q-1}_{loc}$ bundle map $\cH_g$ and the fact that \eqref{eq:Factorization} holds follow from Lemma \ref{Lem_App_Loc Form of H(g,u)} (the latter via an approximation argument). We now work under local trivializations to bound its norm. By Lemma \ref{Lem_App_Loc Form of H(g,u)}, 
    \begin{align*}
      (\mathfrak{\check{h}}(g, u))_\alpha(x) = \mfr_\alpha^{-2} 
      \Big[ \breve K_1\left(x, u_\alpha, du_\alpha, J^{1}g_\alpha|_{(x, u_\alpha(x))}\right)\cdot D^2 u_\alpha 
      + \breve K_2\left(x, u_\alpha, d u_\alpha, J^{1}g_\alpha|_{(x, u_\alpha(x))}\right) \Big]\,,
    \end{align*}
    for some $\breve K_i\in C^{\bar q-1}(\BB_1^n\times \BB_1^{N-n}\times \BB_1^{n(N-n)}\times \RR^{(N\times N)(1+N)})$ depending only on $M, \bar g, \Sigma, \hat\varphi_\alpha$ (but not on $g$), with $\|\breve K_i\|_{C^{\bar q-1}}\leq C(\Lambda, q)$ (this is because by \eqref{Equ_App_MeanCurv_LocDef}, $\breve K_i$ involves at most first order derivatives of $\bar G_\alpha$, which is $C^{\bar q-1}$ bounded since $\bar g_\alpha$ is $C^{\bar q}$ equivalent to the Euclidean metric by Lemma \ref{Lem_MSI w Loc Coord} \ref{Item_Coord_g_p bded}). Here $J^{1}g_\alpha:= g_\alpha\oplus Dg_\alpha$. 
    
    Combined with \eqref{Equ_App_J^2(u) loc form} we find the local form of $\cH_g$: \[
      (\cH_g)_\alpha(x, z\oplus \xi\oplus\eta) = \mfr_\alpha^{-2}\left[ K_1(x, z, \xi, J^{1}g_\alpha|_{(x, z)})\cdot \tilde\chi(\eta) + K_2(x, z, \xi, J^{1}g_\alpha|_{(x, z)}) \right] \,.
    \]
    for some $K_i\in C^{\bar q-1}(\BB_1^n\times \BB_1^{N-n}\times \BB_1^{n(N-n)}\times \RR^{(N\times N)(1+N)})$ also depending only on $M, \bar g, \Sigma, \hat\varphi_\alpha$ (but not on $g$), with $\|K_i\|_{C^{\bar q-1}}\leq C(\Lambda, \bar q)$. Here $\tilde\chi:\R^{n^2(N-n)}\to \R^{n^2(N-n)}$ is a smooth cut-off, equals to $\id$ on the unit ball and vanishes on a slightly larger ball. This is simply to make $(\cH_g)_\al$ $C^q$ bounded in the whole $\R^{n^2(N-n)}$ factor, and it doesn't affect \eqref{eq:Factorization} since $|\frak{J}^2u|_{h_2}<1$ Thus, by the regularity of composition map (see the proof in \cite[Appendix]{Whi87}, one can also prove it using similar argument as Lemma \ref{Lem_Differentiable}), $g_\alpha \mapsto \mfr_\alpha^2\cdot (\cH_g)_\alpha$ is a $C^{(\bar q-1) - (q-1) - 1 = \bar q-q-1}$ map from the space of $C^q$-metrics on $\BB_1^n\times \BB_1^{N-n}$ to $C^{q-1}(\BB_1^n\times \BB_1^{N-n}\times \BB_1^{n(N-n)}\times \RR^{n^2(N-n)})$ with uniform derivative bounds.  
    Moreover, since $\Sigma$ is $\bar g$ minimal, we have $(\cH_{\bar g})_\alpha(\cdot, 0) = (\mathfrak{\check{h}}(\bar g, 0))_\alpha = 0$, and so 
    \begin{align*}
      \|(\cH_g)_\alpha(\cdot, 0)\|_{C^{q-1}} 
      & = \|((\cH_g)_\alpha - (\cH_{\bar g})_\alpha)(\cdot, 0)\|_{C^{q-1}}
      \lesssim_{\Lambda, q} \mfr_\alpha^{-2} (\|K_1\|_{C^{q}} + \|K_2\|_{C^{q}})\|g_\alpha-\bar g_\alpha\|_{C^{q}} \\
      & \sim_{\Lambda, q, \Sigma, \bar g} \mfr_\alpha^{-4} \left(\sum_{\ell=0}^{q} \mfr_\alpha^{\ell+2}\|\bar\nabla^\ell (\mfr_\alpha^2g_\alpha - \mfr_\alpha^2\bar g_\alpha)\|_{C^0, \mfr_\alpha^2\bar g_\alpha} \right) \\
      & \lesssim_{\Lambda, q, \Sigma, \bar g} \mfr_\alpha^{-1} \left(\sum_{\ell=0}^{q} \|\mfr^{\ell-1}\bar\nabla^\ell (g - \bar g)\|_{C^0(\hat U_\alpha), \bar g} \right) \\
      & \lesssim_{\Lambda, q, \Sigma, \bar g} \mfr_\alpha^{-1}\|g-\bar g\|_{C^{q}, \bar g} \,.
    \end{align*}
   Here in the last inequality, we use the fact that $q\geq 1$ and $\|\mfr^{-1}(g-\bar g)\|_{C^0(\hat U_\alpha), \bar g}\lesssim \|g-\bar g\|_{C^1(\hat U_\alpha), \bar g}$ since $g$ coincides with $\bar g$ on $\Sing(\Sigma)$, and $\mfr$ is the distance function to $\Sing(\Sigma)$.
    We can then estimate $\|\cH_g\|_{\scC^{q-1}_{\tau, \tau-2}}$ by the definition \eqref{Equ_Def_scC^q bundle map} and \eqref{Equ_|v|_W^k sim sum_alpha |v_alpha|_W^k}, 
    \begin{align*}
      \|\cH_g\|_{\scC^{q-1}_{\tau, \tau-2}}
      & = \sup_\alpha \mfr_\alpha^2 \|(\cH_g-\cH_{\bar g})_\alpha\|_{C^{q-1}} + \|(\cH_{\bar g})_\alpha\|_{C^{q-1}} + \|\cH_g(0)\|_{W^{q-1,2}_{\tau-2}} \\
      & \lesssim_{\Lambda, q, \Sigma, \bar g} \|g-\bar g\|_{C^{q}} + 1 + \left(\sum_\alpha \mfr_\alpha^{-2(\tau-2)}\|(\cH_g)_\alpha(\cdot, 0)\|_{W^{q-1,2}}^2 \right)^{1/2} \\
      & \lesssim_{\Lambda, q, \Sigma, \bar g} \|g-\bar g\|_{C^{q}} (1+(\sum_\alpha \mfr_\alpha^{2-2\tau})^{1/2}) + 1 \quad
        \lesssim_{\Lambda, q, \Sigma, \bar g, \tau} \|g-\bar g\|_{C^{q}} + 1 \,.
    \end{align*}
    Here in the last inequality we use that $\tau< 1$ and hence $\sum_\alpha \mfr_\alpha^{2-2\tau}<+\infty$. We thus conclude that $\cH_g$ is a bundle map in $\scC^{q-1}_{\tau, \tau-2}$. Finally, recalling in \eqref{Equ_Def_scC^q bundle map}, $\scC^{q-1}_{\tau, \tau-2}$-norm is defined by the supremum of $C^{q-1}$-norm over every local trivialization parametrized by $\alpha$, together with the global $W^{q-1,2}_{\tau-2}$-norm, we see that $g\mapsto \cH_g$ is a $C^{\bar q-q-1}$-map from
    $\Gamma(q; \bar g,\Sigma)$
    to $\scC^{q-1}_{\tau, \tau-2}(J^2\Sigma, T^\perp\Sigma)$.
  \end{proof}

  \begin{proof}[Proof of Proposition \ref{Prop:C^qMSO}.]
    The existence of such an $\mathfrak{\check{h}}$ has been proved in Lemma \ref{Lem_App_Loc Form of H(g,u)}; that being said, by Lemma \ref{Lem_|psi(f_1-f_2)|<|psi||f_1-f_2|}, Lemma \ref{Lem_App_X_alpha, Y_alpha C^2q Bd} and Corollary \ref{Cor_App_cH_g(J^2u)=H_G(g,u)}, for $\tau\in [0, 1)$, $n+3\leq k\leq q-1$, $g\in \Gamma(q; \bar g,\Sigma)$ and $u\in W^{k,2}_\tau(T^\perp\Sigma)$, we have that $\mathfrak{\check{h}}$ is the composition of the following maps \[
      \met^{q}(\bar g, \Sigma) \times W^{k,2}_\tau(T^\perp\Sigma) \xrightarrow{\cH\times \mathfrak{J}^2} \scC^{q-1}_{\tau, \tau-2}(J^2\Sigma, T^\perp\Sigma)\times W^{k-2, 2}_{\tau}(J^2\Sigma) \xrightarrow{\text{compose}} W^{k-2,2}_{\tau-2}(T^\perp\Sigma) 
    \]
    where $\cH$ is $C^{q-k}$, $\mathfrak{J}^2$ is bounded linear (thus differentiable as many times as we like), and the composition map is $C^{q-k}$ by Lemma \ref{Lem_Differentiable}. Hence, $\mathfrak{\check{h}}$ is a $C^{q-k}$ map.
    Finally by definition of $h_0$, note that $\|\cdot\|_{W^{k,2}_\tau(T^\perp\Sigma)}$ (with respect to $h_0$ like we had stipulated) is equivalent to $\|\cdot\|_{W^{k,2}_{\tau+1}(T^\perp\Sigma)}$ as determined by  $\bar g$ (without the rescaling), which is nothing but the usual norm $\|\cdot\|_{W^{k,2}_{\tau+1}(\Sigma;\mbfV)}$ that we defined and employed throughout the main body of the paper. This proves Proposition \ref{Prop:C^qMSO} for the admissible range of $\tau$ therein.
  \end{proof}

\section{The space of Veronese footballs in $\SSp^5$}\label{sec:topology}

\subsection{Veronese embeddings and the Berger space}\label{subs:Veronese}

Here we first collect some ancillary results about the Veronese embedding of $\mathbb{RP}^2$ in the round four-dimensional sphere and the full moduli space of all such embeddings, that is the so-called Berger space.

The following facts are well-known, see \cite{Wolf68, GoeKitSha04} as well as \cite[Section 2]{BalMad23} for a recent exposition; here we keep our presentation as concrete and simple as possible. 

Let 
\be\label{eq:MatrixTraceSpace}
Z_1:=\left\{ A\in \sym(\R^3)\ \st \ \tr(A)=1\right\}, \text{and} \ Z_0:=\left\{ A\in \sym(\R^3)\ \st \ \tr(A)=0\right\}
\ee
and note that there is a natural action, by conjugation, of $\SOgroup(3;\R)$ on $Z_1$
that is obtained by restriction of the action on the space of all $3\times 3$ matrices. 

\begin{lemma}\label{lem:SymVer}
\begin{enumerate}
\item{The map $\Omega_1: \SSp^2\to Z_1$ given by
\[
(x_0, x_1, x_2)\mapsto \begin{pmatrix}
x^2_0 & x_0x_1 & x_0x_2 \\
x_0x_1 & x^2_1 & x_1x_2 \\
x_0x_2 & x_1 x_2 & x^2_2
\end{pmatrix}
\]
descends to a well-defined embedding of $\mathbb{R}\mathbb{P}^2$ in $Z_1$ (not renamed), whose image is exactly equal to the matrices $A\in Z_1$ satisfying $A^2=A$;}
\item{The linear map $\Omega_2: Z_0\to \R^5$ given by
\[
\begin{pmatrix}
d_1 & a_1 & a_2 \\
a_1 & d_2 & a_3 \\
a_2 & a_3 & d_3
\end{pmatrix} \mapsto \left (\sqrt{3}a_1, \sqrt{3}a_2,\sqrt{3}a_3, \frac{\sqrt{3}}{2}(d_1-d_2), \frac{1}{2}(2d_3-d_1-d_2)\right)
\]
is an isometric isomorphism with respect to the rescaled Frobenius norm on the domain $\|\cdot\|_0:=\sqrt{2/3}\cdot \text{Frobenius norm}$,
and the standard Euclidean norm on the target;}
\item{the group $\SOgroup(3;\R)$ acts by isometries on $Z_0$ with respect to the norm defined in the preceding item, namely there holds $\|h\cdot A\|_0^2=\|A\|_0^2$ for all $A\in Z_0$ and $h\in \SOgroup(3;\R)$, hence it acts isometrically on $\R^5$ by declaring
\[
h\cdot v:=h\cdot \Omega_2^{-1}(v)\,, \ \ (v\in \R^5)
\]
so that, in particular, we can regard $\SOgroup(3;\R)$ as a subgroup (henceforth denoted by $H$) of $G:=\SOgroup(5;\R)$;}
\item{with respect to the embedding described in the previous item $H < G$ is a maximal subgroup (namely: it is not contained in any proper subgroup of $\SOgroup(5;\R)$);}
\item\label{item:DefVeroneseMap}{define the Veronese surface as the image of the composite map 
\[
\begin{tikzcd}
\R\mathbb{P}^2
   \arrow[r, "\Omega_1"]
   \arrow[rrr, bend right=20, "\Psi:=\Omega_2\cdot\,\Omega\cdot\,\Omega_1"']
& Z_1
   \arrow[r, "\Omega"]
& Z_0
   \arrow[r, "\Omega_2"]
& (\SSp^4 \subset{}) \ \mathbb{R}^5
\end{tikzcd}
\]

where $\Omega: Z_1\to Z_0$ is given by $A\mapsto A-\rm{diag}(1/3)$. Then the subgroup of $\Ogroup(5;\R)$ leaving
 the Veronese surface invariant coincides with $H$ and is thus isomorphic to the group $\SOgroup(3;\R)$.}
\end{enumerate} 
\end{lemma}

Later in this section we shall also denote the subgroup $H$ in this
lemma by $\SOgroup(3;\R)_{\mathrm{irr}}$,
since its defining representation is the
irreducible five-dimensional representation
of $\SOgroup(3;\R)$ on the space $Z_0$ of
traceless symmetric matrices, acting by conjugation.

\

The quotient $G/H$ in the previous statement is known in the literature as \emph{Berger space} and will henceforth be denoted here by $B^7$; some of its topological properties are collected in the next statement as they serve an input for the study of the space of Veronese footballs in $\SSp^5$.

\begin{lemma}\label{lem:Berger}
The Berger space $G/H$ is a compact, $2$-connected smooth $7$-manifold (diffeomorphic to an $S^3$-bundle over $S^4$), that is a rational homology sphere and satisfies, in integral cohomology
\[
H^k(B^7;\mathbb{Z})=\begin{cases}
\Z & \text{if} \ k=0,7 \\
\Z_{10} & \text{if} \ k=4 \\
0 & \text{if} \ k\neq 0,4,7\,.
\end{cases}
\]
Furthermore, its mod-$2$ cohomology ring is
\[
H^*(B^7;\mathbb Z_2)\cong \mathbb Z_2[x_3,x_4]/(x_3^2,x_4^2),
\qquad |x_3|=3,\ |x_4|=4,
\]
with $x_3x_4=\omega_7\neq 0$ the mod-$2$ fundamental class.
\end{lemma}

\begin{remark}\label{rem:VeroneseOtherFields} 
The previous definition generalizes to the other division algebras as follows. For $\mathbb{K}=\R, \C, \mathbb{H}$ and $\mathbb{O}$, set $m=\dim_{\R}(\mathbb{K})$ and consider the embedding map 
  $\mathbb{K}\mathbb{P}^2 \to \SSp^{3m+1} \subset \RR^{3m+2} = \mathbb{K}^3\times \RR^2$ defined by
  \[
[u: v: w]\mapsto (|u|^2+|v|^2+|w|^2)^{-1}\cdot \left(\sqrt3 v\bar w, \sqrt3 w\bar u, \sqrt3 u\bar v, \frac{\sqrt{3}}{2}(|u|^2 - |v|^2),\frac{1}{2}(2|w|^2 - |u|^2 - |v|^2)\right)
  \]
  In the case $\mathbb{K} = \mathbb{O}$, where multiplication is non-associative, the representatives $u, v, w$ of each point in $\mathbb{O}\mathbb{P}^2$ lie in a common \emph{associative} subalgebra -- for instance, by requiring that one of $u, v, w$ equals $1$.
  
  This gives the (previous) classical Veronese minimal $\mathbb{R}\mathbb{P}^2$ in $\SSp^4$, a minimal $\C\mathbb{P}^2$ in $\SSp^7$, a minimal $\mathbb{H}\mathbb{P}^2$ in $\SSp^{13}$ and a minimal $\mathbb{O}\mathbb{P}^2$ in $\SSp^{25}$.  None of the four projective planes in question bounds a smooth manifold (of real dimension, respectively, $3,5,9$ and $17$) since they have odd Euler characteristic. 
\end{remark}

The Morse index and nullity of the Veronese $\mathbb{KP}^2$ have recently been computed by Rauchenberger and Semmelmann in~\cite{RauSem26} relying upon their relations to isoparametric hypersurfaces in round spheres.

\begin{lemma}[{\cite[Theorem~1]{RauSem26}}]\label{lem:RauSem}
    In the setting of Remark \ref{rem:VeroneseOtherFields}, for the Veronese embeddings of $\mathbb{KP}^2 \subset \mathbb{S}^{3m+1}$  we have
    \begin{itemize}
        \item $\ind(\mathbb{KP}^2) = 3m+2$;
        \item $\nul(\mathbb{RP}^2) = 7$, $\nul(\mathbb{CP}^2) = 20$, $\nul(\mathbb{HP}^2) = 70$, and $\nul(\mathbb{OP}^2) = 273$.
    \end{itemize}

    In particular, in any of the cases above, all Jacobi fields are integrable and determined by the action of a Lie subgroup of the ambient orthogonal group of the appropriate dimension.
\end{lemma}

\begin{remark}\label{rem:VeroneseIndexIdentification}
It follows from the preceding result that the subspace $V_{-}\subset L^2(\Sigma_0;\mbfV_0)$ where the Jacobi operator is negative definite is given by the eigensections of the form
$v^{\perp}_i$ for any $v\in\R^{3m+2}$; based on a well-known computation (due to Simons \cite{Sim68}) each such section has eigenvalue equal to $-2m$.
As a result (thanks to \cite[Lemma 6.1.6]{Sim68}), the cone $\mbfC_{\mathbb{K}}$ in $\R^{3m+2}$ having the Veronese embedded $\mathbb{KP}^2$ as its link is stable (and so the Veronese football $\Sigma_{\mathbb{K}}$ is a locally stable stationary varifold).
Note, also, that (in the notation of \eqref{eq:UpperThreshold}, cf. \cite[Section 2.2]{CLW24}) for $\Sigma_{\mathbb{K}}$ there holds 
$\gamma_{-}(\mbfC)=0.$
\end{remark}

Let now $\Sigma_0 \subset \mathbb{S}^4$ be the Veronese embedding of $\mathbb{RP}^2$ as defined in Lemma \ref{lem:SymVer}, and let $\Sigma \subset \mathbb{S}^5$ denote its spherical suspension, informally referred to as \textit{Veronese football}. This is an example of $\MSI$ as per Definition \ref{def:MSI}, with two singular points at the north and south poles. 

Differently from the notation employed within Lemma \ref{lem:SymVer}, from now onward we set $G:=\Iso(\mathbb{S}^5)=\Ogroup(6;\R)$ and $H:=\Stab_G(\Sigma)$. Every element of $H$ either fixes the two poles or interchanges them, so it has the block form
\[
  \begin{pmatrix}
    A & 0\\
    0 & \pm 1
  \end{pmatrix},
\]
where $A\in\Ogroup(5;\R)$ preserves $\Sigma_0\subset \mathbb{S}^4$. 
Hence, by Lemma \ref{lem:SymVer},
$A\in\SOgroup(3;\R)_{\mathrm{irr}}$.

If we let
\[
\rho: \SOgroup(5;\mathbb{R}) \longrightarrow \SOgroup(6;\mathbb{R}), \quad A \longmapsto \mathrm{diag}(A,1)\,
\]
denote the standard (top left) block embedding,
then there holds the diffeomorphism
\[
G/H \cong \SOgroup(6;\mathbb{R}) \big/ \rho\big(\SOgroup(3;\mathbb{R})_{\mathrm{irr}}\big).
\]
Recalling that the transitive action of $\SOgroup(6;\mathbb{R})$ on $\mathbb{S}^5$ induces a principal $\SOgroup(5;\mathbb{R})$-bundle
\[
\SOgroup(5;\mathbb{R}) \hookrightarrow \SOgroup(6;\mathbb{R}) \longrightarrow \mathbb{S}^5
\]
and keeping in mind the definition of the homogeneous space
\[
B^7 := \SOgroup(5;\mathbb{R}) \big/ \SOgroup(3;\mathbb{R})_{\mathrm{irr}}\,,
\]
the space of Veronese footballs in $\mathbb{S}^5$ can be identified with the associated bundle
\[
E := \SOgroup(6;\mathbb{R}) \times_{\SOgroup(5;\mathbb{R})} B^7,
\]
which fits into a fiber bundle
\[
B^7 \hookrightarrow E \longrightarrow \mathbb{S}^5.
\]

\begin{proposition}\label{prop:E_2_conn}
    $E$ is $2$-connected.
\end{proposition}
\begin{proof}
    Since $B^7$ and $\mathbb{S}^5$ are $2$-connected, it follows from the long exact sequence
    \[
        \dots \to \pi_i(B^7) \to \pi_i(E) \to \pi_i(\mathbb{S}^5) \to \pi_{i - 1}(B^7) \to \cdots
    \]
    that the total space $E$ is also $2$-connected.
\end{proof}

\subsection{The cohomology groups and the cohomology ring structure}We compute the integral cohomology groups and
the mod-$2$ cohomology ring of $E$ using the
Serre spectral sequence of the fibration
\[
B^7\hookrightarrow E\longrightarrow \SSp^5.
\] \S
For the homology groups, see also
\cite[Application 5.3.5]{Wei94}.

\begin{proposition}\label{prop:int_coh_E}
The integral cohomology groups of $E$ are
\[
  H^k(E;\mathbb Z)\cong
  \begin{cases}
    \mathbb Z,& \text{if} \ k=0,5,7,12,\\
    \mathbb Z_{10},& \text{if} \ k=4,9,\\
    0,&\text{otherwise}.
  \end{cases}
\]
\end{proposition}
\begin{proof}
Consider the Serre spectral sequence for
$B^7\hookrightarrow E\to \SSp^5$ with integral
coefficients. Since $\SSp^5$ is simply connected,
\[
E_2^{p,q}
=H^p(\SSp^5;\mathbb Z)\otimes H^q(B^7;\mathbb Z).
\]
Its only nonzero terms are
\[
\begin{aligned}
E_2^{0,0}&=\mathbb Z,
& E_2^{0,4}&=\mathbb Z_{10},
& E_2^{0,7}&=\mathbb Z,\\
E_2^{5,0}&=\mathbb Z,
& E_2^{5,4}&=\mathbb Z_{10},
& E_2^{5,7}&=\mathbb Z.
\end{aligned}
\]
Hence, the only potentially nonzero differential is
\[
d_5:E_5^{0,4}\longrightarrow E_5^{5,0}.
\]
It vanishes because every homomorphism from
$\mathbb Z_{10}$ to $\mathbb Z$ is zero.
Thus all differentials vanish, and the spectral
sequence collapses at the $E_2$-page.
Since each total degree contains at most one
nonzero associated graded group, there are no
additive extension problems, and the asserted
cohomology groups follow.
\end{proof}

Let $u\in H^5(\SSp^5;\mathbb Z_2)$ be the generator, and let $x_3\in H^3(B^7;\mathbb Z_2)$, $x_4\in H^4(B^7;\mathbb Z_2)$ be generators such that
$x_3x_4 \in H^7(B^7;\mathbb Z_2)$ is the mod-$2$ fundamental class (cf. Lemma \ref{lem:Berger}).

\begin{proposition}\label{prop:Z2_coh_ring_E}
As a graded $\mathbb Z_2$-algebra, there holds
\[
  H^*(E;\mathbb Z_2)
  \cong
  \mathbb Z_2[u,x_3,x_4]/(u^2,x_3^2,x_4^2),
  \qquad |u|=5,\ |x_3|=3,\ |x_4|=4.
\]
In particular, a basis is
\[
  1,\ x_3,\ x_4,\ u,\ x_3x_4,\ ux_3,\ ux_4,\ ux_3x_4.
\]
\end{proposition}
\begin{proof}
Consider the Serre spectral sequence for
$B^7\hookrightarrow E\to \SSp^5$ with
$\mathbb Z_2$ coefficients. Since $\SSp^5$ is
simply connected,
\[
E_2^{p,q}
=H^p(\SSp^5;\mathbb Z_2)\otimes
 H^q(B^7;\mathbb Z_2).
\]
Its only nonzero terms are
\[
\begin{aligned}
E_2^{0,0}&=\mathbb Z_2,
& E_2^{0,3}&=\langle x_3\rangle,
& E_2^{0,4}&=\langle x_4\rangle,
& E_2^{0,7}&=\langle x_3x_4\rangle,\\
E_2^{5,0}&=\langle u\rangle,
& E_2^{5,3}&=\langle ux_3\rangle,
& E_2^{5,4}&=\langle ux_4\rangle,
& E_2^{5,7}&=\langle ux_3x_4\rangle.
\end{aligned}
\]
The only potentially nonzero differential is
$d_5:E_5^{0,q}\to E_5^{5,q-4}$, and $E_5=E_2$.
We have $d_5x_3=0$ for degree reasons.
Moreover, $x_4$ is the reduction of an integral
class in $H^4(B^7;\mathbb Z)\cong\mathbb Z_{10}$.
Naturality under reduction of coefficients,
together with the vanishing of the integral
differential, gives $d_5x_4=0$.
The derivation rule then gives
$d_5(x_3x_4)=0$.
Thus, as in the previous proof, the spectral sequence collapses at the $E_2$-page.

Write again $u$ for the pullback to $E$ of the
generator of $H^5(\SSp^5;\mathbb Z_2)$, and choose
classes in $H^3(E;\mathbb Z_2)$ and
$H^4(E;\mathbb Z_2)$ restricting to $x_3,x_4$;
we denote these lifts by the same symbols.
By multiplicativity of the spectral sequence,
\[
1,\ x_3,\ x_4,\ u,\ x_3x_4,\ ux_3,\ ux_4,\
ux_3x_4
\]
represent the eight nonzero associated graded
classes and therefore form an additive basis.

It remains to verify the square relations.
We have $u^2=0$ because $u$ is pulled back from
$\SSp^5$, and $x_3^2=0$ because
$H^6(E;\mathbb Z_2)=0$.
The integral spectral sequence shows that
restriction induces an isomorphism
\[
H^4(E;\mathbb Z)\longrightarrow
H^4(B^7;\mathbb Z).
\]
We may therefore choose $x_4$ to be the reduction
of an integral class $a\in H^4(E;\mathbb Z)$.
By Proposition \ref{prop:int_coh_E},
$H^8(E;\mathbb Z)=0$, so $a^2=0$ and hence
$x_4^2=0$.
These relations and the displayed additive basis
give the claimed ring presentation.
\end{proof}

\subsection{A handle decomposition of $E$ and Morse number}

We define the \emph{Morse number} of a manifold $M$ as the least number of critical points of a smooth \underline{Morse} function $f:M\to\R$, denoted as $MN(M)$.

\begin{proposition}\label{prop:handle_decomp_E}
  The closed $12$-manifold $E$ admits a Morse function with minimal critical points, which induces a handle decomposition with exactly eight handles,
  one in each dimension
  \[
    0,3,4,5,7,8,9,12.
  \]
  In particular, this determines a CW complex structure on $E$.
\end{proposition}
\begin{proof}
    By Proposition~\ref{prop:int_coh_E} and the Poincar\'e duality, the integral homology groups of $E$ are
    \[
      H_k(E;\mathbb Z)\cong
      \begin{cases}
        \mathbb Z,&k=0,5,7,12,\\
        \mathbb Z_{10},&k=3,8,\\
        0,&\text{otherwise}.
      \end{cases}
    \]

    Note that $E$ has dimension greater than $5$ and, by Proposition~\ref{prop:E_2_conn}, is simply connected. By \cite[Theorem~6.1]{Sma62} and the more explicit conclusion in 6.2 therein, there exists a Morse function, the number of whose critical points of index $i$, $M_i$, is given by
    \[
        M_i = \begin{cases}
            1, & i = 0, 3,4,5,7,8,9,12\\
            0, & \text{otherwise.}
        \end{cases}
    \]
    
    Conversely, by the weak Morse inequalities,
every Morse function on $E$ has at least
\[
\sum_{i=0}^{12}\dim H^i(E;\mathbb Z_2)=8
\]
critical points, where the equality follows
from Proposition \ref{prop:Z2_coh_ring_E}.
Thus the constructed Morse function realizes
the Morse number.

    The handle decomposition thus follows from standard facts in Morse theory.
\end{proof}

\begin{corollary}\label{cor:MorseNumber}
    $MN(E) = 8$.
\end{corollary}

\subsection{Lusternik-Schnirelmann category}

Recall that the \emph{Lusternik-Schinrelmann category} $\operatorname{cat}(X)$ of a topological space $X$ is the smallest integer $k$ that $X$ can be covered by $k + 1$ open sets, each of which is contractible in $X$. (We note that there are, about the preceding one and the notions in the next three subsections, various pairwise incompatible definitions in the literature; here we will stick to the monograph \cite{CLOT03}.)

By the ring structure of $H^*(E;\Z_2)$ in Proposition~\ref{prop:Z2_coh_ring_E}, the $\mathbb{Z}_2$-cup length of $E$ is $\operatorname{cl}(E; \Z_2) = 3$. Hence, by~\cite[Proposition~1.5]{CLOT03}, we have a lower bound for $\operatorname{cat}(E)$,
\[
    \operatorname{cat}(E) \geq \operatorname{cl}(E; \Z_2) = 3\,.
\]

The matching upper bound will follow from the
estimate for the Lusternik--Schnirelmann number below.

\subsection{Lusternik-Schnirelmann number}

We define the \emph{Lusternik-Schnirelmann number} of a manifold $M$ as the least number of critical points of a smooth function $f: M \to \mathbb{R}$, denoted as $LSN(M)$.

By the well-known Lusternik-Schnirelmann theorem~\cite[Theorem~1.15]{CLOT03}, where the quantity is denoted by $\operatorname{Crit}(E)$, we have
\[
    LSN(E) \geq \operatorname{cat}(E) + 1 \geq 4\,.
\]

To get an upper bound for $LSN(E)$, we make use of the following result by Takens.

\begin{theorem}[{\cite[Theorem~5.1]{Tak68} (See also~\cite[Proposition~7.28(i)]{CLOT03})}]
    Let $M$ be a compact $p$-connected closed manifold of dimension $n \geq 6$ and $p \geq 1$. Let $r(M) \subset \mathbb{Z}$ be the smallest subset containing all $i \in \mathbb{Z}$ if $H_i(M; \Z) \neq 0$ or $H^i(M; \Z) \neq 0$. If $r(M)$ is included in the union of $s$ closed intervals of length $p$, then $LSN(M) \leq s$.
\end{theorem}

By Proposition~\ref{prop:E_2_conn}, $E$ is $2$-connected. By Proposition~\ref{prop:int_coh_E} and Poincar\'e duality, $r(E) = \{0, 3, 4, 5, 7, 8, 9, 12\}$, which can be covered by $4$ intervals $[0,2], [3, 5], [7,9]$ and $[10,12]$ of length $2$. 

It follows from the previous theorem that
$LSN(E)\leq4$.
Combining this with the preceding lower bounds gives
\[
3\leq\operatorname{cat}(E)\leq LSN(E)-1\leq3.
\]
Hence we reach the following conclusion:
\begin{corollary}\label{cor:LSnumber}
$
\operatorname{cat}(E)=3, LSN(E)=4.$
\end{corollary}

\vspace{1mm}
 
\section*{Acknowledgements}
The authors wish to thank Jacob Bernstein for sharing the result in \cite{BK26} before publication and for related correspondence. 

This project has received funding from the European Research Council (ERC) under the European Union’s Horizon 2020 research and innovation programme (grant agreement No. 947923). The second author was partially supported by the AMS-Simons travel grant 2023-2025 and Simons Foundation Travel Support for Mathematicians 2026-2031.

\vspace{1mm}
\stoptoc

\section*{AI Declaration Statement}
All mathematical ideas, arguments, and results presented in this manuscript are solely due to the authors, and have been obtained without the use of generative AI. Chat GPT was solely exploited in the revision process, towards notational uniformization and streamlined exposition.

\resumetoc

\vspace{2mm}

\bibliography{biblio}

    \end{document}